\documentclass[11pt,reqno,tbtags,draft]{amsart}
\usepackage{amssymb}
\usepackage{url}
\usepackage[square,numbers]{natbib}
\usepackage{bbold}
\usepackage{bbm}
\bibpunct[, ]{[}{]}{;}{n}{,}{,}
\title[QuickQuant and QuickVal density]
{Density functions for {\tt QuickQuant} and {\tt QuickVal}}

\newcommand\urladdrx[1]{{\urladdr{\def~{{\tiny$\sim$}}#1}}}
\author{James Allen Fill}
\address{Department of Applied Mathematics and Statistics,
The Johns Hopkins University,
3400 N.~Charles Street,
Baltimore, MD 21218-2682 USA}
\email{jimfill@jhu.edu}
\urladdrx{http://www.ams.jhu.edu/~fill/}
\thanks{Research of both authors supported by 
the Acheson~J.~Duncan Fund for the Advancement of Research in
Statistics.}

\author{Wei-Chun Hung}
\address{Department of Applied Mathematics and Statistics,
The Johns Hopkins University,
3400 N.~Charles Street,
Baltimore, MD 21218-2682 USA}
\email{whung6@jhu.edu}

\makeatletter
\@namedef{subjclassname@2020}{\textup{2020} Mathematics Subject Classification}
\makeatother

\keywords{{\tt QuickQuant}, {\tt QuickSelect}, {\tt QuickVal}, searching, 
convolutions of distributions, densities, integral equations, asymptotic bounds, tails of distributions, tails of densities, large deviations, moment generating functions}
\numberwithin{equation}{section}

\allowdisplaybreaks

\theoremstyle{plain}% default
\newtheorem{theorem}{Theorem}[section]
\newtheorem{lemma}[theorem]{Lemma}
\newtheorem{proposition}[theorem]{Proposition}
\newtheorem{corollary}[theorem]{Corollary}

\theoremstyle{definition}

\newtheorem{remark}[theorem]{Remark}

\newtheorem*{ack}{Acknowledgement}

\theoremstyle{remark}

\newenvironment{romenumerate}[1][-10pt]{% optional argument changes indentation
\addtolength{\leftmargini}{#1}\begin{enumerate}% gives (i), (ii) etc.
 }{\end{enumerate}}

\newcounter{oldenumi}
{\setcounter{oldenumi}{\value{enumi}}
\begin{romenumerate} \setcounter{enumi}{\value{oldenumi}}}
{\end{romenumerate}}

\newcounter{thmenumerate}

\newcounter{xenumerate}   %no left indentation; thus wider lines

\newcommand{\refT}[1]{Theorem~\ref{#1}}
\newcommand{\refC}[1]{Corollary~\ref{#1}}
\newcommand{\refL}[1]{Lemma~\ref{#1}}
\newcommand{\refR}[1]{Remark~\ref{#1}}
\newcommand{\refS}[1]{Section~\ref{#1}}

\newcommand{\refP}[1]{Proposition~\ref{#1}}

\newcommand{\ignore}[1]{}

\newcommand\marginal[1]{\marginpar{\raggedright\parindent=0pt\tiny #1}}
\newcommand\REM[1]{{\raggedright\texttt{[#1]}\par\marginal{XXX}}}

\begingroup
  \divide\count255 by 60
  \multiply\count255 by -60
  \advance\count255 by \time
  \ifnum \count255 < 10 \xdef\klockan{\the\count1.0\the\count255}
  \else\xdef\klockan{\the\count1.\the\count255}\fi
\endgroup

\def\rompar(#1){\textup(#1\textup)}    % usage: \rompar(...)

\def\xexp(#1){e^{#1}}

\newcommand\floor[1]{\lfloor#1\rfloor}

\newcommand\punkt{.\spacefactor=1000}    % om problem!
    
\newcommand\ie{i.e\punkt}

\newcommand\cf{cf\punkt}

\newcommand\bbR{\mathbb R}

\newcounter{CC}
\newcounter{cc}
\newcommand\E{\operatorname{\mathbb E{}}}
\renewcommand\P{\operatorname{\mathbb P{}}}

\newcommand\cF{\mathcal F}

\newcommand\cT{{\mathcal T}}

\newcommand\dd{\,\mathrm{d}}
\newcommand\ddx{\mathrm{d}}

\newcommand{\tc}{\tilde{c}}

\newcommand{\Polya}{P\'olya}

\newcommand{\ro}[1]{\uppercase\expandafter{\romannumeral #1}}

\begin{document}

\maketitle

\vspace{-.3in}
\begin{center}
\begin{small}
September~29, 2021
\end{small}
\end{center}

\begin{abstract}
We prove that, for every $0 \leq t \leq 1$, the limiting distribution of the scale-normalized number of key comparisons used by the celebrated algorithm {\tt QuickQuant} to find the $t$th quantile in a randomly ordered list has a Lipschitz continuous density function $f_t$ that is bounded above by~$10$. Furthermore, this density $f_t(x)$ is positive for every $x > \min\{t, 1 - t\}$ and, uniformly in~$t$, enjoys superexponential decay in the right tail.  We also prove that the survival function $1 - F_t(x) = \int_x^{\infty}\!f_t(y) \dd y$ and the density function $f_t(x)$ both 
have the right tail asymptotics 
$\exp [-x \ln x - x \ln \ln x + O(x)]$.  We use the right-tail asymptotics to bound large deviations for the scale-normalized number of key comparisons used by {\tt QuickQuant}.
\end{abstract}

\section{Introduction and summary}
\label{S:intro}

\subsection{Introduction}

{\tt QuickQuant} is closely related to an algorithm called {\tt QuickSelect}, which in turn can be viewed as a one-sided analogue of {\tt QuickSort}.  In brief, {\tt QuickSelect}$(n, m)$ is an algorithm designed to find a number of rank 
$m$ in an unsorted list of size~$n$.  It works by recursively applying the same partitioning step as {\tt QuickSort} to the sublist that contains the item of rank~$m$ until the pivot we pick \emph{has} the desired rank or the size of the sublist to be explored has size one. Let $C_{n, m}$ denote the number of comparisons needed by {\tt QuickSelect}$(n, m)$, and note that $C_{n, m}$ and $C_{n, n + 1 - m}$ have the same distribution, by symmetry. Knuth~\cite{MR0403310} finds the formula
\begin{equation}
\label{E:exp_select}
\E C_{n,m} = 2 \left[ (n+1)H_n - (n+3-m) H_{n+1-m} - (m+2) H_m + (n+3) \right]
\end{equation}
for the expectation.

The algorithm {\tt QuickQuant}$(n, t)$ refers to {\tt QuickSelect}$(n, m_n)$ such that the ratio $m_n / n$ converges to a specified value $t \in [0, 1]$ as $n \to \infty$.  It is easy to see that \eqref{E:exp_select} tells us about the limiting behavior of the expected number of comparisons after standardizing:
\begin{equation}
\label{EZ}
\lim_{n \to \infty} \E[n^{-1} C_{n,m_n}] = 2 + 2 H(t),
\end{equation}
where $H(x) := - x \ln x - (1-x) \ln(1 - x)$ with $0 \ln 0 := 0$. 

We 
follow the set-up and notation of Fill and Nakama~\cite{MR3102458}, who use an infinite sequence $(U_i)_{i \geq 1}$ of independent Uniform$(0,1)$-distributed random variables to couple the number of key comparisons $C_{n,m_n}$ for all $n$. Let $L_0(n) := 0$ and $R_0(n) := 1$. For $k \geq 
1$, inductively define
\[
\tau_k(n) := \inf \{i \leq n:L_{k-1}(n) < U_i < R_{k-1}(n)\},
\]
and let $r_k(n)$ be the rank of the pivot $U_{\tau_k(n)}$ in the set $\{U_1, \dots, U_n\}$ if $\tau_k(n) < \infty$ and be $m_n$ otherwise. 
[Recall that the infimum of the empty set is~$\infty$; hence $\tau_k(n) = 
\infty$ if and only if $L_{k - 1}(n) =R_{k - 1}(n)$.]
Also, inductively define
\begin{align}
\label{left}
L_k(n) :=& \mathbb{1}(r_k(n) \leq m_n) U_{\tau_k(n)} + \mathbb{1}(r_k(n) > m_n) L_{k-1}(n),\\
\label{right}
R_k(n) :=& \mathbb{1}(r_k(n) \geq m_n) U_{\tau_k(n)} + \mathbb{1}(r_k(n) < m_n) R_{k-1}(n),
\end{align}
if $\tau_k(n) < \infty$, but
\[
(L_k(n), R_k(n)) := (L_{k-1}(n), R_{k-1}(n)).
\]
if $\tau_k(n) = \infty$. The number of comparisons at the $k^{th}$ step 
is then
\[
S_{n,k} := \sum_{i: \tau_k < i \leq n} \mathbb{1}(L_{k-1}(n) < U_i < R_{k-1}(n)),
\]
and the normalized total number of comparisons equals 
\begin{equation}
\label{E:1}
n^{-1} C_{n,m_n} := n^{-1} \sum_{k \geq 1} S_{n,k}.
\end{equation}

Mahmoud, Modarres and Smythe~\cite{MR1359052} studied {\tt QuickSelect} in the case that the rank~$m$ is taken to be a random variable $M_n$ uniformly distributed on $\{1, \dots, n\}$ and assumed to be independent of the numbers in the list.  They used the Wasserstein metric to prove that $Z_n := n^{-1} C_{n,M_n} \xrightarrow{\mathcal{L}} Y$ as $n \to \infty$ and identified the distribution of $Y$. 
In particular, they proved that $Y$ has an absolutely continuous distribution function. Gr\"{u}bel and R\"{o}sler~\cite{MR1372338} treated all the 
quantiles $t$ simultaneously by letting $m_n \equiv m_n(t)$. Specifically, they considered the normalized process $X_n$ defined by
\begin{equation}
\label{E:2}
X_n(t) := n^{-1} C_{n,\floor{n t} +1} \mbox{ for } 0 \leq t < 1, \quad X_n(t) := n^{-1} C_{n, n} \mbox{\ for\ } t = 1.
\end{equation} 
Working in the Skorohod topology (see Billingsley~\cite[Chapter~3]{MR1700749}), they proved that this process has a limiting distribution as $n \to \infty$, and the value of the limiting process at argument $t$ 
is the sum of the lengths of all the intervals encountered in all the steps of searching for population quantile~$t$. We can use the \emph{same} sequence $(U_i)_{i \geq 1}$ of Uniform$(0, 1)$ random variables to express 
the limiting stochastic process. For $t \in [0, 1]$, 
let $L_0(t) := 0$ and $R_0(t) := 1$, and
let $\tau_0(t) := 0$. 
For $t \in [0, 1]$ and $k \geq 1$, inductively define
\begin{align}
\label{taukt}
\tau_k(t) :=& \inf \{i > \tau_{k - 1}(t) : L_{k-1}(t) \leq U_i \leq R_{k-1}(t) \},\\
\label{Lkt}
L_k(t) :=& \mathbb{1} {(U_{\tau_k(t)} \leq t)}\, U_{\tau_k(t)} + \mathbb{1} {(U_{\tau_k(t)} > t)}\, L_{k-1}(t),\\
\label{Rkt}
R_k(t) :=& \mathbb{1} {(U_{\tau_k(t)} \leq t)}\, R_{k-1}(t) + \mathbb{1} {(U_{\tau_k(t)} > t)}\, U_{\tau_k(t)}.
\end{align}
It is not difficult to see that
\[
\P(\tau_k(t) < \infty\mbox{\ and\ }0 \leq L_k(t) \leq t \leq R_k(t) \leq 1 \mbox{\ for all $0 \leq t \leq 1$ and $k \geq 0$}) = 1
\]
and that for each fixed $t \in (0, 1)$ we have
\[
\P(L_k(t) < t < R_k(t)\mbox{\ for all $k \geq 0$}) = 1.
\]
The limiting process~$Z$ can be expressed as
\begin{equation}
\label{E:3}
Z(t) := \sum_{k=0}^{\infty}{\left[R_k(t) - L_k(t)\right]} = 1 + \sum_{k=1}^{\infty}{\left[R_k(t) - L_k(t)\right]};
\end{equation}
it is not hard to see that
\[
\P(1 < Z(t) < \infty\mbox{\ for all $0 \leq t \leq 1$}) = 1.
\]
Note also that the processes~$Z$ and $(Z(1 - t))_{t \in [0, 1]}$ have the same finite-dimensional distributions.
Gr\"{u}bel and R\"{o}sler~\cite[Theorem 8]{MR1372338} proved that 
we can replace the subscript $\floor{n t} +1$ in \eqref{E:2} by any $m_n(t)$ with $0 \leq m_n(t) \leq n$ such that $m_n(t) / n \to t$ as $n \to \infty$, and then the normalized random variables 
$n^{-1} C_{n, m_n(t)}$ converge (univariately) to the limiting random variable $Z(t)$ for each $t \in [0, 1]$.

Among the univariate distributions of $Z(t)$ for $t \in [0, 1]$, only the common distribution of $Z(0)$ and $Z(1)$ is known at all explicitly.  As established by Hwang and Tsai~\cite{MR1918722}, this distribution is the Dickman distribution; see their paper for a wealth of information about the distribution.  The overarching goal of this paper is to establish basic properties of the distributions of $Z(t)$ for other values of~$t$.

Kodaj and M\'ori~\cite{MR1454110} proved the (univariate) convergence of \eqref{E:1} to $Z(t)$
in the Wasserstein metric. Using the coupling technique and induction, they proved that \eqref{E:1} is stochastically smaller than its continuous counterpart \eqref{E:3}. Combining this fact with knowledge of their expectations (see \eqref{E:exp_select} and~\cite[Lemma 2.2]{MR1454110}), they 
proved that \eqref{E:1} converges to \eqref{E:3} in the Wasserstein metric and thus in distribution.

Gr\"{u}bel~\cite{MR1622443} 
connected {\tt QuickSelect}$(n,m_n)$ to a Markov chain to identify the limiting process. For each fixed $n \geq 1$, he considered the Markov chain 
$(Y_m^{(n)})_{m \geq 0}$ on the state space 
$I_n := \left\{(i, j):1 \leq j \leq i \leq n \right\}$ 
with $Y_0^{(n)} := (n, m_n)$. 
Transition probabilities of $Y^{(n)}$ from the state $(i, j)$ are determined by the partition step of 
{\tt QuickSelect}$(i, j)$ as follows.  If $Y_m^{(n)} = (i, j)$, then $Y_{m+1}^{(n)}$ is selected uniformly at random from the set
\[
\left\{ (i - k, j - k):k = 1, \dots, j - 1 \right\} \cup \left\{ (1, 1) 
\right\} 
\cup \left\{ (i-k, j):k = 1, \dots, i-j \right\};
\]
in particular, $(1, 1)$ is an absorbing state for $Y^{(n)}$.
If we write $Y_m^{(n)} = (S_m^{(n)}, Q_m^{(n)})$, then we know
\begin{equation}
\label{E:4}
n^{-1} C_{n,m_n} \stackrel{\mathcal{L}}{=} n^{-1} \sum_{m \geq 0} \left(S_m^{(n)} - 1\right). 
\end{equation}
Gr\"{u}bel~\cite{MR1622443} 
constructed another Markov chain $Y = (Y_m) = ((S_m, Q_m))$, which is 
a continuous-value counterpart of the process $Y^{(n)}$, and he proved that for all $m \geq 0$, the random vector $Y^{(n)}_m$ converges to $Y_m$ almost surely. Using the dominated convergence theorem, he proved that the 
random variables $n^{-1} \sum_{m=0}^{\infty}{\left(S_m^{(n)}-1\right)}$ 
converge almost surely to 
$\sum_{m=0}^{\infty}{S_m}$; the limiting random variable here is exactly $Z(t)$ of~\eqref{E:3}. Combining with \eqref{E:4}, he concluded that $n^{-1} C_{n,m_n}$ converges in distribution to \eqref{E:3}.  As previously mentioned, Hwang and Tsai~\cite{MR1918722} identified the limiting distribution of \eqref{E:1} when $m_n = o(n)$ as the Dickman distribution.

Fill and Nakama~\cite{MR3102458} 
studied the limiting distribution of the cost of using {\tt QuickSelect} for a variety of cost functions. In particular, when there is simply \emph{unit} cost of comparing any two keys, then their work reduces to study of the number of key comparisons, to which we limit our focus here. They proved 
$L^p$-convergence of \eqref{E:1} for {\tt QuickQuant}$(n, t)$ to \eqref{E:3} for $1 \leq p < \infty$ by first studying the 
distribution of the number of key comparisons needed for another algorithm called {\tt QuickVal}, and then comparing the two algorithms. The algorithm {\tt QuickVal}$(n, t)$ finds the rank of the \emph{population} $t$-quantile in the sample, while its cousin {\tt QuickQuant}$(n, t)$ looks for the sample $t$-quantile. Intuitively, when the sample size is large, we 
expect the rank of the population $t$-quantile to be close to $n t$. Therefore, the two algorithms should behave similarly when $n$ is large. 
Given a set of keys $\{U_1, \dots, U_n\}$, where $U_i$ are i.i.d. Uniform$(0,1)$ random variables, one can regard the operation of {\tt QuickVal}$(n, t)$ as that of finding the rank of the value $t$ in the augmented set 
$\{U_1, ..., U_n, t\}$. 
It works by first selecting a pivot uniformly at random from the set of keys 
$\{U_1, ..., U_n\}$ and then using the pivot to partition the \emph{augmented} set (we don't count the comparison of the pivot with~$t$). We then recursively do the same partitioning step on the subset that contains $t$ 
until the set of the keys on which the algorithm operates reduces to the singleton $\{t\}$.
For {\tt QuickVal}$(n, t)$ with the definitions \eqref{taukt}--\eqref{Rkt} and with
\[
S_{n, k}(t) := \sum_{\tau_k(t) < i \leq n} \mathbbm{1}(L_{k - 1}(t) < U_i < R_{k - 1}(t)),
\]
Fill and Nakama~\cite{MR3102458} showed that $n^{-1} \sum_{k \geq 1} S_{n, k}(t)$ converges (for fixed~$t$) almost surely and also in $L^p$ for $1 
\leq p < \infty$ to \eqref{E:3}. They then used these facts to prove the $L^p$ convergence (for fixed~$t$) of \eqref{E:1} to \eqref{E:3} for 
{\tt QuickQuant}$(n, t)$ for $1 \leq p < \infty$.

Fill and Matterer~\cite{MR3249225} treated distributional convergence for 
the worst-case cost of {\tt Find} for a variety of cost functions.  Suppose, for example, that we continue, as at the start of this section, to assign unit cost to the comparison of any two keys, so that $C_{n, m}$ is the total cost for 
${\tt QuickSelect}(n, m)$.  Then (for a list of length~$n$) the cost of worst-case {\tt Find} is $\max_{1 \leq m \leq n} C_{n, m}$, and its distribution depends on the joint distribution of $C_{n, m}$ for varying~$m$.  We shall not be concerned here with worst-case {\tt Find}, but we wish to 
review the approach and some of the results of~\cite{MR3249225}, since there is relevance of their work to ${\tt QuickQuant}(n, t)$ for fixed~$t$.

Fill and Matterer~\cite{MR3249225} considered tree-indexed processes closely related to the operation of the {\tt QuickSelect} algorithm, as we now describe. For each node in a given rooted ordered binary tree, let $\theta$ denote the binary sequence (or \emph{string}) representing the path from the root to this node, where $0$ corresponds to taking the left child and $1$ to taking the right.  The value 
of~$\theta$ for the root is thus the empty string, denoted $\varepsilon$. 
 Define $L_{\varepsilon} := 0$, $R_{\varepsilon} := 1$, and 
$\tau_{\varepsilon} := 1$.  Given a sequence of i.i.d.\ Uniform$(0,1)$ random variables $U_1, U_2, \dots$, recursively define
\begin{align*}
\tau_{\theta} &:= \inf \{ i: L_{\theta} < U_i < R_{\theta}\},\\
L_{\theta 0} &:= L_{\theta}, \quad L_{\theta 1} := U_{\tau_{\theta}},\\
R_{\theta 0} &:= U_{\tau_{\theta}}, \quad R_{\theta 1} := R_{\theta}.
\end{align*}
Here the concatenated string $\theta 0$ corresponds to the left child of the node with string~$\theta$, while $\theta 1$ corresponds to the right child.  Observe that, when inserting a key $U_i$ arriving at time $i > \tau_{\theta}$ into the binary tree, this key is compared with the ``pivot'' $U_{\tau_{\theta}}$ if and only if $U_i \in (L_{\theta}, R_{\theta})$.  
For~$n$ insertions, the total cost of comparing keys with pivot $U_{\tau_{\theta}}$ is therefore
\[
S_{n, \theta} := \sum_{\tau_{\theta} < i \leq n} \mathbbm{1}(L_{\theta} 
< U_i < R_{\theta}).
\]
We define a binary-tree-indexed stochastic process $S_n = (S_{n, \theta})_{\theta \in \Theta}$, where $\Theta$ is the collection of all finite-length binary sequences.

For each $1 \leq p \leq \infty$, Fill and Matterer~\cite[Definition~3.10 and Proposition~3.11]{MR3249225} defined a Banach space 
$\mathcal{B}^{(p)}$ of binary-tree-indexed stochastic processes that corresponds in a natural way to the Banach space $L^p$ for random variables.  
Let $I_{\theta} := R_{\theta} - L_{\theta}$ and consider the process $I 
= (I_{\theta})_{\theta \in \Theta}$. 
Fill and Matterer~\cite[Theorem~4.1 with $\beta \equiv 1$]{MR3249225} proved the convergence of the processes $n^{-1} S_n$ to $I$ in the Banach space $\mathcal{B}^{(p)}$ for each $2 \leq p < \infty$.  

For the simplest application in~\cite{MR3249225}, 
namely, to {\tt QuickVal}$(n, t)$ with~$t$ fixed, let 
$\gamma(t)$ be the infinite path from the root to the key having value~$t$
in the (almost surely) complete binary search tree formed by successive insertions of $U_1, U_2, \dots$ into an initially empty tree. 
The total cost (call it $V_n$) of {\tt QuickVal}$(n,t)$ can then be computed by summing the cost of comparisons with each (pivot-)node along the path, that is,
\[
V_n := \sum_{\theta \in \gamma(t)}{S_{n,\theta}}.
\]
Using their tree-process convergence theorem described in our preceding paragraph, Fill and Matterer~\cite[Proposition~6.1 with $\beta \equiv 1$]{MR3249225} established $L^p$-convergence, for 
each $0 < p < \infty$, of $n^{-1} V_n$ to $I_{\gamma(t)}$ as $n \to \infty$, where $I_{\gamma(t)} := \sum_{\theta \in \gamma(t)} I_{\theta}$.  Moreover (\cite[Theorem~6.3 with $\beta \equiv 1$]{MR3249225}), they also proved $L^p$-convergence of $n^{-1} Q_n$ to the same limit, again for every $0 < p < \infty$, where $Q_n$ denotes the cost of 
${\tt QuickQuant}(n, t)$.

Throughout
this paper, we will use the standard notations $\mathbb{1}(A)$ to denote the indicator function of the event $A$ and 
 $\E [f;\,A]$ for $\E[f \mathbb{1}(A)]$.
 
 \subsection{Summary}
 
 In \refS{S:density}, by construction we establish the existence of densities $f_t$ for the random variables $Z(t)$ defined in~\eqref{E:3}.  In \refS{S:boundedness} we prove that these densities are uniformly bounded and in 
 \refS{S:uniform continuity} that they are uniformly continuous.  As shown in \refS{S:Integral equation}, the densities satisfy a certain integral 
equation for $0 < t < 1$.  The right-tail behavior of the density functions is examined in \refS{S:decay}, and the left-tail behavior in \refS{S:left}.  In \refS{S:other} we prove that $f_t(x)$ is positive if and only if $x > \min\{t, 1 - t\}$, and we improve the result of \refS{S:uniform continuity} by showing that $f_t(x)$ is Lipschitz continuous in~$x$ for fixed~$t$ and jointly continuous in $(t, x)$.  Sections~\ref{S:improved}--\ref{S:lower} are devoted to sharp logarithmic asymptotics for the right tail of $f_t$, and \refS{S:Large_deviation} uses the results of those two sections to treat right-tail large deviation behavior of {\tt QuickQuant}$(n, t)$ for large but finite~$n$.
 
\section{Existence (and construction) of density functions}
\label{S:density}
In this section, we prove that $Z \equiv Z(t)$ defined in \eqref{E:3} for 
fixed $0 \leq t \leq 1$ has a density. 
For notational simplification, we let $L_k \equiv L_k(t)$ and $R_k \equiv 
R_k(t)$. Let $J \equiv J(t) := Z(t)-1 =  \sum_{k=1}^{\infty}{\Delta_k}$ with $\Delta_k \equiv \Delta_k(t) := R_k(t)-L_k(t)$. We use convolution notation as in Section V.4 of Feller~\cite{MR0270403}.  The following lemma is well known and can be found, for example, in Feller\cite[Theorem V.4.4]{MR0270403} or Durrett\cite[Theorem 2.1.11]{MR2722836}. 

\begin{lemma}
\label{L:Feller}
If $X$ and $Y$ are independent random variables with respective distribution functions~$F$ and~$G$, then $Z = X + Y$ has the distribution function $F \star G$. If, in addition, $X$ has a density~$f$ (with respect to Lebesgue measure), then $Z$ has a density $f \star G$.
\end{lemma}

Let $X = \Delta_1 + \Delta_2$ and $Y = \sum_{k=3}^{\infty}{\Delta_k}$. For the remainder of this paragraph, we suppose $0 < t < 1$.  If we condition on $(L_3, R_3) = (l_3, r_3)$ for some $0 \leq l_3 < t < r_3 \leq 1$ with $(l_3, r_3) \ne (0,1) $, we then have
\begin{equation}
\label{E:Y}
Y = (r_3-l_3) \sum_{k=3}^{\infty}{\frac{R_k-L_k}{r_3-l_3}} = (r_3-l_3) \sum_{k=3}^{\infty}{(R_k'-L_k')},
\end{equation}
where we set $L_k' = (L_k - l_3) / (r_3 - l_3)$ and $R_k' = (R_k - l_3)/(r_3 - l_3)$ for $k \geq 3$. 
Observe that, by definitions 
\eqref{taukt}--\eqref{Rkt},
the stochastic process $((L_k', R_k'))_{k \geq 3}$, conditionally given $(L_3, R_3) = (l_3, r_3)$, has the same distribution as the (unconditional) stochastic process of intervals $((L_k, R_k))_{k \geq 0}$ encountered 
in all the steps of searching for population quantile 
$(t - l_3)/(r_3-l_3)$ (rather than~$t$) by {\tt QuickQuant}.  Note also that (again conditionally) the stochastic processes $((L_k, R_k))_{0 \leq k \leq 2}$ and $((L_k', R_k'))_{k \geq 3}$ are independent.
Thus (again conditionally) $Y / (r_3 - l_3)$ has the same distribution as 
the (unconditional) random variable $Z\left((t - l_3) / (r_3 - l_3)\right)$ 
and is independent of~$X$. 
We will prove later (Lemmas~\ref{L:main}--\ref{L:main01}) that, conditionally given $(L_3, R_3) = (l_3, r_3)$, the random variable $X$ has a 
density.  Let
\[
f_{l_3, r_3}(x) := \P(X \in \ddx x \mid (L_3, R_3) = (l_3, r_3)) / \ddx x
\]
be such a conditional density. We can then use \refL{L:Feller} to conclude that $J = X + Y$ has a conditional density
\[
h_{l_3, r_3}(x) := \P(J \in \ddx x \mid (L_3, R_3) = (l_3, r_3)) / \ddx x.
\]
By mixing $h_{l_3, r_3}(x)$ for all possible values of $l_3, r_3$, we will obtain an unconditional density function for~$J$, as summarized in the following theorem.
\begin{theorem}
\label{T:main}
For each $0 \leq t \leq 1$, the random variable $J(t) = Z(t) - 1$ has a 
density
\begin{equation}
\label{E:density}
f_t(x) := \int{\P((L_3, R_3) \in \ddx (l_3, r_3))} \cdot h_{l_3, r_3}(x),
\end{equation}
and hence the random variable $Z(t)$ has density $f_t(x-1)$.
\end{theorem}

Now, as promised, we prove that, conditionally given $(L_3, R_3) = (l_3, r_3)$, the random variable $X$ has a density $f_{l_3, r_3}$.
We begin with the case $0 < t < 1$.
\begin{lemma}
\label{L:main}
Let $0 \leq l_3 < t < r_3 \leq 1$ with $(l_3, r_3) \ne (0, 1)$. Conditionally given $(L_3, R_3) = (l_3, r_3)$, the random variable $X = \Delta_1 + \Delta_2$ has a right continuous density $f_{l_3, r_3}$.
\end{lemma}
\begin{proof}
We consider three cases based on the values of $(l_3, r_3)$.
\smallskip

\par\noindent
\emph{Case 1}: $l_3 = 0$ and $r_3 < 1$.  Since $L_k$ is nondecreasing in~$k$, from $L_3 = 0$ it follows that $L_1 = L_2 = 0$.
The unconditional joint distribution of $(L_1, R_1, L_2, R_2, L_3, R_3)$ satisfies
\begin{align}
&\P(L_1 = 0, R_1 \in \ddx r_1, L_2 = 0, R_2 \in \ddx r_2, L_3 = 0, R_3 \in \ddx r_3) \nonumber \\
&= \mathbb{1}{(t < r_3 < r_2 < r_1 < 1)} \dd r_1 \frac{\ddx r_2}{r_1} \frac{\ddx r_3}{r_2} \label{joint1}
\end{align}
and hence
\begin{align}
\P(L_3 = 0, R_3 \in \ddx r_3) &= \mathbb{1}{(t < r_3 <1)} \dd r_3\int_{r_2 = r_3}^{1}{\frac{\ddx r_2}{r_2}}\, \int_{r_1=r_2}^{1}{\frac{\ddx r_1}{r_1}} \nonumber\\
&= \mathbb{1}{(t < r_3 <1)} \dd r_3\int_{r_2 = r_3}^{1}{\frac{\ddx r_2}{r_2}}(-\ln r_2) \nonumber\\
&= \frac{1}{2}\, (\ln r_3)^2\, \mathbb{1}{(t < r_3 <1)} \dd r_3. \label{joint2}
\end{align} 
Dividing
\eqref{joint1} by \eqref{joint2}, 
we 
find
\begin{align}
\lefteqn{\hspace{-0.5in}\P(L_1 = 0, R_1 \in \ddx r_1, L_2 = 0, R_2 \in \ddx r_2 \mid L_3 = 0, R_3= r_3)} \nonumber \\
&=\frac{2\, r_1^{-1}r_2^{-1}\dd r_1\dd r_2}{(\ln r_3)^2} \, \mathbb{1}{(t < r_3 < r_2 < r_1 < 1)}.\label{cond1}
\end{align}
Thus for $x \in (2 r_3, 2)$, we 
find
\begin{align}
f_{0,r_3}(x) &= \P(X \in \ddx x \mid L_3 = 0, R_3 = r_3) / \ddx x \nonumber\\
&= \int_{r_2}\!\P(R_2 \in \ddx r_2,\,R_1 \in \ddx x - r_2 \mid L_3 = 0, R_3 = r_3) / \ddx x \nonumber\\
&= \frac{2}{(\ln r_3)^2}\int_{r_2=r_3 \vee (x-1)}^{x / 2}{(x-r_2)^{-1}r_2^{-1}} \dd r_2 \nonumber\\
\label{CD1}
&= \frac{2}{\left(\ln \frac{1}{r_3}\right)^2} \frac{1}{x} \bigg[ \ln \left(\frac{x-r_3}{r_3}\right) \mathbb{1}{(2r_3 \leq x < 1+r_3)}\nonumber\\
&{} \qquad \qquad \qquad \quad +\ln \left(\frac{1}{x-1}\right) \mathbb{1}{(1+r_3 \leq x < 2)}\bigg];
\end{align}
we set $f_{0, r_3}(x) = 0$ for $x \notin (2 r_3, 2)$.
\smallskip

\par\noindent
\emph{Case 2}: $l_3 > 0$ and $r_3 = 1$.
This condition implies that $R_1=R_2=1$. Invoking symmetry, we can skip the derivation and immediately 
write
\begin{align}
f_{l_3,1}(x) = \frac{2}{\left(\ln \frac{1}{1-l_3}\right)^2} \frac{1}{x}&\biggl[ \ln \left(\frac{x-1+l_3}{1-l_3}\right)\mathbb{1}{(2-2l_3 \leq x < 2-l_3)}\nonumber \\
\label{CD2}
&{} \qquad + \ln \frac{1}{x-1} \mathbb{1}{(2-l_3 \leq x < 2)}\biggr]
\end{align}
for $x \in (2 - 2 l_3, 2)$; we set $f_{l_3, 1}(x) = 0$ for $x \notin (2 
- 2 l_3, 2)$.
\smallskip

\par\noindent
\emph{Case 3}: $0 < l_3 < t < r_3 < 1$.
There are six possible scenarios for the random vector $(L_1, R_1, L_2, R_2, L_3, R_3)$, and to help us discuss the cases, we consider values $l_2, r_2$ satisfying $0 < l_2 < l_3 < t < r_3 < r_2 < 1$.\\
\\
(a) $L_1 = l_2, L_2 = L_3 = l_3$ and $R_1 = R_2 =1, R_3 = r_3$.\\

In this subcase, we consider the event that the first pivot we choose locates between $0$ and $l_3$, the second pivot has value $l_3$, and the third pivot has value $r_3$. Denote this event by $E_{llr}$ (with $llr$ indicating that we shrink the search intervals by moving the lefthand, lefthand, and then righthand endpoints). 
We have
\begin{align}
\lefteqn{\hspace{-.5in}\P(L_1 \in \ddx l_2, R_1 = 1, L_2 \in \ddx l_3 , 
R_2 = 1, L_3 \in \ddx l_3, R_3 \in \ddx r_3)} \nonumber \\
\label{E:3a}
&= \mathbb{1}{(0 < l_2 <  l_3 < t < r_3 < 1)} \dd l_2 \frac{\ddx l_3}{1-l_2} \frac{\ddx r_3}{1-l_3}.
\end{align}
Integrating over all possible values of $l_2$, we 
get
\begin{align*}
\lefteqn{\hspace{-0.5in}\P(L_3 \in \ddx l_3, R_3 \in \ddx r_3, E_{llr})} \\
&= \mathbb{1}{(0<l_3<t<r_3<1)}\, \frac{1}{1-l_3} \ln \left(\frac{1}{1-l_3}\right) \dd l_3 \dd r_3.
\end{align*}
(b) $L_1 = L_2 = 0, L_3 = l_3$ and $R_1 = r_2, R_2 = R_3 = r_3$.\\

In this and all subsequence subcases, we use notation like that in subcase~(a).  In this subcase, we invoke symmetry in comparison with subcase~(a).  The results 
are
\begin{align}
\lefteqn{\hspace{-0.5in}\P(L_1 = 0, R_1 \in \ddx r_2, L_2 = 0, L_3 \in \ddx l_3 , R_2 = R_3 \in \ddx r_3)} \nonumber \\
\label{E:3b}
&= \mathbb{1}{(0 < l_3 < t < r_3 < r_2 < 1)} \dd r_2 \frac{\ddx r_3}{r_2} \frac{\ddx l_3}{r_3}
\end{align}
and
\[
\P(L_3 \in \ddx l_3, R_3 \in \ddx r_3, E_{rrl}) = \mathbb{1}{(0<l_3<t<r_3<1)}\, \frac{1}{r_3} \ln \left(\frac{1}{r_3}\right) \dd l_3 \dd r_3.
\]
(c) $L_1 = L_2 = l_2, L_3 = l_3$ and $R_1 = 1, R_2 = R_3 = r_3$.\\

In this subcase we 
have

\begin{align}
\lefteqn{\hspace{-0.5in}\P(R_1 = 1, L_1 = L_2 \in \ddx l_2, L_3 \in \ddx l_3 , R_2 = R_3 \in \ddx r_3)} \nonumber \\
\label{E:3c}
&= \mathbb{1}{(0 < l_2 < l_3 < t < r_3 < 1)} \dd l_2 \frac{\ddx r_3}{1-l_2} \frac{\ddx l_3}{r_3-l_2}.
\end{align}
Integrating over the possible values of $l_2$, we find
\begin{align*}
\lefteqn{\P(L_3 \in \ddx l_3, R_3 \in \ddx r_3, E_{lrl})/(\ddx l_3 \dd r_3)} \\
&= \mathbb{1}{(0<l_3<t<r_3<1)} \, \frac{1}{1-r_3}
\left[ \ln \left(\frac{1}{r_3-l_3}\right)- \ln \left(\frac{1}{r_3}\right) 
- \ln \left(\frac{1}{1-l_3}\right) \right].
\end{align*}
(d) $L_1 = 0, L_2 = L_3 = l_3$ and $R_1 = R_2 = r_2, R_3 = r_3$.\\

In this subcase, by symmetry with subcase~(c), we have
\begin{align}
\lefteqn{\hspace{-0.5in} \P(L_1 = 0, L_2 = L_3 \in \ddx l_3 , R_1 = 
R_2 \in \ddx r_2, R_3 \in \ddx r_3)} \nonumber\\
\label{E:3d}
&= \mathbb{1}{(0 < l_3 < t < r_3 < r_2 < 1)} \dd r_2 \frac{\ddx l_3}{r_2} \frac{\ddx r_3}{r_2-l_3}.
\end{align}
and
\begin{align*}
\lefteqn{\hspace{-0.1in}\P(L_3 \in \ddx l_3, R_3 \in \ddx r_3, E_{rlr})/(\ddx l_3 \dd r_3)} \\ 
&= \mathbb{1}{(0<l_3<t<r_3<1)}\, \frac{1}{l_3} \left[ \ln \left(\frac{1}{r_3-l_3}\right) - \ln \left(\frac{1}{1-l_3}\right) - \ln \left(\frac{1}{r_3}\right) \right].
\end{align*}
(e) $L_1 = 0, L_2 = l_2, L_3 = l_3$ and $R_1 = R_2 = R_3 = r_3$.\\

In this subcase we 
have
\begin{align}
\lefteqn{\hspace{-0.5in}\P(L_1 = 0, L_2 \in \dd l_2 , L_3 \in \ddx l_3, 
R_1 = R_2 = R_3 \in \ddx r_3)} \nonumber\\
\label{E:3e}
&= \mathbb{1}{(0 < l_2 < l_3 < t < r_3 < 1)} \dd r_3 \frac{\ddx l_2}{r_3} \frac{\ddx l_3}{r_3-l_2}.
\end{align}
Integrating over the possible values of $l_2$, we 
have
\begin{align*}
\lefteqn{\hspace{-0.2in} \P(L_3 \in \ddx l_3, R_3 \in \ddx r_3, E_{rll})} 
\\ 
&= \mathbb{1}{(0<l_3<t<r_3<1)}\, \frac{1}{r_3} \left[ \ln \left(\frac{1}{r_3-l_3}\right) - \ln \left(\frac{1}{r_3}\right) \right] \dd l_3 \dd r_3.
\end{align*}
(f) $L_1 = L_2 = L_3 = l_3$ and $R_1 = 1, R_2 = r_2, R_3 = r_3$.\\

In this final subcase, by symmetry with subcase~(e), we 
have
\begin{align}
\lefteqn{\hspace{-0.5in}\P(L_1 = L_2 = L_3 \in \ddx l_3 , R_1 = 1, R_2  \in \ddx r_2, R_3 \in \ddx r_3)} \nonumber\\
\label{E:3f}
&= \mathbb{1}{(0 < l_3 < t < r_3 < r_2 < 1)} \dd l_3 \frac{\ddx r_2}{1-l_3} \frac{\ddx r_3}{r_2-l_3}
\end{align}
and
\begin{align*}
\lefteqn{\hspace{-0.2in}\P(L_3 \in \ddx l_3, R_3 \in \ddx r_3, E_{lrr})/(\ddx l_3 \dd r_3)} \\ 
&= \mathbb{1}{(0<l_3<t<r_3<1)}\, \frac{1}{1-l_3} \left[ \ln \left(\frac{1}{r_3-l_3}\right) - \ln \left(\frac{1}{1-l_3}\right) \right].
\end{align*}

Summing results from the six subcases, we conclude in Case~3 that
\begin{equation}
\label{E:4}
\P(L_3 \in \ddx l_3, R_3 \in \ddx r_3) = \mathbb{1}{(0<l_3<t<r_3<1)} \, 
g(l_3, r_3) \dd l_3 \dd r_3,
\end{equation}
where
\begin{align}
g(l_3, r_3) 
&:= \left[\frac{1}{l_3(1-l_3)}+\frac{1}{r_3(1-r_3)}\right]\ln \left(\frac{1}{r_3-l_3}\right) \label{gdef} \\
&{} \qquad \qquad - \left(\frac{1}{l_3}+\frac{1}{1-r_3}\right)\left[\ln \left(\frac{1}{r_3}\right) + \ln \left(\frac{1}{1-l_3}\right)\right]. \nonumber
\end{align}

The conditional joint distribution of $(L_1, R_1, L_2, R_2)$ given $(L_3, 
R_3) = (l_3, r_3)$ can be derived by dividing 
each of \eqref{E:3a}--\eqref{E:3f} by~\eqref{E:4}, and we can then compute $f_{l_3,r_3}$ from these conditional distributions. 
Let us write
\begin{equation}
\label{E:5}
f_{l_3,r_3}(x) =  \mathbb{1}{(0 < l_3 < t < r_3 <1)}\, \frac{1}{g(l_3,r_3)} \sum_{i = 1}^{6}{f_{l_3,r_3}^{(i)}(x)},
\end{equation}
where $f_{l_3, r_3}^{(i)}(x)\dd l_3\dd r_3\dd x$ is the contribution to
\[
\P(L_3 \in \ddx l_3,\,R_3 \in \ddx r_3, \,X \in \ddx x)
\]
arising from the $i$th subcase of the six.

In subcase~(a) we know that $X = R_1 - L_1 + R_2 - L_2 = 2 - l_2 - l_3$. 
Changing variables from $l_2$ to~$x$, from~\eqref{E:3a} we find
\[
f_{l_3,r_3}^{(1)}(x) = \mathbb{1}{(2-2l_3 \leq x < 2-l_3)}\, \frac{1}{1-l_3} \, \frac{1}{x-1+l_3}.
\]

In subcase~(b) we know that $X = r_2 + r_3$.
Changing variables from $r_2$ to~$x$, from~\eqref{E:3b} we find
\[
f_{l_3,r_3}^{(2)}(x) = \mathbb{1}{(2r_3 \leq x < 1+ r_3)}\, \frac{1}{r_3} \, \frac{1}{x-r_3}.
\]

In subcase~(c), we know that $X = 1 - 2 l_2 + r_3$.  
Changing variables from $l_2$ to~$x$, from~\eqref{E:3c} we find
\[
f_{l_3,r_3}^{(3)}(x) = \mathbb{1}{(1+r_3-2l_3 \leq x < 1+r_3)}\, \frac{1}{x+1-r_3} \, \frac{2}{x+r_3-1}.
\]

In subcase~(d), we know that $X = 2 r_2 - l_3$. 
Changing variables from $r_2$ to~$x$, from~\eqref{E:3d} we find
\[
f_{l_3,r_3}^{(4)}(x) = \mathbb{1}{(2r_3-l_3 \leq x < 2-l_3)}\, \frac{2}{x+l_3} \, \frac{1}{x-l_3}.
\]

In subcase~(e), we know that $X = 2 r_3 - l_2$.  
Changing variables from $l_2$ to~$x$, from~\eqref{E:3e} we find
\[
f_{l_3,r_3}^{(5)}(x) = \mathbb{1}{(2r_3-l_3 \leq x < 2r_3)}\, \frac{1}{r_3} \, \frac{1}{x-r_3}.
\]

Finally, in subcase~(f), we know that $X = 1 + r_2 - l_3$. 
Changing variables from $r_2$ to~$x$, from~\eqref{E:3f} we find
\[
f_{l_3,r_3}^{(6)}(x) = \mathbb{1}{(r_3-2l_3+1 \leq x < 2-2l_3)}\, \frac{1}{1-l_3} \, \frac{1}{x-1+l_3}.
\]

The density functions $f_{0,r_3}$ and $f_{l_3,1}$ we have found in Cases~1 and~2 are continuous.
We have chosen to make the functions $f_{l_3, r_3}^{(i)}$ (for $i = 1, \dots, 6$) right continuous in Case~3.  Thus the density $f_{l_3, r_3}$ we have determined at~\eqref{E:5} in Case~3 is right continuous.  
\end{proof}

Our next lemma handles the cases $t = 0$ and $t = 1$ that were excluded from \refL{L:main}, and its proof is the same as for Cases~1 and~2 in the proof of \refL{L:main}.
\begin{lemma}
\label{L:main01}
{\rm (a)} Suppose $t = 0$.  Let $0 < r_3 < 1$.  Conditionally given $(L_3, R_3) = (0, r_3)$, the random variable $X = \Delta_1 + \Delta_2$ has the right continuous density $f_{0, r_3}$ specified in the sentence containing~\eqref{CD1}.

{\rm (b)} Suppose $t = 1$.  Let $0 < l_3 < 1$.  Conditionally given $(L_3, R_3) = (l_3, 1)$, the random variable $X = \Delta_1 + \Delta_2$ has the right continuous density $f_{l_3, 1}$ specified in the sentence containing~\eqref{CD2}. 
\end{lemma}

We need to check the trivariate measurability of the function $f_t(l_3, r_3, x) := f_{l_3, r_3} (x)$ before diving into the derivation of the density function of $J$. 
Given a topological space~$S$, let $\mathcal{B}(S)$ denote its Borel $\sigma$-field, that is, the $\sigma$-field generated by the open sets of~$S$.  Also, given $0 < t < 1$, let 
\[S_t := \{(l_3, r_3) \ne (0, 1):0 \leq l_3 < t < r_3 \leq 1\}.
\]
\begin{lemma}
\label{L:mble}
{\rm (a)}~For $0 < t < 1$, 
the conditional density function $f_t(l_3, r_3, x)$, formed to be a right 
continuous function of $x$, is measurable with respect to 
$\mathcal{B}(S_t \times \mathbb{R})$. 

{\rm (b)}~For $t = 0$, the conditional density function $f_0(r_3, x) := 
f_{0, r_3}(x)$, continuous in~$x$, is measurable 
$\mathcal{B}((0, 1) \times \mathbb{R})$. 

{\rm (c)}~For $t = 1$, the conditional density function $f_1(l_3, x) := 
f_{l_3, 1}(x)$, continuous in~$x$, is measurable 
$\mathcal{B}((0, 1) \times \mathbb{R})$.
\end{lemma}

We introduce (the special case of real-valued $f$ of) 
a lemma taken from Gowrisankaran \cite[Theorem 3]{MR291403} giving a sufficient condition for the measurability of certain functions~$f$ defined on product spaces. 
 The lemma will help us prove \refL{L:mble}.
\begin{lemma}[Gowrisankaran~\cite{MR291403}]
\label{L:Gowrisankaran}
Let $(\Omega, \cF)$ be a measurable space.  Let $f:\Omega \times \mathbb{R} \to \mathbb{R}$.  Suppose that the section mapping $f(\cdot, y)$ is $\cF$-measurable for each $y \in \mathbb{R}$ and that the section mapping $f(\omega, \cdot)$ is either right continuous for each $\omega \in \Omega$ 
or left continuous for 
each $\omega \in \Omega$.  Then~$f$ is measurable with respect to the product $\sigma$-field $\cF \otimes \mathcal{B}(\mathbb{R})$.  
\end{lemma}

\begin{proof}[Proof of \refL{L:mble}]
We prove~(b), then~(c), and finally~(a).
\smallskip

(b)~Recall the expression~\eqref{CD1} for $f_0(r_3, x)$ [for $0 < r_3 < 1$ and $x \in (2 r_3, 2)$].  We apply \refL{L:Gowrisankaran} with 
$(\Omega, \cF) = ((0, 1), \mathcal{B}((0, 1)))$.  The right continuity of $f_0(r_3, \cdot)$ has already been established in \refL{L:main01}(a).  
On the other hand, when we fix $x$ and treat $f(0, r_3, x)$ as a function 
of $r_3$, the conditional density function can be separated into the following cases:
\begin{itemize}
\item If $x \leq 0$ or $x \geq 2$, then $f_0(r_3, x) \equiv 0$.
\item If $0 < x < 2$, then from~\eqref{CD1} we see that $f_0(r_3, x)$ is piecewise continuous (with a finite number of measurable domain-intervals), and hence measurable, in $r_3$.
\end{itemize} 
Since the product $\sigma$-field $\mathcal{B}((0, 1)) \otimes \mathcal{B}(\mathbb{R})$ equals $\mathcal{B}((0, 1) \times \mathbb{R})$, the desired 
result follows.
\smallskip

(c)~Assertion~(c) can be proved by a similar argument or by invoking symmetry.
\smallskip

(a)~We 
apply \refL{L:Gowrisankaran} with $(\Omega, \cF) = (S_t, \mathcal{B}(S_t))$.  The right continuity of $f(l_3, r_3, \cdot)$ has already been established in \refL{L:main}, so it suffices to show for each $x \in \mathbb{R}$ that $f(l_3, r_3, x)$ is measurable in $(l_3, r_3) \in S_t$.  For this it is clearly sufficient to show that $f(0, r_3, x)$ is measurable in $r_3 \in (t, 1)$, that $f(l_3, 1, x)$ is measurable in $l_3 \in (0, t)$, and that $f(l_3, r_3, x)$ is measurable in $(l_3, r_3) \in (0, t) \times (t, 1)$.  All three of these assertions follow from the fact that piecewise continuous functions (with a finite number of measurable domain-pieces) 
are measurable; in particular, for the third assertion, note that the function~$g$ defined at~\eqref{gdef} is continuous in $(l_3, r_3) \in (0, t) 
\times (t, 1)$ and that each of the six expressions $f_{l_3, r_3}^{(i)}(x)$ appearing in~\eqref{E:5} is piecewise continuous (with a finite number 
of measurable domain-pieces) in these values of $(l_3, r_3)$ for each fixed~$x \in \mathbb{R}$.
\smallskip

This complete the proof.
\end{proof}

As explained at the outset of this section, a conditional density $\mbox{$h_{l_3, r_3}(\cdot)$}$ for $J(t)$ given $(L_3, R_3) = (l_3, r_3)$ 
can now be formed by convolving the conditional density function of $X$, namely $f_{l_3,r_3}(\cdot)$, with the conditional distribution function 
of~$Y$.  That is, we can write
\begin{equation}
\label{E:6}
h_{l_3, r_3}(x) = \int{f_{l_3,r_3}(x-y)\,\P\left(Y \in \ddx y \mid (L_3, R_3) = (l_3, r_3)\right)}.
\end{equation}
We now prove in the next two lemmas that the joint measurability of $f_{l_3, r_3}(x)$ with respect to $(l_3, r_3, x)$ ensures the same for $h_{l_3, r_3}(x)$.

\begin{lemma}
\label{L:mble_h_x}
Let $(\Omega, \mathcal{F})$ be 
a measurable space. Let $g: \Omega \times \bbR \to \bbR$ be a nonnegative 
function measurable with respect to the product $\sigma$-field 
$\mathcal{F}  \otimes \mathcal{B}(\bbR)$. Let $V$ and $Y$ be two measurable functions defined on a common probability space and taking values in 
$\Omega$ and $\bbR$, respectively.  Then a regular conditional probability distribution $\P(Y \in dy \mid V )$ for $Y$ given $V$ exists, and the function $\cT g:\Omega \times \bbR \to \bbR$ defined by
\[
\cT g(v, x) := \int g(v,x-y) \P(Y \in \ddx y \mid V = v )
\]
is measurable with respect to the product $\sigma$-field $\mathcal{F}  \otimes \mathcal{B}(\bbR)$.
\end{lemma}
\begin{proof}
This is standard.  For completeness, we provide a proof making use of Kallenberg~\cite[Lemma 3.2(i)]{MR4226142}.  First, since~$Y$ is a real-valued random variable, by Billingsley~\cite[Theorem 33.3]{MR2893652} or Durrett~\cite[Theorem 4.1.18]{MR2722836} or Kallenberg~\cite[Theorem 8.5]{MR4226142} there exists a regular conditional probability distribution for~$Y$ given~$V$; this is a probability kernel from~$\Omega$ to~$\bbR$ and can trivially be regarded as a kernel from $\Omega \times \bbR$ to~$\bbR$.  Let $S := \Omega \times \bbR$, $T := \bbR$, 
$\mu_{v, x}(\ddx y) := \P(Y \in \ddx y \mid V = v)$, and $f((v, x), y) := g(v, x - y)$.  The conclusion of our lemma is then an immediate consequence of the first assertion in the aforementioned Kallenberg lemma.
\end{proof}

We can now handle the measurability of $(l_3, r_3, x) \mapsto h_{l_3,r_3}(x)$.

\begin{lemma}
\label{L:mble_h}
\ \\
\indent
{\rm (a)}~For $0 < t < 1$,  
the mapping $(l_3,r_3, x) \mapsto h_{l_3, r_3}(x)$ is $\mathcal{B}(S_t \times \bbR)$ measurable. 

{\rm (b)}~For $t = 0$, the mapping $(r_3, x) \mapsto h_{0, r_3}(x)$ is $\mathcal{B}((0, 1) \times \bbR)$ measurable. 

{\rm (c)}~For $t = 1$, the mapping $(l_3, x) \mapsto h_{l_3, 1}(x)$ is $\mathcal{B}((0, 1) \times \bbR)$ measurable.
\end{lemma}
\begin{proof}
We prove {\rm (a)}; the claims {\rm (b)} and {\rm (c)} are proved similarly.
Choosing $\Omega = S_t$ and $g(l_3, r_3, x) = f_{l_3,r_3}(x)$ with $(l_3,r_3) \in \Omega$ in~\refL{L:mble_h_x}, we conclude that
\[
(l_3,r_3, x) \mapsto h_{l_3, r_3}(x) = \cT g(l_3, r_3, x) 
\]
is $\mathcal{B}(S_t \times \bbR)$ measurable.
\end{proof}
Recall the definition of $f_t(x)$ at~\eqref{E:density}. It then follows from~\refL{L:mble_h} that $f_t(x)$ is well defined and measurable with respect to $x \in \bbR$ for fixed $0 \leq t \leq 1$. This completes the proof of~\refT{T:main}.

\section{Uniform boundedness of the density functions}
\label{S:boundedness}

In this section, we prove that the functions $f_t$ are 
uniformly bounded for $0 \leq t \leq 1$.
\begin{theorem}
\label{T:bdd}
The densities $f_t$ are uniformly bounded by $10$ for $0 \leq t \leq 1$.
\end{theorem}
The proof
 of~\refT{T:bdd} is our later Lemmas~\ref{L:bound} and~\ref{L:bdd_rho_0}.
In particular, the numerical value $10$ comes from the bound in the last line of the proof of \refL{L:bound} plus two times the bound in the last sentence of the proof of \refL{L:bdd_rho_0}.
A bound on the function 
$f_t$ is established by first finding a bound on the conditional density function $f_{l_3,r_3}$. Observe that the expressions in the proof of \refL{L:main} for  
$f^{(i)}_{l_3, r_3}(x)$ (for $i = 1, \dots, 6$) in Case~3 all involve indicators of intervals.  The six endpoints of these intervals are
\begin{center} 
$2 r_3 - l_3$, $2 r_3$, $1 + r_3 - 2 l_3$, $1 + r_3$, $2 - 2 l_3$, and $2 
- l_3$, 
\end{center}
with $0 < l_3 < t < r_3 < 1$.  The relative order of these six endpoints is determined once we know the value of 
$\rho = \rho (l_3, r_3) := l_3/(1 - r_3)$.  Indeed:
\begin{itemize}
\item When $\rho \in (0, 1/2)$, the order is
\[
2 r_3 - l_3 < 2 r_3 < 1+ r_3 - 2 l_3 < 1 + r_3 < 2 - 2 l_3 < 2 - l_3.
\]
\item When $\rho \in (1/2, 1)$, the order is
\[
2 r_3 - l_3 < 1+ r_3 - 2 l_3 < 2 r_3 <  2 - 2 l_3 < 1 + r_3 < 2 - l_3.
\]
\item When $\rho \in (1, 2)$, the order is
\[
1+ r_3 - 2 l_3 < 2 r_3 - l_3 < 2 - 2 l_3  < 2 r_3 < 2 - l_3 < 1 + r_3.
\]
\item When $\rho \in (2, \infty)$, the order is
\[
1+ r_3 - 2 l_3 < 2 r_3 - l_3 < 2 - 2 l_3  < 2 - l_3 < 2 r_3 < 1 + r_3.
\]
\end{itemize}  

When $\rho = 0$ (i.e.,\ in Case~1 in the proof of \refL{L:main}:\ $l_3 = 0 < r_3 < 1$), the function $f_{l_3, r_3}$ is given by $f_{0, r_3}$ at~\eqref{CD1}.  When $\rho = \infty$ (i.e.,\ in Case~2 in the proof of \refL{L:main}:\ $0 < l_3 < r_3 = 1$), the function $f_{l_3, r_3}$ is given by $f_{l_3, 1}$ at~\eqref{CD2}.  The result~\eqref{E:5} for Case~3 in 
the proof of \refL{L:main} can be reorganized as follows, where we define 
the following functions to simplify notation:
\begin{align*}
m_1(x,l_3,r_3) :=& \frac{1}{r_3(x-r_3)}+\frac{2}{(x+l_3)(x-l_3)},\\
m_2(x,l_3,r_3) :=& \frac{1}{(1-l_3)(x-1+l_3)}+\frac{2}{(x+l_3)(x-l_3)}, 
\\
m_3(x,l_3,r_3) :=& \frac{2}{(x+1-r_3)(x+r_3-1)}+\frac{1}{(1-l_3)(x+l_3-1)}, \\
m_4(x,l_3,r_3) :=& \frac{1}{r_3 (x-r_3)}+\frac{2}{(x+1-r_3)(x+r_3-1)}.
\end{align*}

When $\rho \in (0, 1)$, the conditional density $f_{l_3, r_3}$ satisfies
\begin{align}
f_{l_3,r_3}(x)\, g(l_3,r_3)&=\mathbb{1}{(2r_3-l_3 \leq x < 1+r_3-2l_3)}\, m_1(x,l_3,r_3) \label{rho<1}\\
&+\mathbb{1}{(1+r_3-2l_3\leq x<1+r_3)}\, [m_2(x,l_3,r_3)+m_4(x,l_3,r_3)] \nonumber\\
&+\mathbb{1}{(1+r_3 \leq x < 2-l_3)}\, m_2(x,l_3,r_3). \nonumber
\end{align}

Lastly, when $\rho \in (1, \infty)$, the conditional density $f_{l_3, r_3}$ satisfies
\begin{align}
f_{l_3,r_3}(x)\, g(l_3,r_3) &=\mathbb{1}{(1+r_3-2l_3 \leq x < 2r_3-l_3)}\, m_3(x,l_3,r_3) \label{rho>1}\\
&+\mathbb{1}{(2r_3-l_3\leq x<2-l_3)}\, [m_2(x,l_3,r_3) + m_4(x,l_3,r_3)] \nonumber\\
&+\mathbb{1}{(2-l_3 \leq x < 1+r_3)}\, m_4(x,l_3,r_3). \nonumber
\end{align}

Recall the definition of $f_t(x)$ at~\eqref{E:density}.  For any $x \in \bbR$ we can decompose $f_t(x)$ into three
contributions:
\begin{align*}
f_t(x) &= \E\,h_{L_3, R_3}(x) \\
&= \E [h_{L_3, R_3}(x);\,\rho(L_3, R_3) = 0] + \E [h_{L_3, R_3}(x);\,\rho(L_3, R_3) = \infty] \\
&{} \qquad + \E [h_{L_3, R_3}(x);\,0 < \rho(L_3, R_3) < \infty].
\end{align*}

We first consider the contribution from the case $0 < \rho(L_3, R_3) < \infty$ for any $0 < t < 1$, noting
that this case doesn't contribute to $f_t(x)$ when $t = 0$ or $t=1$. Define
\begin{equation}
\label{E:bdd_rho_finite}
b(l_3, r_3) = \frac{1}{g(l_3, r_3)} \frac{3}{2} \left[ \frac{1}{r_3(r_3 
- l_3)} + \frac{1}{(1-l_3)(r_3 - l_3)}\right].
\end{equation}

\begin{lemma}
\label{L:bound}
The contribution to the density function $f_t$ from the case $0 < \rho(L_3, R_3) <\infty$ is uniformly bounded for $0 < t < 1$. More precisely, given $0 < t < 1$ and $0  < l_3 < t < r_3 < 1$, we have 
\begin{equation}
\label{fhbound}
f_{l_3,r_3}(x) \leq b(l_3, r_3)\mbox{\rm \ \ and\ \ }h_{l_3,r_3}(x) \leq b(l_3, r_3)\mbox{\rm \ for all\ }x \in \bbR; 
\end{equation}
and, moreover, $\E [h_{L_3, R_3}(x);\,0 < \rho(L_3, R_3) < \infty]$ is uniformly bounded for $0 < t < 1$.
\end{lemma}
\begin{proof}
For~\eqref{fhbound}, we need only establish the bound on~$f$.

We start to bound \eqref{rho<1} for $0 < \rho < 1$. The function $m_1$ is 
a decreasing function of $x$ and thus reaches its maximum in~\eqref{rho<1} when $x = 2r_3-l_3$:
\[
m_1(x, l_3, r_3) \leq m_1(2r_3 -l_3, l_3, r_3) = \frac{3}{2} \frac{1}{r_3 (r_3 - l_3)}.
\]
The function $m_2 + m_4$ is also a decreasing function of $x$, and the maximum in~\eqref{rho<1} occurs at $x = 1+r_3 - 2 l_3$. Plug in this $x$-value and use the fact that $1 - l_3 > r_3$ when $\rho < 1$ to obtain
\[
(m_2 + m_4)(x, l_3, r_3) \leq \frac{3}{2} \frac{1}{r_3 (r_3 - l_3)} + \frac{3}{2} \frac{1}{(1-l_3) (r_3 - l_3)}.
\]
Finally, the function $m_2$ is again a decreasing function of $x$, and the maximum in~\eqref{rho<1} occurs at $x = 1 + r_3$.  Plug in this 
$x$-value and use the facts that $1+r_3 - l_3 > 2 r_3 $ and $1 + r_3 + l_3 > r_3 - l_3$ to conclude
\[
m_2(x, l_3, r_3)  \leq \frac{1}{r_3 (r_3 - l_3)} + \frac{1}{(1-l_3) (r_3 - l_3)}.
\]
By the above three inequalities, we summarize that for $0 < \rho < 1$ we have for all~$x$ the inequality
\[
f_{l_3, r_3}(x) \leq \frac{1}{g(l_3, r_3)} \frac{3}{2} \left[ \frac{1}{r_3 (r_3 - l_3)} + \frac{1}{(1-l_3) (r_3 - l_3)} \right].
\]
The method to upper-bound \eqref{rho>1} is similar to that for~\eqref{rho<1}, or one can again invoke symmetry, and we skip the proof here. 

For the expectation of $b(L_3, R_3)$, we see immediately that 
\begin{align*}
\E [b\left(L_3, R_3\right);\,0 < \rho(L_3, R_3) < \infty] 
&= \frac{3}{2} \int_0^{t}{\int_{t}^1\!\left[ \frac{1}{r (r - l)} + \frac{1}{(1-l) (r - l)} \right] \dd r \dd l} \\
&= \frac{\pi^2}{4} + \frac{3}{2} (\ln t)[\ln(1 - t)] \leq \frac{\pi^2}{4} + \frac{3}{2} (\ln 2)^2.
\end{align*}
\end{proof}

For the cases $\rho(L_3, R_3) = 0$ and $\rho(L_3, R_3) = \infty$, we cannot find a constant bound $b(l_3, r_3)$ on the function $f_{l_3, r_3}$ 
such that the corresponding contributions to $\E b(L_3, R_3)$ are bounded 
for $0 \leq t \leq 1$.  Indeed, although we shall omit the proof since it 
would take us too far afield, there exists no such bound $b(l_3, r_3)$.

Instead, to prove the uniform boundedness of the contributions in these two cases, we take a different approach. The following easily-proved lemma 
comes from Gr\"{u}bel and R\"{o}sler~\cite[proof of Theorem~9]{MR1372338}.
\begin{lemma}
\label{L:Stochastic_upper_bdd}
Consider a sequence of independent random variables $V_1, V_2, \dots$, each uniformly distributed on $(1/2, 1)$, and let
\begin{equation}
\label{E:V}
V := 1+ \sum_{n=1}^{\infty} \prod_{k=1}^n V_k.
\end{equation}
Then the random variables $Z(t)$, $0 \leq t \leq 1$, defined at~\eqref{E:3} are all stochastically dominated by $V$. Furthermore, $\E V = 4$; and $V$ has everywhere finite moment generating function~$m$ and therefore superexponential decay in the right tail, in the sense that for any $\theta \in (0, \infty)$ we have  $\P(V \geq x) = o(e^{-\theta x})$ as $x \to \infty$.
\end{lemma}

The following lemma pairs the stochastic upper bound~$V$ on $Z(t)$ with a 
stochastic lower bound.  These stochastic bounds will be useful in later sections.  Recall that the Dickman distribution with support $[1, \infty)$ is the distribution of $Z(0)$. 
\begin{lemma}
\label{L:Stochastic_lower_bdd}
Let $D$ be a random variable following the Dickman distribution with support $[1, \infty)$. Then for all $0 \leq t \leq 1$ we have $D \leq Z(t) \leq V$ stochastically.
\end{lemma}
\begin{proof}
Recall that $\Delta_1(t) = R_1(t) - L_1(t)$.  We first use a coupling argument 
to show that $\Delta_1(t)$ is stochastically increasing for $0 \leq t \leq 1 /2$.  Let $U = U_1 \sim \mbox{uniform}(0, 1)$ be the first key in the construction of~$Z$ as described in \eqref{taukt}--\eqref{E:3}. Let $0 
\leq t_1 < t_2 \leq 1/2$. It is easy to see that $\Delta_1(t_1) = \Delta_1(t_2)$ unless $t_1 < U < t_2$, in which case $\Delta_1(t_1) = U < t_2 \leq 1/2 \leq 1 - t_2 < 1 - U = \Delta_1(t_2)$.  

Let $V_1 \sim \mbox{uniform}(1/2, 1)$, as in \refL{L:Stochastic_upper_bdd}, and let $0 \leq t \leq 1/2$.  Since $\Delta_1(0) \overset{\mathcal{L}}{=} U$ and $\Delta_1(1/2) \overset{\mathcal{L}}{=} V_1$, we immediately have $U \leq \Delta_1(t) \leq V_1$ stochastically.  This implies by a simple induction argument on~$k$ involving the conditional distribution of 
$\Delta_k(t)$ given $\Delta_{k - 1}(t)$ that $D \leq Z(t) \leq V$ stochastically. 
\end{proof}

\begin{remark}
(a)~Note that we do \emph{not} claim that $Z(t)$ is stochastically increasing in $t \in [0, 1/2]$.  Indeed, other than the stochastic ordering $D = Z(0) \leq Z(t)$, we do not know whether any stochastic ordering relations hold among the random variables $Z(t)$.

(b)~The random variable~$V$ can be interpreted as a sort of ``limiting greedy (or `on-line') worst-case {\tt QuickQuant} normalized key-comparisons cost''.  Indeed, if upon each random bisection of the search interval one always chooses the half of greater length and sums the lengths to get $V^{(n)}$, then the limiting distribution of $V^{(n)} / n$ is that of~$V$.
\end{remark}

\begin{lemma}
\label{L:bdd_rho_0}
The contributions to the density function $f_t$ from the cases $\rho(L_3, 
R_3) = 0$ and $\rho(L_3, R_3) = \infty$ are uniformly bounded for 
$0 \leq t \leq 1$.
\end{lemma}

\begin{proof}
Because the Dickman density is bounded above by $e^{-\gamma}$, we need only consider $0 < t < 1$. 
The case $\rho(L_3(t), R_3(t)) = 0$ corresponds to $L_3(t) = 0$, while the case $\rho(L_3(t), R_3(t)) = \infty$ corresponds to $R_3(t) = 1$. 
By symmetry, the contribution from $R_3(t) = 1$ is the same as the contribution from $L_3(1 - t) = 0$, so we need only show that the contribution from $L_3(t) = 0$ is uniformly bounded. We will do this by showing that the larger contribution from $L_2(t) = 0$ is uniformly bounded.

By conditioning on the value of $R_2(t)$, the contribution from $L_2(t) = 
0$ is
\begin{align}
c_t(x) &:= \P(L_2(t) = 0,\,J(t) \in \ddx x) / \ddx x \nonumber \\
\label{cx}
&= \int_{r \in (t, 1)}\!\int_{z > 1}\!{\bf 1}(r \leq x - r z < 1)\,(x - 
r z)^{-1}\,\P(Z(t / r) \in \ddx z)\dd r.
\end{align}
If $z > 1$ and $r \leq x - r z$, then $(x / r) - 1 \geq z > 1$ and so $r < x / 2$.  Therefore we find
\[
c_t(x) \leq \int_t^{\min\{1,\,x / 2\}}\!\int_{\max\{1,\,(x - 1) / r\} < z 
\leq (x / r) - 1}\!(x - r z)^{-1}\,\P(Z(t / r) \in \ddx z)\dd r.
\] 
The integrand (including the implicit indicator function) in the inner integral is an increasing function of~$z$ over the interval $(- \infty, (x / r) - 1]$, with value $r^{-1}$ at the upper endpoint of the interval.  We can thus extend it to a nondecreasing function $\phi \equiv \phi_{x, r}$ with domain 
$\bbR$ by setting $\phi(z) = r^{-1}$ for $z > (x / r) - 1$. It then follows that
\begin{align}
c_t(x) 
&\leq \int_0^{\min\{1,\,x / 2\}} \left[ \int_{\max\{1,\,(x - 1) / r\} < z 
\leq (x / r) - 1}\!(x - r z)^{-1}\,\P(V \in \ddx z) \right. 
\nonumber\\
&{} \hspace{1.5in} + \left. r^{-1}\,\P\left( V > \frac{x}{r} - 1 \right) \right] \dd r \nonumber \\
&\leq \int_0^{\min\{1,\,x / 2\}}\!\int_{\max\{1,\,(x - 1) / r\} < z \leq (x / r) - 1}\!(x - r z)^{-1}\,\P(V \in \ddx z)\dd r \nonumber\\
\label{cxbound}
&{} \hspace{1.5in} + \int_0^{x / 2}\!r^{-1}\,\P\left( V > \frac{x}{r} - 1 
\right) \dd r.
\end{align}

By the change of variables $v = (x / r) - 1$, the second term in~\eqref{cxbound} equals
\[
\int_1^{\infty} (v + 1)^{-1} \P(V > v)\dd v \leq \frac{1}{2} \int_0^{\infty} \P(V > v)\dd v = \frac{1}{2} \E V = 2. 
\]
Comparing the integrals $c_t(x)$ at~\eqref{cx} and the first term in~\eqref{cxbound}, we see that the only constraint that has been discarded is 
$r > t$.  We therefore see by the same argument that produces~\eqref{cx} that the first term in~\eqref{cxbound} is the value of the density for $W 
:= U_1 (1 + U_2 V)$ at~$x$, where $U_1$, $U_2$, and $V$ are independent 
and $U_1$ and $U_2$ are uniformly distributed on $(0, 1)$.  Thus to obtain the desired uniform boundedness of $f_t$ we need only show that~$W$ has 
a bounded density.  For that, it suffices to observe that the conditional 
density of~$W$ given $U_2$ and~$V$ is bounded above by~$1$ (for any values of $U_2$ and~$V$), and so the unconditional density is bounded by~$1$.  
We conclude that $c_t(x) \leq 3$, and this completes the proof.
\end{proof}
\begin{remark}
Based on simulation results, we conjecture that the density functions $f_t$ are uniformly bounded by $e^{-\gamma}$ (the sup-norm of the right-continuous 
Dickman density $f_0$) for $0 \leq t \leq 1$.
\end{remark}

\section{Uniform continuity of the density function $f_t$}
\label{S:uniform continuity}

From the previous section, we know that for $0 < t < 1$ in the case $0 < l < t 
< r < 1$ (i.e.,\ the case $0 < \rho < \infty$) the function $f_{l, r}$ is 
c\`adl\`ag (that is, a right continuous function with left limits)
and bounded above by $b(l, r)$, where the corresponding contribution $\E[b(L_3,R_3);\,0 < \rho(L_3, R_3) < \infty]$ is finite. Applying the dominated convergence theorem, we conclude that the contribution to $f_t$ from this case is also c\`adl\`ag. 

For the cases $0 = l < t < r < 1$ ($\rho = 0$) and $0 < l < t < r = 
1$ ($\rho = \infty$), the functions $f_{0,r}$ and $f_{l,1}$ are both continuous on the real line.  In this section, we will build bounds $b_t(l, 
r)$ (note that these bounds depend on $t$) for these two cases (\refL{L:b1} for $\rho = 0$  and \refL{L:b2} for $\rho = \infty$) in similar fashion as for~\refL{L:bdd_rho_0} such that both $\E[b_t(L_3,R_3)\,;\rho(L_3, R_3) = 0]$ and $\E[b_t(L_3,R_3) ;\rho(L_3, R_3) = \infty]$ are finite for any $0 < t < 1$.  Given these bounds, we can apply the dominated convergence theorem to conclude that the density $f_t$ is c\`adl\`ag. Later, this result will be sharpened substantially in~\refT{T:uniform_cont}.

Let 
$\alpha \approx 3.59112$ be the unique real solution of $1 + x - x \ln x = 0$ and let $\beta := 1 / \alpha \approx 0.27846$. Define
\[
b_1(r) := \frac{2}{\ln r^{-1}} \frac{1}{1+r} \quad \mbox{and} \quad
b_2(r) := \frac{2}{(\ln r^{-1})^2} \frac{1}{r} \beta.
\]

\begin{lemma}
\label{L:b1}
Suppose 
$\rho = 0$, i.e.,\ $0 = l_3 < t < r_3 < 1$. If $t \geq \beta$, then the optimal constant upper bound on $f_{l_3, r_3}$ is 
\[
b_t(l_3, r_3) = b_1(r_3), 
\]
with corresponding contribution
\[
\E[b_t(L_3(t), R_3(t));\,\rho\left(L_3(t), R_3(t)\right) = 0] 
= \int_t^1 \frac{\ln r^{-1}}{1 + r}\dd r \leq \int_{\beta}^1\!\frac{\ln 
r^{-1}}{1 + r}\dd r < \frac{1}{4}
\]
to $\E b_t(L_3(t), R_3(t))$.
If $t < \beta$, then the optimal constant upper bound on $f_{l_3, r_3}$ is
the continuous function
\[
b_t(l_3, r_3) = b_1(r_3) \mathbb{1}{(\beta \leq r_3 < 1)} + b_2(r_3) \mathbb{1}{(t < r_3 < \beta)}
\] 
of $r_3 \in (t, 1]$,
with corresponding contribution 
\begin{align*}
\E[b_t\left(L_3(t), R_3(t)\right);\,\rho\left(L_3(t), R_3(t)\right) = 0]
&= \int_{\beta}^1 \frac{\ln r^{-1}}{1 + r}\dd r + \beta (\ln \beta - \ln t) \\ 
&< \frac{1}{4} + \beta (\ln \beta - \ln t).
\end{align*}
\end{lemma}

\begin{lemma}
\label{L:b2}
Suppose $\rho = \infty$, i.e.,\ $0 < l_3 < t < r_3 = 1$. If $t \leq 1-\beta$, then the optimal constant upper bound on $f_{l_3, r_3}$ is
\[
b_t(l_3, r_3) =b_1(1 - l_3),
\]
with corresponding contribution
\[
\E[b_t\left(L_3(t), R_3(t)\right);\,\rho\left(L_3(t), R_3(t)\right) = 0]
= \int_{1 - t}^1 \frac{\ln r^{-1}}{1 + r}\dd r \leq \int_{\beta}^1\!\frac{\ln r^{-1}}{1 + r}\dd r < \frac{1}{4}.
\]
If $t > 1-\beta$, then the optimal constant upper bound on $f_{l_3, r_3}$ 
is the continuous function
\[
b_t(l_3, r_3) = b_1(1-l_3)\mathbb{1}{(0 < l_3 \leq 1-\beta)} + b_2(1-l_3) \mathbb{1}{(1-\beta < l_3 < t)}
\]
of $l_3 \in [0, t)$, with corresponding contribution
\begin{align*}
\E[b_t\left(L_3(t), R_3(t)\right);\,\rho\left(L_3(t), R_3(t)\right) = 0]
&= \int_{\beta}^1 \frac{\ln r^{-1}}{1 + r}\dd r + \beta [\ln \beta - \ln(1 - r)] \\ 
&< \frac{1}{4} + \beta [\ln \beta - \ln(1 - t)].
\end{align*}
\end{lemma}

We prove \refL{L:b1} here, and \refL{L:b2} follows similarly or by symmetry.
\begin{proof}[Proof of \refL{L:b1}]
When $\rho = 0$, we have $l_3 = 0$, and the conditional density function is $f_{0,r_3}$ in~\eqref{CD1}. 
The expression in square brackets at~\eqref{CD1} is continuous and unimodal in~$x$, with maximum value at $x = 1 + r_3$.  Because the factor $1 / x$ is decreasing, it follows that the maximum value of $f_{0, r_3}(x)$ is the maximum of
\[
\frac{2}{\left(\ln r_3^{-1} \right)^2} \frac{1}{x} \ln \left(\frac{x-r_3}{r_3}\right)
\]
over $x \in [2 r_3, 1 + r_3]$, i.e.,\ the maximum of
\[
\frac{2}{r_3 \left(\ln r_3^{-1} \right)^2} \frac{\ln y}{1 + y}
\]
over $y \in [1, 1 / r_3]$.  A simple calculation shows that the displayed 
expression is strictly increasing for $y \in [1, \alpha]$ and strictly decreasing for $y \in [\alpha, \infty)$.  Thus the maximum for $y \in [1, 1 
/ r_3]$ is achieved at $y = \alpha$ if $\alpha \leq 1 / r_3$ and at 
$y = 1 / r_3$ if $\alpha \geq 1 / r_3$.  Equivalently, $f_{0, r_3}(x)$ is maximized at $x = r_3 (\alpha + 1)$ if $r_3 \leq \beta$ and at $x = 
1 + r_3$ if $r_3 \geq \beta$.  The claims about the optimal constant upper bound on $f_{l_3, r_3}$ and the contribution to $\E b_t(L_3(t), R_3(t))$ now follow readily.   
\end{proof}

\begin{remark}
\label{R:simple_bdd}
If we are not concerned about finding the best possible upper bound, then 
for the case $\rho = 0$ we can choose $b_t(l , r) := b_2(r)$; for the 
case $\rho = \infty$, we can choose $b_t(l , r) := b_2(1-l)$.  These two bounds still get us the desired finiteness of the contributions to 
$\E b_t(L_3, R_3)$ for any $0 < t < 1$.
\end{remark}

\begin{theorem}
\label{T:uniform_cont}
For $0 < t < 1$, the density function $f_t:\bbR \to [0, \infty)$ is uniformly 
continuous.
\end{theorem}

\begin{proof}
Fix 
$0 < t < 1$. By the dominated convergence theorem, the contributions to $f_t(x)$ from $0 = l < t < r < 1$ and $0 < l < t < r = 1$, 
namely, the 
functions
\[
c_0(x) := \int_{r, y}\!{f_{0, r}(x - y)\,\P(L_3(t) = 0,\,R_3(t) \in \ddx r,\,Y \in \ddx y)}
\]
and
\[
c_1(x) := \int_{l, y}\!{f_{l, 1}(x - y)\,\P(L_3(t) \in \ddx l,\,R_3(t) = 1,\,Y \in \ddx y)},
\]
are continuous for $x \in \bbR$.
Further, according to~\eqref{E:4} and~\eqref{E:5}, the contribution from $0 < l < t < r < 1$ is $\sum_{i = 1}^6 c^{(i)}(x)$, where we define
\[
c^{(i)}(x) := \int_{l, r, y}\!{f^{(i)}_{l, r}(x - y)\,\P(Y \in \ddx y \mid (L_3(t), R_3(t)) = (l, r))}\dd l\dd r.
\]
It is easy to see that all the functions $c_0$, $c_1$, and $c^{(i)}$ for $i = 1,\dots,6$ vanish for arguments $x \leq 0$. To prove the uniform continuity of $f_t(x)$ for $x \in \bbR$, it thus suffices to show that each of the six functions $c^{(i)}$ for $i = 1,\dots,6$ is continuous on the real line and that each of the eight functions $c_0$, $c_1$, and $c^{(i)}$ for $i = 1,\dots,6$ vanishes in the limit as argument $x \to \infty$.

Fix $i \in \{1, \dots, 6\}$.  Continuity of $c^{(i)}$ holds since $f^{(i)}_{l, r}$ is bounded by $b(l,r)$ defined at~\eqref{E:bdd_rho_finite} and is continuous except at the boundary of its support. To illustrate, consider, for example, $i = 3$.  For each fixed $0 < l < t < r < 1$ and $x \in \bbR$, 
we have $f_{l, r}^{(3)}(x + h - y) \to f_{l, r}^{(3)}(x - y)$ as $h \to 0$ for all but two exceptional values of~$y$, namely, $y = x - (1 + r - 2 l)$ and $y = x - (1 - r)$.  From the discussion following~\eqref{E:Y} 
and from \refT{T:main}, we know that the conditional law of~$Y$ given $(L_3(t), R_3(t)) = (l, r)$ has a density with respect to Lebesgue measure, and hence the set of two exceptional points has zero measure under this 
law.  We conclude from the dominated convergence theorem that 
\[
\int_{y}\!{f^{(i)}_{l, r}(x - y)\,\P(Y \in \ddx y \mid (L_3(t), R_3(t)) = 
(l, r))}
\]
is continuous in $x \in \bbR$.  It now follows by another application of the dominated convergence theorem that $c^{(i)}$ is continuous on the real line.

Since the eight functions $f_{0, r}$, $f_{l, 1}$, and $f_{l, r}^{(i)}$ for $i = 1, \dots 6$ all vanish for all sufficiently large arguments, another application of the dominated convergence theorem shows that $c_0(x)$, $c_1(x)$, and $c^{(i)}(x)$ for $i = 1, \dots, 6$ all vanish in the limit as $x \to \infty$.  This completes the proof.
\end{proof}

\begin{remark}
For any $0 < t < 1$, by the fact that 
\[
J(t) \geq R_1(t) - L_1(t) \geq \min (t, 1 - t),
\]
we have $\P(J(t) < \min (t, 1 - t)) = 0$ and thus $f_t(\min (t, (1-t))) = 
0$ by~\refT{T:uniform_cont}. This is a somewhat surprising result since we know that the right-continuous Dickman density $f_0$ satisfies 
$f_0(0) = e^{-\gamma} > 0$.
\end{remark}
\begin{remark}
Since $f_0$ 
is both (uniformly) continuous on and piecewise differentiable on $(0, \infty)$, it 
\ignore{
\marginal{JF$\to$WH:\ Here is a much stronger fact about $f_0$.  For each 
$k \in \{0, 1, \dots\}$, the Dickman density $f_0$ is piecewise 
$k$-times continuously differentiable with $k + 1$ pieces, namely, $[j - 1, j]$ for $j = 1, \dots, k$ and $[k, \infty)$.  Is it reasonable to conjecture anything like this for $f_t$ with $0 < t < 1$?}
}
might be natural to conjecture that the densities $f_t$ for $0<t<1$ are also piecewise differentiable.

Later, in \refT{T:Lipschitz_cont}, we prove that the densities $f_t$ are Lipschitz continuous, which implies that each of them is almost everywhere differentiable. 
\end{remark}

\section{Integral equation for the density functions}
\label{S:Integral equation}
In this section we prove that for $0 \leq t < 1$ and $x \in \bbR$, the density function $f_t(x)$ is jointly Borel measurable in $(t, x)$. By symmetry, we can conclude that $f_t(x)$ is jointly Borel measurable in $(t, x)$ for $0 \leq t \leq 1$. We then use this result to establish an integral equation for the densities.
\smallskip

Let $F_t$ denote the distribution function for $J(t)$.  Because $F_t$ is right continuous, it is Borel measurable (for each~$t$).

\begin{lemma}
\label{L:discmeas}
For each positive integer~$n$, the mapping
\[
(t, x) \mapsto F_{\frac{\lfloor n t \rfloor + 1}{n}}(x) \quad (0 \leq t < 
1,\ x \in \bbR)
\]
is Borel measurable.
\end{lemma}

\begin{proof}
Note that
\[
F_{\frac{\lfloor n t \rfloor + 1}{n}}(x) = \sum_{j = 1}^n \mathbb{1}\left( \frac{j - 1}{n} \leq t < \frac{j}{n} \right) F_{\frac{j}{n}}(x).
\]
Each term is the product of a Borel measurable function of~$t$ and a Borel measurable function of~$x$ and so is a Borel measurable of $(t, x)$.  The same is then true of the sum.
\end{proof}

\begin{lemma}
\label{L:Frc}
For each $0 \leq t < 1$ and $x \in \bbR$, as $n \to \infty$ we have
\[
F_{\frac{\lfloor n t \rfloor + 1}{n}}(x) \to F_t(x).
\]
\end{lemma}

\begin{proof}
We reference Gr\"{u}bel and R\"{o}sler~\cite{MR1372338}, who construct a process $J = (J(t))_{0 \leq t \leq 1}$ with $J(t)$ having distribution function $F_t$ for each~$t$ and with right continuous sample 
paths.  It follows (for each $t \in [0, 1)$) that $F_u$ converges weakly to $F_t$ as $u \downarrow t$.  But we know that $F_t$ is a continuous (and even continuously differentiable) function, so for each $x \in \bbR$ we have $F_u(x) \to F_t(x)$ as $u \downarrow t$.  The result follows.
\end{proof}

\begin{proposition}
\label{P:Fmeas}
The mapping
\[
(t, x) \mapsto F_t(x) \quad (0 \leq t < 1,\ x \in \bbR)
\]
is Borel measurable.
\end{proposition}

\begin{proof}
According to Lemmas~\ref{L:discmeas}--\ref{L:Frc}, this mapping is the pointwise limit as $n \to \infty$ of the Borel measurable mappings in \refL{L:Frc}.
\end{proof}

Let $f_t$ denote the continuous density for $F_t$, as in \refT{T:uniform_cont}. 

\begin{theorem}
\label{T:fmeas}
The mapping
\[
(t, x) \mapsto f_t(x) \quad (0 \leq t < 1,\ x \in \bbR)
\]
is Borel measurable.
\end{theorem}

\begin{proof}
By the fundamental theorem of integral calculus, $f_t = F'_t$.  The mapping in question is thus the (sequential) limit of difference quotients that are Borel measurable by \refP{P:Fmeas} and hence is Borel measurable.
\end{proof}

Now we are ready to derive integral equations.  We start with an integral 
equation for the distribution functions $F_t$.

\begin{proposition}
\label{P:Fint}
The distribution functions $(F_t)$ satisfy the following integral equation for $0 \leq t \leq 1$ and $x \in \bbR$:
\begin{equation}
\label{Fint}
F_t(x) = \int_{l \in (0, t)}\!F_{\frac{t - l}{1 - l}}\!\left( \frac{x}{1 - l} - 1 \right)\dd l + \int_{r \in (t, 1)}\!F_{\frac{t}{r}}\!\left( \frac{x}{r} - 1 \right)\dd r.
\end{equation}
\end{proposition}

\begin{proof}
This follows by conditioning on the value of $(L_1(t), R_1(t))$.  Observe 
that each of the two integrands is (by \refP{P:Fmeas} for $t \notin \{0, 1\}$ and by right continuity of $F_0$ and $F_1$ for $t \in \{0, 1\}$) indeed 
[for fixed $(t, x)$] a Borel measurable function of the integrating variable.
\end{proof}

\begin{remark}
It follows from (i)~the changes of variables from~$l$ to $v = (t - l) / 
(1 - l)$ in the first integral in~\eqref{Fint} and from~$r$ to $v = t / 
r$ in the second integral, (ii)~the joint continuity of $f_t(x)$ in $(t, x)$ established later in \refC{C:joint_continuous}, and (iii)~Leibniz's formula that $F_t(x)$ is differentiable with respect to $t \in (0, 1)$ for 
each 
fixed $x \in \bbR$.
\end{remark}

Integral equation~\eqref{Fint} for the distribution functions $F_t$ immediately leads us to an integral equation for the density functions $f_t$.

\begin{proposition}
\label{P:fint}
The continuous density functions $(f_t)$ satisfy the following integral equation for $0 < t < 1$ and $x \in \bbR$:
\[
f_t(x) = \int_{l \in (0, t)}\!(1 - l)^{-1} f_{\frac{t - l}{1 - l}}\!\left( \frac{x}{1 - l} - 1 \right)\dd l + \int_{r \in (t, 1)}\!r^{-1} f_{\frac{t}{r}}\!\left( \frac{x}{r} - 1 \right)\dd r.
\]
\end{proposition}

\begin{proof}
Fix $t \in (0, 1)$.  Differentiate~\eqref{Fint} with respect to~$x$.  It is easily proved by an argument applying the dominated convergence theorem to difference quotients and the mean value theorem that we can differentiate under the integral signs in~\eqref{Fint} provided that 
\begin{equation}
\label{f*}
\int_{l \in (0, t)}\!(1 - l)^{-1} f_{\frac{t - l}{1 - l}}^*\dd l + \int_{r \in (t, 1)}\!r^{-1} f_{\frac{t}{r}}^*\dd r
\end{equation}
is finite, where $f^*_t$ denotes any upper bound on $f_t(x)$ as~$x$ varies over~$\bbR$.  By \refT{T:bdd} we can 
simply choose $f^*_t = 10$. Then~\eqref{f*} equals~$10$ times
\[
- \ln (1-t) - \ln t,
\]
which is finite.
\end{proof}

In the next proposition, we provide an integral equation based on the formula for $f_t$ in~\eqref{E:density}; this integral equation will be useful in the next section. Recall that $Y(t) = \sum_{k=3}^\infty \Delta_k(t)$. Using \eqref{E:Y}, the conditional distribution of $Y(t) / (r_3 - l_3)$ given 
$(L_3, R_3) = (l_3, r_3)$ is the (unconditional) distribution of $ Z(\frac{t - l_3}{r_3-l_3}) = 1 + J(\frac{t - l_3}{r_3-l_3})$. Apply~\refT{T:main} on $Z(\frac{t - l_3}{r_3-l_3})$ leads us to an integral equation for the density function of $J(t)$.

\begin{proposition}
\label{Integral equation}
The continuous density functions $f_t$ for the random variables $J(t) = 
Z(t) - 1$ satisfy the integral equation
\[
f_t(x) = \int{\P((L_3(t), R_3(t)) \in \ddx (l_3, r_3))} \cdot h_t(x \mid l_3, r_3)
\]
for $x \geq 0$, where
\[
h_t(x \mid l_3, r_3) = \int{f_{l_3,r_3}(x-y)\, (r_3 - l_3)^{-1} \, f_{\frac{t - l_3}{r_3 - l_3}}\!\left(\frac{y}{r_3 - l_3} - 1\right)\dd y}.
\]
\end{proposition}

\section{Right-tail behavior of the density function}
\label{S:decay}

In this section we will prove, \emph{uniformly} for 
$0 < t < 1$, that the continuous density functions $f_t$ enjoy the same superexponential decay bound as Gr\"{u}bel and R\"{o}sler~\cite[Theorem~9]{MR1372338} proved for the survival functions $1 - F_t$. By a separate and easier argument, one could include the cases $t = 0, 1$. Let $m_t$ denote the moment generating function of $Z(t)$ and recall that~$m$ denotes 
the moment generating function of $V$ at~\eqref{E:V}. By \refL{L:Stochastic_upper_bdd}, the random variables $Z(t)$, $0 \leq t \leq 1$, are stochastically dominated by $V$.  As a consequence, if $\theta \geq 0$, then 
\[
m_t(\theta) \leq m(\theta)< \infty
\] 
for every $t \in (0, 1)$.

\begin{theorem}
\label{T:Exp_decay}
Uniformly for $t \in (0, 1)$, the continuous {\tt QuickQuant} density functions $f_t(x)$ enjoy superexponential decay when $x$ is large. More precisely, for any $\theta > 0$ we have
\[
f_t (x) < 4 \theta^{-1} e^{2 \theta} m(\theta) e^{-\theta x}
\]
for $x \geq 3$, where~$m$ is the moment generating function of the random 
variable $V$ at~\eqref{E:V}.
\end{theorem}
\begin{proof}
Our starting point is the following equation from the discussion preceding~\refP{Integral equation}:
\begin{align}
\label{heart}
\lefteqn{f_t(x)} \\
&= \int_{l, r} \P((L_3, R_3) \in (\ddx l, \ddx r))\,\int_y\!f_{l, r}(x - y)\,\P\left( (r - l) Z\!\left( \frac{t - l}{r - l} \right) \in \ddx y \right) \nonumber \\
&= \int_{l, r} \P((L_3, R_3) \in (\ddx l, \ddx r))\,\int_z\!f_{l, r}(x - (r - l) z)\,\P\left( Z\!\left( \frac{t - l}{r - l} \right) \in \ddx z \right). \nonumber
\end{align}

By \refL{L:Stochastic_upper_bdd}, for any $\theta \in \bbR$ we can obtain a probability measure $\mu_{t, \theta}(\ddx z) := m_t(\theta)^{-1} e^{\theta z} \P(Z(t) \in \ddx z)$ by exponential tilting.  Since 
$m_t(\theta) \leq m(\theta) < \infty$ for every $\theta \geq 0$ and $t \in (0, 1)$, we can rewrite and bound~\eqref{heart} as follows (for any $\theta \geq 0$): 
\begin{align*}
\lefteqn{f_t(x)} \\
&= \int_{l, r} \P((L_3, R_3) \in (\ddx l, \ddx r))\,m_{\frac{t - l}{r - 
l}}(\theta)\,\int_z\!e^{- \theta z}\,f_{l, r}(x - (r - l) z)\,\mu_{\frac{t - l}{r - l}, \theta}(\ddx z) \\
&\leq m(\theta) \int_{l, r} \P((L_3, R_3) \in (\ddx l, \ddx r))\,\int_z\!e^{- \theta z}\,f_{l, r}(x - (r - l) z)\,\mu_{\frac{t - l}{r - l}, \theta}(\ddx z).
\end{align*}
Recall 
that $f_{l, r}(x)$ is bounded above by $b_t(l, r)$ 
(Lemmas~\ref{L:bound} and \ref{L:b1}--\ref{L:b2}) and vanishes for $x \geq 2$.  Therefore, if $\theta \geq 0$ then
\begin{align}
\lefteqn{f_t(x)} \nonumber \\
&\leq m(\theta) \int_{l, r} \P((L_3, R_3) \in (\ddx l, \ddx r))\,b_t(l, r) \int_{z > \frac{x - 2}{r - l}}\!e^{- \theta z}\,\mu_{\frac{t - l}{r - l}, \theta}(\ddx z) \nonumber \\
\label{exp}
&\leq m(\theta) \int_{l, r} \P((L_3, R_3) \in (\ddx l, \ddx r))\,b_t(l, r)\,\exp\left(- \theta\,\frac{x - 2}{r - l} \right).
\end{align}

Suppose $x \geq 3$ and $\theta > 0$.  We now consider in turn the contribution to~\eqref{exp} for $l = 0$, for $r = 1$, and for $0 < l < t < r 
< 1$.  For $l = 0$, the contribution is $m(\theta)$ times the following:
\begin{align*}
\lefteqn{\hspace{-0.5in}\int_t^1\!r^{-1} \beta \exp[- \theta r^{-1} (x - 2)]\dd r} \\
&\leq \beta \int_0^1\!r^{-2} \exp[- \theta r^{-1} (x - 2)]\dd r \\
&= \beta\,[\theta (x - 2)]^{-1} \exp[- \theta (x - 2)]
\leq  \beta \theta^{-1} e^{2 \theta} e^{- \theta x}.
\end{align*}
Similarly (or symmetrically), the contribution for $r = 1$ is bounded by the same $\beta \theta^{-1} e^{2 \theta} m(\theta) e^{- \theta x}$.
For $0 < l < t < r < 1$, by symmetry we may without loss of generality suppose that $0 < t \leq 1/2$, and then the contribution is $\frac{3}{2} m(\theta)$ times the following:
\begin{align*}
\lefteqn{\int_0^t{\int_{t}^1\!\left[ \frac{1}{r (r - l)} + \frac{1}{(1-l) 
(r - l)} \right] \exp\left(- \theta\,\frac{x - 2}{r - l} \right) \dd r \dd l}} \\
&= \int_0^t{\int_{t - l}^{1 - l}\!\left[ \frac{1}{(s + l) s} + \frac{1}{(1-l) s} \right] \exp[- \theta s^{-1} (x - 2)] \dd s \dd l} \\
&= \int_0^t{\int_{t - l}^{1 - l}\!(1 - l)^{-1} (s + l)^{-1} (1 + s) s^{-1} \exp[- \theta s^{-1} (x - 2)] \dd s \dd l} \\
&\leq 4 \int_0^t{\int_{t - l}^{1 - l}\!s^{-2} \exp[- \theta s^{-1} (x - 2)] \dd s \dd l} \\
&\leq 4 \int_0^{1/2}{\int_0^1\!s^{-2} \exp[- \theta s^{-1} (x - 2)] \dd s 
\dd l} \\
&= 2 [\theta (x - 2)]^{-1} \exp[- \theta (x - 2)]
\leq 2 \theta^{-1} e^{2 \theta} e^{- \theta x}.  
\end{align*}

Summing all the contributions, we find
\begin{equation}
\label{decay}
f_t(x) \leq (3 + 2 \beta) \theta^{-1} e^{2 \theta} m(\theta) e^{- \theta x} < 4\,\theta^{-1} e^{2 \theta} m(\theta) e^{- \theta x}, 
\end{equation}
for any $0 < t < 1$, $x \geq 3$, and $\theta > 0$, demonstrating the uniform superexponential decay.
\end{proof}

\begin{remark}
Since $f_t$ is uniformly bounded by $10$ by~\refT{T:bdd}, for any $\theta 
> 0$, by choosing the coefficient $C_{\theta} := \max \{ 10 e^{3 \theta}, 4 \theta^{-1} e^{2 \theta} m(\theta)\}$, we can extend the superexponential bound on $f_t(x)$ in~\refT{T:Exp_decay} for $x \geq 3$ to $x \in \bbR$ as
\begin{equation}
\label{E:exp_bdd}
f_t(x) \leq C_{\theta} e^{-\theta x} \mbox{ for } x \in \bbR \mbox{ and } 
0 < t < 1.
\end{equation}
Note 
that this bound is not informative for $x \leq \min \{t,1-t \}$ since we know $f_t(x) = 0$ for such~$x$ by \refT{T:uniform_cont}, but it will simplify our proof of~\refT{T:Lipschitz_cont}.
\end{remark}

\section{Positivity and Lipschitz continuity of the continuous density functions}
\label{S:other}

In this section we establish several properties of the continuous density 
function $f_t$. We prove that $f_t(x)$ is positive for every $x > \min \{t, 1-t\}$ (\refT{T:pos}), Lipschitz continuous for $x \in \bbR$ (\refT{T:Lipschitz_cont}), and jointly continuous for $(t,x) \in (0,1) \times \bbR$ (\refC{C:joint_continuous}).

\subsection{Positivity}
\label{S:pos}

\begin{theorem}
\label{T:pos}
For any $0 < t < 1$, the continuous density $f_t$ satisfies
\[
f_t(x) > 0\mbox{\rm \ if and only if $x > \min\{t, 1 - t\}$}.
\]
\end{theorem}

We already know that $f_t(x) = 0$ if $x \leq \min\{t, 1 - t\}$, so we need only prove the ``if'' assertion.
Our starting point for the proof is the following lemma.  Recall from Chung~\cite[Exercise 1.6]{MR1796326} that a point~$x$ is said to belong to the support of a distribution function~$F$ if for every $\epsilon > 0$ we have
\begin{equation}
\label{supp}
F(x + \epsilon) - F(x - \epsilon) > 0.
\end{equation}
Note that to prove that~$x$ is in the support of~$F$ we may choose any $\epsilon_0(x) > 0$ and establish~\eqref{supp} for all 
$\epsilon \in (0, \epsilon_0(x))$.

\begin{lemma}
\label{L:supp}
For any $0 < t < 1$, the support of $F_t$ is $[\min\{t, 1 - t\}, \infty)$.
\end{lemma}

\begin{proof}
Clearly the support of $F_t$ is contained in $[\min\{t, 1 - t\}, \infty)$, so we need only establish the reverse containment. 
Since $F_t = F_{1 - t}$ by symmetry, we may fix $t \leq 1 / 2$.  Also fixing $x \geq t$, write
\[
x = t + K + b
\]
where $K \geq 0$ is an integer and $b \in [0, 1)$.  We will show that~$x$ 
belongs to the support of $F_t$.  Let
\[
A := \bigcap_{k = 1}^K \{1 - k \epsilon < R_k < 1 - (k - 1) \epsilon\}.
\]
We break our analysis into four cases:\ (i) $t < b < 1$, (ii) $b = t$, (iii) $0 < b < t$, and (iv) $b = 0$.
\medskip

\par\noindent
(i) $t < b < 1$.  Let
\[
B := \{b < R_{K + 1} < b + \epsilon\}\,\bigcap\,\{t < R_{K + 2} < t + \epsilon\}\,\bigcap\,\{t - \epsilon < L_{K + 3} < t\} 
\]
and
\begin{equation}
\label{Cdef}
C := \left\{ 0 \leq \sum_{k = K + 4}^{\infty} \Delta_k < 6 \epsilon \right\}.
\end{equation}
Upon observing that for $\delta_1, \delta_2 \in (0, \epsilon)$ we have by 
use of Markov's inequality that
\begin{align}
\lefteqn{\P(C \mid (L_{K + 3}, R_{K + 3}) = (t - \delta_1, t + \delta_2))} \nonumber \\ 
&= \P\left( (\delta_1 + \delta_2) J\left( \frac{\delta_1}{\delta_1 + \delta_2} \right) < 6 \epsilon \right)
\geq \P\left( J\left( \frac{\delta_1}{\delta_1 + \delta_2} \right) < 3 \right) \nonumber \\
\label{Markov}
&\geq 1 - \frac{1}{3} \E J\left( \frac{\delta_1}{\delta_1 + \delta_2} \right)
\geq 1- \frac{1}{3} \left[ 1 + 2 H\left( \frac{1}{2} \right) \right] > 0.2 > 0.
\end{align}
We then see that $\P(A \cap B \cap C) > 0$ for all sufficiently small~$\epsilon$.  But if the event $A \cap B \cap C$ is realized, then
\[
J(t) > \sum_{k = 1}^K (1 - k \epsilon) + b + t = x - {{K + 1} \choose 
{2}} \epsilon
\]
and
\[
J(t) < \sum_{k = 1}^K [1 - (k - 1) \epsilon] + (b + \epsilon) + (t + \epsilon) + 2 \epsilon + 6 \epsilon \leq x + 10 \epsilon.
\]
We conclude that~$x$ is in the support of $F_t$.
\medskip

\par\noindent
(ii) $b = t$.  Let
\[
B := \{t < R_{K + 2} < R_{K + 1} < t + \epsilon\}\,\bigcap\,\{t - \epsilon < L_{K + 3} < t\} 
\]
and define~$C$ by~\eqref{Cdef}.
We then see that $\P(A \cap B \cap C) > 0$ for all sufficiently small~$\epsilon$.  But if the event $A \cap B \cap C$ is realized, then
\[
J(t) > \sum_{k = 1}^K (1 - k \epsilon) + t + t =  x - {{K + 1} \choose {2}} \epsilon
\]
and
\[
J(t) < \sum_{k = 1}^K [1 - (k - 1) \epsilon] + 2 (t + \epsilon) +2 \epsilon + 6 \epsilon \leq x + 10 \epsilon.
\]
We conclude that~$x$ is in the support of $F_t$.
\medskip

\par\noindent
(iii) $0 < b < t$.  Let
\[
B := \{t < R_{K + 1} < t + \epsilon\} \bigcap \{t - b - \epsilon < L_{K 
+ 2} < t - b\} \bigcap \{t - \epsilon < L_{K + 3} < t\}
\]
and define~$C$ by~\eqref{Cdef}.
We then see that $\P(A \cap B \cap C) > 0$ for all sufficiently small~$\epsilon$.  But if the event $A \cap B \cap C$ is realized, then
\[
J(t) > \sum_{k = 1}^K (1 - k \epsilon) + t + b =  x - {{K + 1} \choose {2}} \epsilon
\]
and
\[
J(t) < \sum_{k = 1}^K [1 - (k - 1) \epsilon] + (t + \epsilon) + (b + 2 \epsilon) + 2 \epsilon + 6 \epsilon \leq x + 11 \epsilon.
\]
We conclude that~$x$ is in the support of $F_t$.
\medskip

\par\noindent
(iv) $b = 0$.  Let
\[
B := \{t < R_{K + 1} < t + \epsilon\} \bigcap \{t - \epsilon < L_{K + 2} < t\}
\]
and define~$C$ by~\eqref{Cdef}, but with $K + 4$ there changed to $K + 3$.
We then see that $\P(A \cap B \cap C) > 0$ for all sufficiently small~$\epsilon$.  But if the event $A \cap B \cap C$ is realized, then
\[
J(t) > \sum_{k = 1}^K (1 - k \epsilon) + t =  x - {{K + 1} \choose {2}} \epsilon
\]
and
\[
J(t) < \sum_{k = 1}^K [1 - (k - 1) \epsilon] + (t + \epsilon) +2 \epsilon + 6 \epsilon \leq x + 9 \epsilon.
\]
We conclude that~$x$ is in the support of $F_t$.
\end{proof}

We next use~\eqref{cx} together with \refL{L:supp} to establish \refT{T:pos} in a special case.

\begin{lemma}
\label{L:pos}
For any $0 < t < 1$, the continuous density $f_t$ satisfies
\[
f_t(x) > 0\mbox{\rm \ for all $x > 2 \min\{t, 1 - t\}$}.
\]
\end{lemma}

\begin{proof}
We may fix $t \leq 1/2$ and $x > 2 t$ and prove $f_t(x) > 0$.  To do this, we first note from~\eqref{cx} that
\begin{align*}
f_t(x) &\geq c_t(x) = \P(L_2(t) = 0,\,J(t) \in \ddx x) / \ddx x \\
&= \int_{r \in (t, 1)} \int\!{\bf 1}(r \leq x - r z < 1)\,(x - r z)^{-1}\,\P(Z(t / r) \in \ddx z)\dd r \\
&\geq \int_{r \in (t, 1)} \int_{(x - 1) / r}^{(x / r) - 1}\,\P(Z(t / r) \in \ddx z)\dd r \\
&\geq \int_{r \in (t, 1)} \P\left( \frac{x - 1}{r} < Z\left( \frac{t}{r} \right) < \frac{x}{r} - 1 \right)\dd r.
\end{align*}
According to \refL{L:supp}, for the integrand in this last integral to be 
positive, it is necessary and sufficient that $(x - 1) / r < (x / r) - 1$ 
(equivalently, $r < 1$) and
\[
\tfrac{x}{r} - 1 > 1 + \min\{\tfrac{t}{r}, 1 - \tfrac{t}{r}\}
\]
[for which it is sufficient that $r < (x + t) / 3$].  Thus
\[
f_t(x) \geq \int_{r \in (t, \min\{(x + t) / 3, 1\})} \P\left( \frac{x - 1}{r} < Z\left( \frac{t}{r} \right) < \frac{x}{r} - 1 \right)\dd r > 0
\]
because (recalling $x > 2 t$) the integrand here is positive over the nondegenerate interval of integration.
\end{proof}

Finally, we use a different contribution to $f_t(x)$ together with \refL{L:pos} to establish \refT{T:pos}.

\begin{proof}[Proof of \refT{T:pos}]
We may fix $t \geq 1 / 2$ and $x > 1 - t$ and prove $f_t(x) > 0$.  To do this, we first note that
\begin{align*}
f_t(x) 
&\geq \int_{l \in (0, t)}\!\int_{r \in (t, 1)}\!\P(L_1(t) = L_2(t) \in \dd l,\,R_2(t) \in \ddx r) \\ 
&{} \qquad \left[ \P\left( (r - l) J\left( \frac{t - l}{r - l} \right) \in \dd x - [(1 - l) + (r - l)] \right) / \ddx x \right] \\
&= \int_{l \in (0, t)}\!\int_{r \in (t, 1)}\! (1-l)^{-1} (r - l)^{-1} f_{\frac{t - l}{r - l}} \left( \frac{x - (1 + r - 2 l)}{r - l} \right)\dd r\dd l.
\end{align*}
According to \refL{L:pos}, for the integrand in this double integral to be positive, it is sufficient that
\[
\frac{x - (1 + r - 2 l)}{r - l} > 2 \min\left\{ \frac{t - l}{r - l},\,\frac{r - t}{r - l} \right\},
\]
or, equivalently,
\[
x > \min\{1 + 2 t + r - 4 l, 1 - 2 t + 3 r - 2 l\}.
\]
This strict inequality is true (because $x > 1 - t$) when $l = t$ and $r = t$ and so, for sufficiently small $\epsilon > 0$ is true for 
$l \in (t - \epsilon, t)$ and $r \in (t, t + \epsilon)$.  Thus
\[
f_t(x) 
\geq \int_{l \in (t - \epsilon, t)}\!\int_{r \in (t, t + \epsilon)}\!(1 - 
l)^{-1} (r - l)^{-1} f_{\frac{t - l}{r - l}} \left( \frac{x - (1 + r - 2 l)}{r - l} \right)\dd r\dd l > 0
\]
because the integrand here is positive over the fully two-dimensional rectangular region of integration.
\end{proof}

\subsection{Lipschitz continuity}
\label{S:Lip}

We now prove that, for each $0< t <1$, the density function $f_t$ is Lipschitz continuous, which is a result stronger than~\refT{T:uniform_cont}. 

\begin{theorem}
\label{T:Lipschitz_cont}
For each $0 < t < 1$, the density function $f_t$ is Lipschitz continuous. 
\end{theorem}

That is, there exists a constant $\Lambda_t \in (0, \infty)$ such that for any $x, z \in \bbR$, we have $| f_t(z) - f_t(x) | \leq \Lambda_t | z - x |$.  The proof of \refT{T:Lipschitz_cont} will reveal that one can take
$\Lambda_t = \Lambda [t^{-1} \ln t] [(1 - t)^{-1} \ln (1 - t)]$ for some constant $\Lambda < \infty$.  Thus the densities $f_t$ are in fact uniformly Lipschitz continuous for~$t$ in any compact subinterval of $(0, 1)$.
 \smallskip

We break the proof of~\refT{T:Lipschitz_cont} into two lemmas.~\refL{L:Lipschitz_1} deals with the contribution to $f_t$ from the disjoint-union event $\{0 = L_3(t) < t < R_3(t) < 1\} \cup \{0 < L_3(t) < t < R_3(t) = 
1\}$ while ~\refL{L:Lipschitz_2} deals with the contribution from the event $\{0 < L_3(t) < t < R_3(t) < 1\}$.
\begin{lemma}
\label{L:Lipschitz_1}
For each $0 < t < 1$, the contribution to $f_t$ from the event $\{0 = L_3(t) < t < R_3(t) < 1\} \cup \{0 < L_3(t) < t < R_3(t) = 1\}$ is Lipschitz continuous.
\end{lemma}

\begin{proof}
Fix $0 < t < 1$.  By symmetry, we need only consider the contribution to $f_t(x)$ from the event 
$\{0 = L_3(t) < t < R_3(t) < 1\}$.  Recall that this contribution is 
\[
c_0(x) := \frac{1}{2} \int_{r, y}\!{(\ln r)^2 \, f_{0, r}(x - y) \, \P(Y \in dy \, | \,L_3(t) = 0,\,R_3(t) = r)\dd r},
\]
and that the conditional probability in the integrand can be written as 
\[
\P(Y \in \dd y \, | \,L_3(t) = 0,\,R_3(t) = r) = \frac{1}{r} \, f_{\frac{t}{r}}\left(\,\frac{y}{r} - 1\,\right)\dd y.
\]
Let $z, x \in \bbR$ with $z > x$ and fixed $r \in (t, 1)$. Writing
\[
d_r (x, z, y) :=\frac{1}{2} (\ln r)^2 [f_{0,r}(z-y) - f_{0,r}(x-y)],
\]
we are interested in bounding the absolute difference
\[
| c_0(z) - c_0(x) | \leq \int_{r, y}\!{ | d_r (x, z, y) | \, \frac{1}{r} \, f_{\frac{t}{r}}\left(\,\frac{y}{r} - 1\,\right)\dd y. }
\]
\smallskip

\par\noindent
\emph{Case} 1.\ $z - x \leq 1- r$.
We bound $d_r (x, z, y)$ for~$y$ in each of the seven subintervals of the 
real line determined by the six partition points 
\[
x - 2 < z - 2 \leq x - (1+r) < z - (1+r) \leq x - 2 r < z - 2r,
\]
and then the contribution to our bound on $| c_0(z) - c_0(x) |$ from all~$y$ in that subinterval (and all~$r$ satisfying the restriction of Case~1).
For the two subcases $y \leq x - 2$ and $y > z - 2 r$, we have $d_r (x, z, y) = 0$.  We bound the five nontrivial subcases 
as follows.
\smallskip

\par\noindent
\emph{Subcase} 1(a).\ $x - 2 < y \leq z-2$.
We have 
\[
| d_r (x, z, y) | = \left| \frac{1}{x- y} \right| \ln \left(\frac{1}{x- 
y - 1}\right) \leq \frac{1}{1+r} \ln \frac{1}{r},
\]
and the contribution to $| c_0(z) - c_0(x) | $ is bounded by
\[
\int_{r=t}^1 \frac{1}{r(1+r)} \left( \ln \frac{1}{r} \right) \int_{y = 
x-2}^{z-2} {f_{\frac{t}{r}}\left(\,\frac{y}{r} - 1\,\right) \dd y \dd r} \leq 10(z-x) \frac{1-t}{t(1+t)} \ln \frac{1}{t}
\]
since $f_{t/r}$ is bounded by $10$.
\smallskip

\par\noindent
\emph{Subcase} 1(b).\ $z - 2 < y \leq x - (1+r)$.
We have
\begin{align*}
d_r (x, z, y) 
&= \frac{1}{z-y} \ln \left(\frac{1}{z-y-1}\right) - \frac{1}{x-y} \ln \left(\frac{1}{x-y-1}\right)\\
&= \frac{1}{z-y} \left[ \ln  \left(\frac{1}{z-y-1} \right) - \ln \left( 
\frac{1}{x-y-1}\right) \right] \\
&{} \qquad + \left(\frac{1}{z-y} - \frac{1}{x-y} \right) \ln \left(\frac{1}{x-y-1}\right).
\end{align*}
Observe that $z - y > x - y > 1+ r$ and that the function $\ln[1 / (x-1)]$ is differentiable for $x > 1$.  We then use the mean value theorem to obtain
\begin{align*}
| d_r (x, z, y) | &\leq \frac{1}{1+r} \left| \ln  \left(\frac{1}{z-y-1} \right) - \ln \left( \frac{1}{x-y-1}\right) \right| + \frac{(z-x)}{(1+r)^2} \ln \frac{1}{r}\\
&\leq (z-x) \left[ \frac{1}{r(1+r)}+ \frac{1}{(1+r)^2} \ln \frac{1}{r} \right].
\end{align*}
The contribution to $| c_0(z) - c_0(x) | $ is then bounded by
\[
(z-x) \int_{r=t}^1 \left[ \frac{1}{r(1+r)}+ \frac{1}{(1+r)^2} \ln \frac{1}{r} \right] \dd r \leq (z-x) \frac{1-t}{1+t} \left(\frac{1}{t}+ \frac{1}{1+t} \ln \frac{1}{t}\right).
\]
\smallskip

\par\noindent
\emph{Subcase} 1(c).\ $x - (1+r) < y \leq z - (1+r)$.
We have
\begin{align*}
d_r (x, z, y) &= \frac{1}{z-y} \ln \left(\frac{1}{z-y-1}\right) - \frac{1}{x-y} \ln \left(\frac{x-y-r}{r}\right)\\
&= \frac{1}{z-y} \left[ \ln \left(\frac{1}{z-y-1}\right) - \ln \left(\frac{x-y-r}{r}\right) \right] \\
&+ \left(\frac{1}{z-y} - \frac{1}{x-y}\right) \ln \left(\frac{x-y-r}{r}\right).
\end{align*}
Using the inequalities $z - y \geq 1+r$ and $x - y > 2r$, we have
\[
| d_r (x, z, y) | = \frac{1}{1+r} \left|\ln \left(\frac{1}{z-y-1}\right) - \ln \left(\frac{x-y-r}{r}\right) \right| + \frac{z-x}{2r(1+r)} \ln \frac{1}{r}.
\]
We can bound the absolute-value term here by
\begin{align*}
&\left|\ln \frac{1}{z-y-1} - \ln \frac{1}{(1+r) - 1} \right| + \left|\ln \frac{(1+r) - r}{r} - \ln \frac{x-y-r}{r} \right|\\
&\leq \frac{1}{r} [z - y - (1+r)] + \frac{1}{r} [(1+r) - (x-y)] = (z-x)\frac{1}{r},
\end{align*}
where the above inequality comes from two applications of the mean value theorem. 
The contribution to $| c_0(z) - c_0(x) | $ is then bounded by
\[
(z-x) \int_{r = t}^1 \left[ \frac{1}{r(1+r)} + \frac{1}{2r(1+r)} \ln \frac{1}{r} \right] \dd r \leq (z-x) \frac{1-t}{t(1+t)} \left(1+\frac{1}{2} 
\ln \frac{1}{t} \right).
\]
\smallskip

\par\noindent
\emph{Subcase} 1(d).\ $z - (1+r) < y \leq x - 2r$.
We have
\begin{align*}
d_r (x, z, y) 
&= \frac{1}{z-y} \ln \left( \frac{z-y-r}{r}\right) - \frac{1}{x-y} \ln \left( \frac{x-y-r}{r}\right)\\
&= \frac{1}{z-y} \left[ \ln \left( \frac{z-y-r}{r}\right) - \ln \left( \frac{x-y-r}{r}\right) \right]\\
&{} \qquad + \left(\frac{1}{z-y} - \frac{1}{x-y} \right) \ln \left( \frac{x-y-r}{r}\right).
\end{align*}
Using the inequality $z - y > x - y \geq 2r$, we obtain
\begin{align*}
| d_r (x, z, y)| &\leq \frac{1}{2 r} [\ln(z-y-r) - \ln (x-y-r)] + \frac{z-x}{(2 r)^2} \ln \frac{1}{r}\\
&\leq (z-x) \left[ \frac{1}{2 r} \frac{1}{r} + \frac{1}{(2 r)^2} \ln \frac{1}{r} \right]
\end{align*}
by the differentiability of $\ln(x-r)$ for $x > r$ and the mean value theorem. The contribution to $| c_0(z) - c_0(x) | $ is then bounded 
by
\[
(z-x) \int_{r = t}^1 \left[ \frac{1}{2 r} \frac{1}{r} + \frac{1}{(2 r)^2} \ln \frac{1}{r} \right] \dd r 
= (z-x)\, \frac{(1 - t )+ \ln(1 / t)}{4 t}.
\]
\smallskip

\par\noindent
\emph{Subcase} 1(e).\ $x - 2 r < y \leq z - 2r$.
Using the inequality $2r \leq z - y < 1+r$, we have
\[
| d_r (x, z, y) | = \frac{1}{z-y} \ln \left( \frac{z-y-r}{r} \right) \leq \frac{1}{2r} \ln \frac{1}{r},
\]
and the contribution to $| c_0(z) - c_0(x) | $ is then bounded by
\[
\int_{r=t}^1 \frac{1}{2r^2} \left( \ln \frac{1}{r} \right) \int_{y = x-2r}^{z-2r} {f_{\frac{t}{r}}\left(\,\frac{y}{r} - 1\,\right) \dd y  \dd r} \leq 10(z-x) \frac{1-t}{2t^2} \ln \frac{1}{t}.
\]
This completes the proof for Case~1.
\smallskip

\par\noindent
\emph{Case} 2.\ $z - x > 1-r$.
We directly bound
\[
| d_r (x, z, y) |  \leq \frac{1}{2} (\ln r)^2 [f_{0,r}(z-y) + f_{0,r}(x-y)].
\]
If $z - x \leq 1 - t$, use the bound in~\refR{R:simple_bdd}; we can then bound the contribution to $| c_0(z) - c_0(x) |$ by
\[
\int_{r = 1 - (z-x)}^1 {\frac{2 \beta}{r} \dd r} \leq (z-x) \frac{2 \beta}{t}.
\]
On the other hand, if $z - x > 1 - t$, then we can bound the contribution 
to $| c_0(z) - c_0(x) | $ by
\[
\frac{z-x}{1-t} \int_{r = t}^1 {\frac{2 \beta}{r} \dd r} \leq (z-x) \frac{2 \beta}{t}.
\]
This completes the proof for Case~2.
We conclude that $c_0$ is a Lipschitz continuous function; note that the Lipschitz constant we have obtained depends 
on~$t$.
\end{proof}

\begin{lemma}
\label{L:Lipschitz_2}
For each $0 < t < 1$,  the contribution to $f_t$ from the event $\{0 < L_3(t) < t < R_3(t) < 1\}$ is Lipschitz continuous.
\end{lemma}

\begin{proof}
Fix $0 < t < 1$. According to~\eqref{E:4} and~\eqref{E:5}, the contribution from the event in question to $f_t(x)$ is $\sum_{i = 1}^6 c^{(i)}(x)$, where we define
\[
c^{(i)}(x) := \int_{l, r, y}\!{f^{(i)}_{l, r}(x - y)\,\P(Y \in \ddx y \mid (L_3(t), R_3(t)) = (l, r))}\dd l\dd r.
\]
We show here that $c^{(3)}$ is Lipschitz continuous, and the claims that the other contributions $c^{(i)}$ are Lipschitz continuous are proved 
similarly. 

Let $x, z \in \bbR$ with $z > x$ and consider $(l,r)$ satisfying $0 < l < 
t < r < 1$. Define
\[
d_{l,r}(x, z, y) := f_{l,r}^{(3)}(z-y) - f_{l,r}^{(3)}(x-y)
\]
and reformulate
\[
f_{l,r}^{(3)}(x) = \mathbb{1}(1+r-2l\leq x < 1+r) \frac{1}{x}\left(\frac{1}{x+1-r} + \frac{1}{x+r-1}\right)
\]
from the expression for $f_{l,r}^{(3)}(x)$ found in \refS{S:density}.
We are interested in bounding the quantity
\begin{equation}
\label{E:c_3_Lip}
|c^{(3)}(z) - c^{(3)}(x)| \leq \int_{l, r, y}\!{|d_{l,r}(x, z, y)|\,\P(Y \in \ddx y \mid (L_3, R_3) = (l, r))}\dd l\dd r,
\end{equation}
where the conditional probability can also be written in density terms as
\[
\P(Y \in \ddx y \mid (L_3, R_3) = (l, r)) = \frac{1}{r-l} f_{\frac{t-l}{r-l}}\left(\frac{y}{r-l} - 1\right)\dd y.
\]
Just as we did for \refL{L:Lipschitz_1}, we break the proof into consideration of two cases. 
\smallskip

\par\noindent
\emph{Case} 1.\ $z -x < 2 l$. As in the proof for Case~1 of \refL{L:Lipschitz_1}, we bound $d_{l, r}(x, z, y)$ for~$y$ in each of the five subintervals of the real line determined by the four partition points  
\[
x-(1+r) < z -(1+r) < x - (1+r - 2l) < z-(1+r-2l).
\]
For the two subcases $y \leq x-(1+r)$ and $y > z-(1+r-2l)$, we have $d_r(x, z, y) = 0$.  We bound the three nontrivial subcases (listed in order 
of convenience of exposition, not in natural order) 
as follows. 
\smallskip

\par\noindent
\emph{Subcase} 1(a).\ $z -(1+r) < y \leq x - (1+r - 2l)$.
We have
\begin{align*}
&d_{l,r}(x, z, y) \\
&= \frac{1}{z-y} \left(\frac{1}{z-y+1-r} - \frac{1}{z-y+r-1}\right) \\
&- \frac{1}{x-y} \left(\frac{1}{x-y+1-r} - \frac{1}{x-y+r-1}\right)\\
&= \frac{1}{z-y} \left(\frac{1}{z-y+1-r} - \frac{1}{x-y+1-r}\right) + \left(\frac{1}{z-y} - \frac{1}{x-y}\right) \frac{1}{x-y+1-r}\\
&- \frac{1}{z-y} \left(\frac{1}{z-y+r-1} - \frac{1}{x-y+r-1} \right) - \left(\frac{1}{z-y} - \frac{1}{x-y}\right) \frac{1}{x-y+r-1}.
\end{align*}
Using the 
inequality $z-y > x-y \geq 1+r-2l$, we obtain
\begin{align*}
& |d_{l,r}(x, z, y) | \\
&\leq \frac{1}{1+r-2l} \frac{z-x}{(2-2l)^2} + \frac{z-x}{(1+r-2l)^2}\frac{1}{2-2l}\\
&{} \qquad + \frac{1}{1+r-2l} \frac{z-x}{(z-y+r-1)(x-y+r-1)} + \frac{z-x}{(1+r-2l)^2} \frac{1}{2(r-l)}.
\end{align*}
Except for the third term, it is easy to see (by direct computation) that 
the corresponding contribution to the bound~\eqref{E:c_3_Lip} on $|c^{(3)}(z) - c^{(3)}(x)|$ is bounded by a constant (depending on~$t$) times $z - x$. So we now focus on bounding the contribution from the third term. Note that since $1+r-2l > 1-t > 0$, we need only bound
\begin{equation}
\label{focus}
\int_{l, r, y}\!{\frac{1}{(z-y+r-1)(x-y+r-1)}\,\P(Y \in \ddx y \mid (L_3, 
R_3) = (l, r))}\dd l\dd r
\end{equation}
by a constant (which is allowed to depend on~$t$, but our constant will not).

We first focus on the integral in~\eqref{focus} with respect to~$y$ and write it, using a change of variables, as
\begin{equation}
\label{E:int_y}
\int_{y \in I}\!{d_{l,r}^*(x, z, y)\,f_{\frac{t-l}{r-l}}(y) \dd y},
\end{equation}
with
\[
d_{l,r}^*(x, z, y) = \frac{1}{[z-(r-l)(y+1)+r-1][x-(r-l)(y+1)+r-1]}
\] 
and
 $I := \left\{y: \frac{z-(1+r)}{r-l} -1 < y \leq \frac{x-(1+r-2l)}{r-l} 
- 1\right\}$. 
Because the support of the density $f_{\frac{t - l}{r - l}}$ is contained 
in the nonnegative real line, the integral~\eqref{E:int_y} vanishes unless the right endpoint of the interval~$I$ is positive, which is true if and only if
 \[
r < \frac{x-1+3l}{2}.
\]
So we see that the integral of~\eqref{E:int_y} over $r \in (t, 1)$ vanishes unless this upper bound on~$r$ is larger than~$t$, which is true if and only if 
\begin{equation}
\label{lbound}
l > \frac{1-x+2t}{3}.
\end{equation}
But then the integral of~\eqref{E:int_y} over $\{(l, r):0 < l < t < r < 1\}$ vanishes unless this lower bound on~$l$ is smaller than~$t$, which is 
true if and only if $x > 1 - t$; we conclude that for $x \leq 1-t$, that integral vanishes. 

So we may now suppose $x > 1 - t$, and we have seen that the integral of~\eqref{E:int_y} over $\{(l, r):0 < l < t < r < 1\}$ is bounded above by its integral over the region
\[
R := \left\{ (l, r) : \frac{1 - x + 2 t}{3} \vee 0 < l < t < r < 1 \wedge \frac{x - 1 + 3 l}{2} \right\}.
\]
Observe that on~$R$ we have
\begin{equation}
\label{on_R}
\frac{x-(1+r-2l)}{r-l} - 1 = \frac{x-1+l}{r-l} -2 > \frac{2}{3} \frac{x+t-1}{r-l} -2 > \frac{1}{2} \frac{x+t-1}{r-l} -2.
\end{equation}
Define
\[
B := \left\{ (l,r) : \frac{x+t-1}{2(r-l)} - 2 > 0 \right\}.
\]
We now split our discussion of the contribution to the integral of~\eqref{E:int_y} over $(l, r) \in R$ into two terms, corresponding to (i) $R \cap B^c$ and (ii) $R \cap B$. 
\smallskip

\par\noindent
\emph{Term} (i).\ $R \cap B^c$.
Using~\eqref{on_R}, we can bound~\eqref{E:int_y} by extending the range of integration from~$I$ to 
\[
I^* := \left\{y : \frac{x+t-1}{2(r-l)} - 2 < y \leq \frac{x-(1+r-2l)}{r-l} - 1\right\}.
\]
Making use of the inequality~\eqref{E:exp_bdd}, the integral~\eqref{E:int_y} is bounded, for any $\theta > 0$,  by
\[
\int_{y \in I^*} {\frac{1}{4(r-l)^2}\, C_{\theta} e^{-\theta y} \dd y} \leq \frac{C_{\theta}}{4 \theta(r-l)^2} \exp \left[-\frac{x+t-1}{2(r-l)} \theta + 2 \theta \right].
\]
The integral over $(l, r) \in R \cap B^c$ of~\eqref{E:int_y} is therefore 
bounded by
\begin{align}
&\frac{C_{\theta}}{4 \theta} \, e^{2 \theta} \int_{l = (1-x+2t)/3}^{t}\,\int_{r=t}^{(x-1+3l)/2} {\frac{1}{(r-l)^2} 
\exp \left[ -\frac{x+t-1}{2(r-l)} \theta \right]  \dd r  \dd l} \nonumber 
\\
&= \frac{C_{\theta}}{4 \theta} \, e^{2 \theta} \int_{l = (1-x+2t)/3}^{t} \int_{s=t-l}^{(x-1+l)/2} {\frac{1}{s^2} 
\exp \left(- \frac{x+t-1}{2} \theta s^{-1} \right) \dd s \dd l} \nonumber 
\\
&\leq \frac{C_{\theta}}{4 \theta} \, e^{2 \theta} \frac{2}{\theta (x+t-1)} \int_{l = (1-x+2t)/3}^{t} 
{\exp \left(- \frac{x+t-1}{2} \theta \frac{2}{x-1+l} \right) \dd l} \nonumber \\
\label{previous}
&\leq \frac{C_{\theta}}{2 \theta^2} \, e^{2 \theta} \frac{1}{x+t-1} e^{-\theta} \left(t  - \frac{1-x+2t}{3} \right) 
= \frac{C_{\theta}}{6 \theta^2} \, e^{\theta} < \infty.
\end{align}
\smallskip

\par\noindent
\emph{Term} (ii).\ $R \cap B$.
We can bound~\eqref{E:int_y} by the sum of the integrals of the same integrand over the intervals $I^*$ and 
\[
I' := \left\{y: 0 < y \leq \frac{x+t-1}{2(r-l)} - 2\right\}.
\]
The bound for the integral over $I^*$ is the same as the bound for the $R 
\cap B^c$ term. To bound the integral over $I'$, we first observe 
that
\[
d_{l,r}^*(x, z, y) \leq \frac{1}{[\frac{1}{2} (x-t-1) + 2r - l]^2} \leq \frac{4}{(x+t-1)^2},
\]
where the last inequality holds 
because $l < t < r$.
The contribution to~\eqref{focus} can be bounded by integrating $4 / (x + 
t - 1)^2$ with respect to $(l,r) \in R \cap B$.  We then extend this region of integration to~$R$, and thus bound the contribution by
\begin{align*}
\frac{4}{(x+t-1)^2} \int_{l = \frac{2t+1-x}{3}}^t {\left(\frac{x-1+3l}{2} - t\right) \dd l} &\leq \frac{2}{(x+t-1)} \left(t - \frac{2t+1-x}{3}\right)\\
&=2/3.
\end{align*}
This completes the proof for Subcase 1(a).
\smallskip

\par\noindent
\emph{Subcase} 1(b).\ $x -(1+r -2l) < y \leq z - (1+r-2l)$.
First note that in this subcase we have $f^{(3)}(x-y) = 0$. We proceed in similar fashion as for Subcase 1(a), this time setting
\[
I := \left\{y: \frac{x-1+l}{r-l} -2 < y \leq \frac{z-1+l}{r-l} -2\right\}.
\]
Again using a linear change of variables, the integral (with respect to~$y$ only, in this subcase) appearing on the right in~\eqref{E:c_3_Lip} in this subcase can be written as 
\begin{equation}
\label{E:int_y_2}
\int_{y \in I} {d^*_{l,r}(z, y) f_{\frac{t-l}{r-l}}(y) \dd y}
\end{equation}
where now
\[
d^*_{l,r}(z, y) = \frac{1}{z-(r-l)(y+1) +1-r} \times \frac{2}{z-(r-l)(y+1)+r-1}.
\]
Note that, unlike its analogue in Subcase 1(a), here $d^*_{l, r}(z, y)$ does not possess an explicit factor $z - x$. 

By the same discussion as in Subcase 1(a), 
we are interested in the integral of~\eqref{E:int_y_2} with respect to $(l,r) \in R$, where this time
\[
R := \left\{ (l, r) : \frac{1 - z + 2 t}{3} \vee 0 < l < t < r < 1 \wedge \frac{z - 1 + 3 l}{2} \right\}.
\]
and we may suppose that $z > 1 - t$.

Observe that on~$R$ we have 
\[
\frac{z-1+l}{r-l} - 2 > \frac{2}{3} \frac{z + t - 1}{r-l} - 2 > \frac{1}{2} \frac{z + t - 1}{r-l} - 2.
\]
Following a line of attack similar to that for Subcase 1(a), we define
\[
W := \left\{(l,r) : \frac{x-1+l}{r-l} - 2 > \frac{z + t - 1}{2(r-l)} - 2\right\}
\]
and split our discussion of the integral of~\eqref{E:int_y_2} over $(l, r) \in R$ into two terms, corresponding to (i) $R \cap W^c$ and (ii) $R \cap W$.
\smallskip

\par\noindent
\emph{Term} (i).\ $R \cap W$.
We bound~\eqref{E:int_y_2} by using the inequality~\eqref{E:exp_bdd} (for 
any $\theta > 0$) and obtain
\begin{align*}
\lefteqn{\hspace{-0.5in}\int_{y \in I} {\frac{1}{2-2l} \frac{1}{r-l} C_{\theta} 
\exp\left[ -\theta \left(\frac{z + t - 1}{2(r-l)} - 2\right) \right] \dd y}} \\ 
&\leq 
\frac{1}{2} \frac{1}{1-t} \frac{1}{(r-l)^2} C_{\theta} e^{2\theta} \exp\left[ -\theta \left(\frac{z + t - 1}{2(r-l)}\right) \right] (z-x).
\end{align*}
Integrate this terms with respect to $(l,r) \in R \cap W$, we get no more than
\[
\frac{1}{2} \frac{(z-x) }{1-t} C_{\theta} e^{2\theta} \int_{l = (1-z+2t)/3}^t \int_{r = t}^{(z-1+3l)/2} {\frac{1}{(r-l)^2} 
\exp\left[ - \frac{z + t - 1}{2(r-l)} \theta \right] \dd r \dd l},
\]
which [consult~\eqref{previous}] is bounded by $(z-x)$ times a constant depending only on~$t$ and~$\theta$.
\smallskip

\par\noindent
\emph{Term} (ii).\ $R \cap W^c$.
We partition the interval~$I$ of $y$-integration into the two subintervals 
\[
I^* := \left\{y : \frac{z + t - 1}{2(r-l)} - 2 < y \leq \frac{z-1+l}{r-l} -2\right\}
\]
and
\[
I' := \left\{y: \frac{x-1+l}{r-l} -2 < y \leq \frac{z + t - 1}{2(r-l)} - 2\right\}.
\]
Observe that  
the length of each of the intervals $I^*$ and $I'$ is no more than the length of~$I$, which is $(z - x) / (r - l)$. 
We can bound the integral over $y \in I^*$ and $(l,r) \in R \cap W^c$ just as we did for Term~(i). For the integral over 
$y \in I'$  and $(l,r) \in R \cap W^c$, observe the following 
inequality:
\begin{align*}
d^*_{l,r}(z, y)  \leq \frac{1}{2-2t} \frac{2}{\frac{1}{2} (z + t - 1) + 2r-l-t}.
\end{align*}
Using the constant bound in~\refT{T:bdd}, the integral of $d^*_{l, r}(z, y) f_{\frac{t-l}{r-l}}(y)$ with respect to $y \in I'$ and $(l,r) \in R \cap W^c$ is bounded above by
\begin{equation}
\label{10eq}
10 \frac{(z-x)}{1-t} \int_{l = (1-z+2t)/3}^t \int_{r = t}^{(z-1+3l)/2} {\frac{1}{r-l} \frac{1}{\frac{1}{2} (z + t - 1) + 2r-l-t} \dd r \dd l}.
\end{equation}
Write the integrand here in the form
\[
\frac{1}{r-l} \frac{1}{\frac{z + t - 1}{2} + 2r-l-t} = \left(\frac{1}{r-l} - \frac{2}{2r-l-t + \frac{z + t - 1}{2} }\right) \frac{1}{l-t+\frac{z 
+ t - 1}{2}},
\]
and observe that $l-t+\frac{z + t - 1}{2} > \frac{z + t - 1}{6} > 0$.  Hence we can bound~\eqref{10eq} by
\begin{align*}
&10 \frac{(z-x)}{1-t} \int_{l = (1-z+2t)/3}^t {\frac{1}{l-t+\frac{z + t 
- 1}{2}} \left[ \ln \frac{z-1+l}{2} - \ln (t-l) \right] \dd l}\\
&\leq 
10 \frac{(z-x)}{1-t} \frac{6}{z + t - 1} \left[ \frac{z + t - 1}{3} \ln \frac{z + t - 1}{2} - \int_{l = \frac{1-z+2t}{3}}^t {\ln (t-l) \dd l }\right]\\
&= 20 \frac{(z-x)}{1-t} \left(\ln \frac{z + t - 1}{2} - \ln \left(\frac{z + t - 1}{3}\right) + 1 \right)\\
&= 20 \left(1 + \ln \frac{3}{2} \right) \frac{(z-x)}{1-t}.
\end{align*}
This completes the proof for Subcase 1(b). 
\smallskip

\par\noindent
\emph{Subcase} 1(c).\ $x -(1+r) < y \leq z - (1+r)$.
In this case, the contribution from $f^{(3)}(z-y)$ vanishes. Without loss 
of generality we may suppose $z - x < t$, otherwise we can insert a factor $(z-x)/t$ in our upper bound, and the desired upper bound follows from the fact that the densities $f_{\tau}$ are all bounded by~$10$. Observe that the integrand $|d_{l, r}(x, z, y)|$ in the bound~\eqref{E:c_3_Lip} is
\[
\frac{1}{x - y + 1 -r} \frac{2}{x - y + r -1} \leq \frac{1}{x - z + 2} \frac{2}{x - z + 2r} \leq \frac{1}{2-t} \frac{2}{2r - t} 
\leq \frac{2}{t (2 - t)}.
\]
Integrate this constant bound directly with respect to 
\[
P(Y \in \ddx y | (L_3,R_3) = (l,r)) \dd r \dd l
\]
on the region $x -(1+r) < y \leq z - (1+r)$ and $0 < l < t < r < 1$ and use the fact that the density is bounded by $10$; we conclude that this contribution is bounded by $(z-x)$ times a constant that depends on $t$. This completes the proof for Subcase 1(c) and also for Case~1.
\smallskip

\par\noindent
\emph{Case}~2.\ $z -x \geq 2 l$. In this case we 
simply use
\[
| d_{l,r}(x,z, y) | \leq f^{(3)}(z-y) + f^{(3)}(x-y),
\]
and show that each of the two terms on the right contributes at most a constant (depending on~$t$) times $(z - x)$ to the bound in~\eqref{E:c_3_Lip}.  Accordingly, let~$w$ be either~$x$ or~$z$.  We are interested in bounding 
\begin{equation}
\label{E:int_y_3}
\int_{l = 0}^{\frac{z-x}{2} \wedge t} \int_{r=t}^1 \int_{y = w-(1+r)}^{w-(1+r-2l)} {\frac{1}{w-y+1-r} \frac{2}{w-y+r-1} \mu(\ddx y, \ddx r, \ddx l)}
\end{equation}
with $\mu(\ddx y, \ddx r, \ddx l) := P(Y \in \ddx y \mid (L_3,R_3) =(l,r)) \dd r \dd l$. We bound the integrand as follows:
\[
\frac{1}{w-y+1-r} \frac{2}{w-y+r-1} \leq \frac{1}{2-2l} \frac{2}{2r - 2l} 
\leq \frac{1}{2} \frac{1}{1-t} \frac{1}{r-l}.
\]
We first suppose $z-x < t$ and bound~\eqref{E:int_y_3} by
\begin{align*}
\frac{1}{2} \frac{1}{1-t} \int_{l = 0}^{\frac{z-x}{2}} \int_{r=t}^1 {\frac{1}{r-l} \dd r \dd l} 
&\leq \frac{1}{2} \frac{1}{1-t} \int_{l = 0}^{\frac{z-x}{2}} {[- \ln (t-l)] \dd l}\\
&\leq \frac{1}{2} \frac{1}{1-t} \left[ - \ln \left(t - \frac{z-x}{2}\right) \right] \frac{z-x}{2}\\
&\leq (z-x) \frac{\ln(2 / t)}{4(1-t)}.
\end{align*}
If instead $z - x \geq t$, we bound~\eqref{E:int_y_3} by
\[
\frac{1}{2} \frac{z-x}{t} \int_{l = 0}^{t} \int_{r=t}^1 {\frac{1}{1-l}\frac{1}{r-l} \dd r \dd l} \leq (z-x) \frac{\pi^2}{12\,t}.
\]
This completes the proof for Case~2 and thus the proof of Lipschitz continuity of $c^{(3)}$.
\end{proof}

We immediately get the following corollary from the proof of~\refT{T:Lipschitz_cont}. 
\begin{corollary}
\label{C:equi_continuous}
For any 
$0 < \eta < 1/2$, the uniform continuous family $\{f_t:t \in [\eta, 1-\eta]\}$ is a uniformly equicontinous family.
\end{corollary}

\begin{proof}
We observe from the proof of \refT{T:Lipschitz_cont} that for any $0 < \eta < 1/2$, the Lipschitz constants $L_t$ in~\refT{T:Lipschitz_cont} are bounded for $t \in [\eta, 1-\eta]$ by some universal constant $C < \infty$.  The result follows. 
\end{proof}

\subsection{Joint continuity}
\label{S:joint}

As noted in the proof of~\refL{L:Frc}, we reference Gr\"{u}bel and R\"{o}sler~\cite{MR1372338} to conclude that for each $t \in [0, 1)$, the distribution functions $F_u$ converge weakly to $F_t$ as $u \downarrow t$. It follows by symmetry that the convergence also holds for each $t \in (0, 1]$ as $u \uparrow t$. We now deduce the convergence from $f_u$ to $f_t$ for each $t \in (0,1)$ as $u \to t$, according to the following lemma.

\begin{lemma}
\label{L:density_converge}
For each $0 < t < 1$ we have $f_u \to f_t$ uniformly as $u \to t$.
\end{lemma}

\begin{proof}
We fix $0 < t \leq 1/2$ and choose $0< \eta < t$. By the weak convergence 
of  $F_u$ to $F_t$ as $u \to t$, the uniform boundedness of the density functions (\refT{T:bdd}), the fact that $f_t(x) \to 0$ as $x \to \pm \infty$, and the uniform equicontinuity of the family $\{f_u:u \in [\eta, 1-\eta]\}$ (\refC{C:equi_continuous}), we conclude from Boos~\cite[Lemma 1]{MR773179} (a converse to Scheff\'{e}'s theorem) that $f_u \to f_t$ uniformly as $u \to t$.
\end{proof}

\begin{remark}
The uniform equicontinuity in~\refC{C:equi_continuous} does not hold for the family $\{f_t:t \in (0, 1)\}$. Here is a proof. For the sake of contradiction, suppose to the contrary. We symmetrize $f_t(x)$ at $x=0$ for every $0 \leq t \leq 1$ to create another family of continuous densities $g_t$; that is, consider $g_t(x) := [f_t(x) + f_t(-x)]/2$.  Observe that the supposed uniform equicontinuity of the functions $f_t$ for $t \in (0,1)$ extends to the functions $g_t$. Now suppose (for each $t \in [0, 1]$) that $W(t)$ is a random variable with density $g_t$. By a simple calculation we have $W(t) \Rightarrow W(0)$, and it follows by Boos~\cite[Lemma 1]{MR773179} that $g_t(x) \to g_0(x)$ uniformly in $x$. This contradicts to the fact that $g_t(0) = 0$ for all $t \in (0,1)$ but $g_0(0) = e^{-\gamma}$.
\end{remark}

\begin{remark}
\label{R:KS_conti}
Since $(F_t)_{t \in [0, 1]}$ is weakly continuous in~$t$ and $F_t$ is atomless for $0 \leq t \leq 1$, it follows from a theorem of \Polya\ (\cite[Exercise~4.3.4]{MR1796326}) that $(F_t)_{t \in [0, 1]}$ is continuous in the sup-norm metric, \ie,\,that $(J(t))$ [or $(Z(t))$] is continuous in the Kolmogorov--Smirnov metric on distributions.
\end{remark}

\begin{corollary}
\label{C:joint_continuous}
The 
density $f_{t}(x)$ is jointly continuous in $(t,x) \in (0,1) \times \bbR$.
\end{corollary}
\begin{proof}
As $(t', x') \to (t, x) \in (0, 1) \times \bbR$, we have
\begin{align*}
\limsup |f_{t'}(x') - f_t(x)| 
&\leq \limsup |f_{t'}(x') - f_t(x')| + \limsup |f_t(x') - f_t(x)| \\
&\leq \limsup \|f_{t'} - f_t\|_{\infty} + \delta_t(|x' - x|) \\
&= 0
\end{align*}
where the sup-norm $\|f_{t'} - f_t\|_{\infty}$ tends to~$0$ as $t' \to t$ by \refL{L:density_converge} and the modulus of uniform continuity $\delta_t$ of the function $f_t$ tends to~$0$ as $x' \to x$ by \refT{T:uniform_cont}. 
\end{proof}

\begin{remark}
The positivity of $f_t(x)$ for each $0 < t < 1$ and $x > \min \{t,1-t\}$ in~\refT{T:pos} can be proved alternatively by using the integral equation~\refP{P:fint} and the joint continuity result of \refC{C:joint_continuous}. Here is the proof.

Fix (for now) $t_0, t_1, t_2 \in (0,1)$ with $t_1 > t_0 > t_2$.  We will show that $f_{t_0}(x) > 0$ for all $x > t_0$, using $t_1$ and $t_2$ in auxiliary fashion.  Since this is true for arbitrarily chosen $t_0$, invoking symmetry ($f_t \equiv f_{1 - t}$) then completes the proof.  

We certainly know that $f_{t_0}(y_0) > 0$ for some $y_0 > t_0$; choose and fix such a $y_0$.  Use~\refP{P:fint} to represent the density $f_{t_1}(x)$. 
We observe that the integrand of the integral with respect to~$l$ is positive at $l = l_1 = (t_1 - t_0) / (1 - t_0)$ and $x = y_1 = (1-l_1)(y_0+1)$. From \refC{C:joint_continuous} we conclude that the integrand is positive in a neighborhood of $l_1$ and thus $f_{t_1}(y_1) > 0$.

Further, use ~\refP{P:fint} to represent the density $f_{t_0}(x)$. We observe that the integrand of the integral with respect to~$r$ is positive at $r = r_2 = \frac{t_0}{t_1}$ and $x = y_2 = r_2 (y_1 + 1)$. From 
$f_{t_1}(y_1) > 0$ and~\refC{C:joint_continuous} we conclude that $f_{t_0}(y_2) > 0$.

Now letting $y_2 = y_0 + \epsilon_1$, we have
\[
\epsilon_1 = \left(\frac{t_0}{t_1} \frac{1-t_1}{1-t_0} - 1\right) y_0 + 
\frac{t_0}{t_1} \left(1+\frac{1-t_1}{1-t_0}\right).
\]
Observe that as $t_1 \downarrow t_0$ we have $\epsilon_1 \to 2$, while as 
$t_1 \uparrow 1$ we have $\epsilon_1 \downarrow -y_0 + t_0 < 0$.
Thus, given $\delta \in (0, 2 - t_0 + y_0)$ it is possible to choose $t_1 
\in (t_0, 1)$ such that $\epsilon_1 = - y_0 + t_0 + \delta$, \ie, $y_2 = t_0 + \delta$.  We conclude that $f_{t_0}(x)$ is positive for every $x > t_0$, as desired.
\end{remark}

\section{Left-tail behavior of the density function}
\label{S:left}

We consider the densities $f_t$ with $t \in (0, 1)$; since $f_t \equiv f_{1 - t}$ by symmetry, we may without loss of generality suppose $t \in (0, 1/2]$.  As previously noted (recall Theorems~\ref{T:pos} and~\ref{T:uniform_cont}), $f_t(x) = 0$ for all $x \leq t$ and $f_t(x) > 0$ for all $x > t$.  In this section we consider the left-tail behavior of $f_t$, by which we mean the behavior of $f_t(x)$ as $x \downarrow t$.

As a warm-up, we first show that $f_t$ has a positive right-hand derivative at~$t$ that is large when~$t$ is small.

\begin{lemma}
\label{L:RHderiv}
\ \vspace{-.1in}\\

{\rm (a)}~Fix $t \in (0, 1 / 2)$.  Then the density function $f_t$ has right-hand derivative $f_t'(t)$ at~$t$ equal to 
$c_1 / t$, where
\[
c_1 := \int_0^1\!\E [2 - w + J(w)]^{-2}\dd w \in (0.0879, 0.3750).
\] 

{\rm (b)}~Fix $t = 1 / 2$.  Then the density function $f_t$ has right-hand derivative $f_t'(t)$ at~$t$ equal to 
$2 c_1 / t = 4 c_1$.
\end{lemma}

\begin{proof}
(a)~We begin with two key observations.  First, if $L_1(t) > 0$, then $J(t) > 1 - t$.  Second, if $1 > R_1(t) > R_2(t)$, then $J(t) > 2 t$.  It follows that if $0 < z < \min\{1 - 2 t, t\}$, then, with $Y \equiv Y(t)$ as 
defined at~\eqref{E:Y},  
\begin{align*}
\lefteqn{f_t(t + z)\dd z} \\
&= \P(J(t) - t \in \ddx z) \\ 
&= \P(R_1(t) < 1,\,L_2(t) > 0,\,J(t) - t \in \ddx z) \\
&= \int\!\!\!\int_{\substack{y > x > 0:\,x + y < z, \\ x < 1 - t,\ y - x < t}}\!
\P(R_1(t) - t \in \ddx x,\,t + x - L_2(t) \in \ddx y,\,Y(t) \in \ddx z - x - y) \\
&= \int\!\!\!\int_{\substack{y > x > 0:\,x + y < z, \\ x < 1 - t,\ y - x < t}}\!
\dd x\,\frac{\ddx y}{t + x}\,\P\left( y\,J\left(\frac{y - x}{y}\right) \in \ddx z - x - y \right) \\
&= \int\!\!\!\int_{\substack{y > x > 0:\,x + y < z, \\ x < 1 - t,\ y - x < t}}\!
\dd x\,\frac{\ddx y}{t + x}\,y^{-1}\,f_{1 - \frac{x}{y}}\left( \frac{z - x - y}{y} \right) \dd z.
\end{align*}
Now make the changes of variables from~$x$ to $u = x / z$ and from~$y$ to $v = y / z$.  We then find
\begin{align*}
f_t(t + z)
&= z \int\!\!\!\int_{\substack{v > u > 0:\,u+ v < 1 \\ u < (1 - t) / z,\ v - u < t / z}}\!
(t + u z)^{-1}\,v^{-1}\,f_{1 - \frac{u}{v}}\left( \frac{1 - u - v}{v} \right)\dd u\dd v \\
&= z \int\!\!\!\int_{v > u > 0:\,u+ v < 1}\!
(t + u z)^{-1}\,v^{-1}\,f_{1 - \frac{u}{v}}\left( \frac{1 - u - v}{v} \right)\dd u\dd v,
\end{align*}
where the second equality follows because $(1 - t) / z > (1 - 2 t) / z > 1$ and $t / z > 1$ by assumption.  Thus, as desired,
\[
f_t(t + z) \sim \frac{c_1 z}{t}
\]
as $z \downarrow 0$ by the dominated convergence theorem, if we can show that
\[
\tc := \int\!\!\!\int_{v > u > 0:\,u+ v < 1}\!v^{-1}\,f_{1 - \frac{u}{v}}\left( \frac{1 - u - v}{v} \right)\dd u\dd v
\]
equals~$c_1$.  For that, make another change of variables from~$u$ to $w = u / v$; then we find
\begin{align*}
\tc 
&= \int_0^1\,\int_0^{(1 + w)^{-1}}\!f_{1 - w}(v^{-1} - (1 + w)) \dd v\dd w \\
&= \int_0^1\,\int_0^{(2 - w)^{-1}}\!f_w(v^{-1} + w - 2) \dd v\dd w.
\end{align*}
Make one last change of variables, from~$v$ to $s = v^{-1} + w - 2$, to 
conclude
\[
\tc = \int_0^1\,\int_0^{\infty}\!(2 - w + s)^{-2}\,f_w(s)\dd s\dd w = 
c_1,
\]
as claimed.

To obtain the claimed upper bound on~$c_1$, we note, using the facts that 
$J(w)$ and $J(1 - w)$ have the same distribution and that $J(w) > w$ for $w \in (0, 1 / 2)$, that
\begin{align*}
c_1 
&= \int_0^{1 / 2}\!\E [2 - w + J(w)]^{-2}\dd w + \int_{1 / 2}^1\!\E [2 - w + J(w)]^{-2}\dd w \\
&= \int_0^{1 / 2}\!\E [2 - w + J(w)]^{-2}\dd w + \int_0^{1 / 2}\!\E [1 + w + J(w)]^{-2}\dd w \\
&< \int_0^{1 / 2}\!\tfrac{1}{4}\dd w + \int_0^{1 / 2}\!(1 + 2 w)^{-2}\dd w = \tfrac{1}{8} + \tfrac{1}{4} = \tfrac{3}{8} = 0.3750.
\end{align*}

To obtain the claimed lower bound on~$c_1$, we combine Jensen's inequality with the known fact [\cf\,\eqref{EZ}] that 
$\E J(w) = 1 + 2 H(w)$ with 
$H(w) = - w \ln w - (1 - w) \ln(1 - w)$:
\begin{align*} 
c_1 
&= \int_0^1\!\E [2 - w + J(w)]^{-2}\dd w \\
&\geq \int_0^1\!(\E [2 - w + J(w)])^{-2}\dd w
= \int_0^1\!(3 - w + 2 H(w))^{-2}\dd w > 0.0879.
\end{align*}

(b)~By an argument similar to that at the start of the proof of~(a), if $0 < z < 1 / 2$, then, using symmetry at the third equality,
\begin{align*}
f_t(t + z)\dd z
&= \P(J(t) - t \in \ddx z) \\ 
&= \P(R_1(t) < 1,\,L_2(t) > 0,\,J(t) - t \in \ddx z) \\ 
&{} \quad + \P(L_1(t) > 0,\,R_2(t) < 1,\,J(t) - t \in \ddx z) \\
&= 2 \P(R_1(t) < 1,\,L_2(t) > 0,\,J(t) - t \in \ddx z) \\
&\sim \frac{2 c_1 z}{t} = 4 c_1 z.
\end{align*}
Here the asymptotic equivalence is as $z \downarrow 0$ and follows by the 
same argument as used for~(a).
\end{proof}

We are now prepared for our main result about the left-tail behavior of~$f_t$.

\begin{theorem}
\label{T:left}
\ \vspace{-.1in}\\

{\rm (a)}~Fix $t \in (0, 1 / 2)$.  Then $f_t(t + t z)$ has the uniformly absolutely convergent power series expansion
\[
f_t(t + t z) = \sum_{k = 1}^{\infty} (-1)^{k - 1} c_k z^k
\]
for $z \in [0, \min\{t^{-1} - 2, 1\})$, where for $k \geq 1$ the coefficients
\[
c_k := \int_0^1\!(1 - w)^{k - 1} \E [2 - w + J(w)]^{-(k + 1)}\dd w,
\] 
not depending on~$t$, are strictly positive, have the property that $2^k c_k$ is strictly decreasing in~$k$, and satisfy
\[
0 < (0.0007) 2^{- (k + 1)} (k + 1)^{-2} < c_k < 2^{- (k + 1)} k^{-1} (1 + 
2^{-k}) < 0.375 < \infty.
\]
{\rm [}In particular, $2^k c_k$ is both $O(k^{-1})$ and $\Omega(k^{-2})$.{\rm ]}

{\rm (b)}~Fix $t = 1 / 2$.  Then $f_t(t + t z)$ has the uniformly absolutely convergent power series expansion
\[
f_t(t + t z) = 2 \sum_{k = 1}^{\infty} (-1)^{k - 1} c_k z^k
\]
for $z \in [0, 1)$.
\end{theorem}

\begin{proof}
(a)~As shown in the proof of \refL{L:RHderiv}, for $z \in [0, \min\{t^{-1} - 2, 1\})$ we have
\begin{align}
\label{ft_integral}
f_t(t + t z)
&= z \int\!\!\!\int_{v > u > 0:\,u+ v < 1}\!(1 + u z)^{-1}\,v^{-1}\,f_{1 - \frac{u}{v}}\left( \frac{1 - u - v}{v} \right)\dd u\dd v.
\end{align}
Note that the expression on the right here doesn't depend on~$t$.  Further, since 
$z \leq 1$ and $0 < u < 1/2$ in the range of integration,
\begin{align*}
\lefteqn{\hspace{-.5in}\frac{1}{2} \int\!\!\!\int_{v > u > 0:\,u+ v < 1}\!(1 - u z)^{-1}\,v^{-1}\,
f_{1 - \frac{u}{v}}\left( \frac{1 - u - v}{v} \right)\dd u\dd v} \\
&< \int\!\!\!\int_{v > u > 0:\,u+ v < 1}\!v^{-1}\,f_{1 - \frac{u}{v}}\left( \frac{1 - u - v}{v} \right)\dd u\dd v \\
&= \tc = c_1 < 3 / 8 < \infty,
\end{align*}
with $\tc$ and $c_1$ as in the proof of \refL{L:RHderiv}.  It follows that $f_t(t + t z)$ has the uniformly absolutely convergent power series expansion
\[
f_t(t + t z) = \sum_{k = 1}^{\infty} (-1)^{k - 1} c_k z^k
\]
for $z \in [0, \min\{t^{-1} - 2, 1\})$, where for $k \geq 1$ we have 
\begin{align*}
c_k 
&= 2 \times 2^{- k} 
\int\!\!\!\int_{v > u > 0:\,u+ v < 1}\!(2 u)^{k - 1}\,v^{-1}\,f_{1 - \frac{u}{v}}\left( \frac{1 - u - v}{v} \right)\dd u\dd v \\
&= \int_0^1\!(1 - w)^{k - 1} \E [2 - w + J(w)]^{-(k + 1)}\dd w;
\end{align*}
the second equality follows just as for $c = c_1$ in the proof of \refL{L:RHderiv}.  From the first equality it is clear that these coefficients 
have the property that $2^k c_k$ is strictly decreasing in~$k$.

To obtain the claimed upper bound on~$c_k$, proceed just as in the proof of \refL{L:RHderiv} to obtain
\begin{align*}
c_k 
&< 2^{- (k + 1)} \int_0^{1 / 2}\!(1 - w)^{k - 1}\dd w + \int_0^{1 / 2}\!w^{k - 1} (1 + 2 w)^{- (k + 1)}\dd w \\ 
&= 2^{- (k + 1)} k^{-1} (1 - 2^{-k}) + k^{-1} 4^{-k} = 2^{- (k + 1)} k^{-1} (1 + 2^{-k}).
\end{align*}

The claimed lower bound on~$c_k$ follows from \refL{L:RHderiv} for $k = 
1$ but for $k \geq 2$ requires more work. We begin by establishing a lower bound on $\P(J(w) \leq 2 w)$ for $w \leq 
1/3$, using what we have already proved:
\begin{align*}
\P(J(w) \leq 2 w)
&= \int_0^w f_w(w + x) \dd x
= w \int_0^1 f_w(w + w z) \dd z \\
&\geq w \int_0^1 (c_1 z - c_2 z^2) \dd z 
= w (\tfrac{1}{2} c_1 - \tfrac{1}{3} c_2) \\
&> [0.04395 - (1/3) (1/8) (1/2) (5/4)] w 
> 0.0179\,w.  
\end{align*}
Thus $c_k$ is at least $0.0179\ 2^{- (k + 1)}$ times the following expression:
\begin{align*}
\lefteqn{2^{k + 1} \int_0^{1 / 3}\!w (1 - w)^{k - 1} (2 + w)^{- (k + 1)} \dd w} \\
&\geq \int_0^{1 / 3}\!w \exp[-2 (k + 1) w]  \exp[- (k + 1) w / 2] \dd w \\
&\geq \int_0^{1 / (k + 1)}\!w \exp[- 5 (k + 1) w / 2] \dd w \\
&\geq e^{-5 / 2} \int_0^{1 / (k + 1)}\!w \dd w = \frac{1}{2} e^{-5 / 2} 
(k + 1)^{-2}.  
\end{align*}

(b) The claim of part~(b) is clear from the proof of \refL{L:RHderiv}.
\end{proof}

\begin{corollary}
\label{C:nice}
\ \vspace{-.1in}\\

{\rm (a)}~Fix $t \in (0, 1 / 2)$.  Then, for all $x \in (t, \min\{1 - t, 2 t\})$, the density $f_t(x)$ is infinitely differentiable, strictly increasing, strictly concave, and strictly log-concave.

{\rm (b)}~Fix $t = 1 / 2$.  Then for all $x \in [1 / 2, 1)$, the density $f_{1/2}(x)$ is infinitely differentiable, strictly increasing, strictly concave, and strictly log-concave.
\end{corollary}

\begin{proof}
Once again it is clear that we need only prove~(a).  The result is actually a corollary to~\eqref{ft_integral}, rather than to \refT{T:left}.  It is easy to justify repeated differentiation with respect to~$z$ under the 
double integral of~\eqref{ft_integral}.  In particular, 
for $z \in (0, \min\{t^{-1} - 2, 1\})$ we have
\begin{align*}
t f_t'(t + t z)
&= \int\!\!\!\int_{v > u > 0:\,u+ v < 1}\!(1 + u z)^{-1}\,v^{-1}\,f_{1 - \frac{u}{v}}\left( \frac{1 - u - v}{v} \right)\dd u\dd v \\
&{} \quad 
- z \int\!\!\!\int_{v > u > 0:\,u+ v < 1}\!u (1 + u z)^{-2}\,v^{-1}\,f_{1 
- \frac{u}{v}}\left( \frac{1 - u - v}{v} \right)\dd u\dd v \\
&= \int\!\!\!\int_{v > u > 0:\,u+ v < 1}\!(1 + u z)^{-2}\,v^{-1}\,f_{1 - \frac{u}{v}}\left( \frac{1 - u - v}{v} \right)\dd u\dd v > 0
\end{align*}
and
\begin{align*}
t^2 f_t''(t + t z)
&= -2 \int\!\!\!\int_{v > u > 0:\,u+ v < 1}\!u (1 + u z)^{-3}\,v^{-1}\,f_{1 - \frac{u}{v}}\left( \frac{1 - u - v}{v} \right)\dd u\dd v \\
&< 0.
\end{align*}
Strict log-concavity of the positive function $f_t$ follows immediately from strict concavity.
\end{proof}

\begin{remark}
(a)~By
extending the computations of the first and second derivatives of $f_t$ in the proof of \refC{C:nice} to higher-order derivatives, it is easy to see that $f_t(x)$ is real-analytic
 for~$x$ in the intervals as specified in \refC{C:nice}(a)--(b). For 
  the definition of real analytic function, see Krantz and Parks~\cite[Definition 1.1.5]{MR1916029}.

(b)~It may be that, like the Dickman density $f_0$, the densities $f_t$ with $0 < t < 1$ are log-concave everywhere and hence strongly unimodal.  Even if this is false, we conjecture that the densities $f_t$ are all unimodal.
\end{remark} 

\section{Improved right-tail asymptotic upper bound}
\label{S:improved}

In this section, we will prove that for $0 < t < 1$ and $x > 4$, the continuous density function $f_t$ satisfies
\[
f_t(x) \leq \exp [-x \ln x -x \ln \ln x + O(x)]
\]
uniformly in~$t$.
We first bound the moment generating function of the random variable $V$ treated in~\refL{L:Stochastic_upper_bdd}.

\begin{lemma}
Denote the moment generating function of $V$ by~$m$. Then for every $\epsilon > 0$ there exists a constant $a \equiv a(\epsilon) > 0$ such that for all $\theta > 0$ we have
\begin{equation}
\label{E:bdd_mgf_V}
m(\theta) \leq \exp [(2+\epsilon) \theta^{-1} e^\theta + a \theta].
\end{equation} 
\end{lemma}

\begin{proof}
The idea of the proof comes from Janson~\cite[Lemma 6.1]{janson2015tails}. Observe that the random variable $V$ satisfies the following distributional identity
\[
V \overset{\mathcal{L}}{=} 1 + V_1 \cdot V
\]
where $V_1 \sim \mbox{Uniform}(1 / 2, 1)$ is independent of $V$. It follows by conditioning on $V_1$ that the moment generating function~$m$ satisfies
\begin{equation}
\label{E:int_eq_mgf_V}
m(\theta) = 2 e^\theta \int_{v = 1/2}^1 {m(\theta v) \dd v} = 2 e^\theta \int_{u = 0}^{1/2} {m(\theta (1-u)) \dd u}.
\end{equation}
Since~$m$ is continuous and $m(0) = 1$, there exists a $\theta_1 > 0$ such that the inequality~\eqref{E:bdd_mgf_V} holds (for \emph{any} constant $a > 0$) for $\theta \in [0,\theta_1]$. Choose and fix $\theta_2 > \max\{\theta_1, 5\}$ and choose $a \in [1, \infty)$ large enough such that the inequality~\eqref{E:bdd_mgf_V} holds for $\theta \in [\theta_1, \theta_2]$. 

We now suppose for the sake of contradiction that~\eqref{E:bdd_mgf_V} fails at some $\theta > \theta_2$. Define 
$T := \inf \{\theta > \theta_2:\mbox{\eqref{E:bdd_mgf_V} fails}\}$; then by continuity we have $m(T) =  \exp [(2+\epsilon) T^{-1} e^T + a T]$. 

Since $m(\theta u) \geq 1$ for any $\theta > 0$ and $0 < u < 1/2$, we can 
conclude from~\eqref{E:int_eq_mgf_V} that $m$ satisfies
\begin{equation}
\label{E:int_ineq_mgf_V}
m(\theta) \leq 2 e^{\theta} \int_{u = 0}^{1/2} {m(\theta u)\,m(\theta (1-u)) \dd u}
\end{equation}
for every $\theta > 0$, including for $\theta = T$.
The proof is now completed effortlessly by applying exactly the same argument as for the limiting {\tt QuickSort} moment generating function in Fill and Hung~\cite[proof of Lemma~2.1]{MR3909444}; 
indeed, using only~\eqref{E:int_ineq_mgf_V} they prove that when $\theta = T$ the right-hand side of~\eqref{E:int_ineq_mgf_V} is strictly smaller than $m(T)$, which is the desired contradiction.
\end{proof}

Thus, for $\epsilon > 0$ and $\theta > 0$, the moment generating functions $m_t$ all satisfy
\begin{equation}
\label{mtbound}
m_t(\theta) \leq m(\theta) \leq \exp[(2 + \epsilon) \theta^{-1} e^{\theta} + a \theta].
\end{equation}

We now deduce a uniform right-tail upper bound on the survival functions $1 - F_t$ for $0 < t < 1$.

\begin{theorem}
\label{L:improved_rt_bdd_distribution}
Uniformly in $0 < t < 1$, for $x > 1$ the distribution function $F_t$ for 
$J(t)$ satisfies
\[
1 - F_t(x) \leq \exp [-x \ln x - x \ln \ln x + O(x)].
\]
\end{theorem}

\begin{proof}
The proof is essentially the same as for Fill and Hung~\cite[proof of Proposition~1.1]{MR3909444}, but for completeness we sketch the simple proof 
here. 
Fix $\epsilon >0$.  For any $\theta > 0$ we have the Chernoff bound
\begin{align*}
1 - F_t(x) = \P(J(t) > x) &\leq \P(Z(t) > x) \\
&\leq e^{-\theta x} m_t(\theta) \leq e^{-\theta x} \exp[(2 + \epsilon) \theta^{-1} e^{\theta} + a \theta]
\end{align*}
by~\eqref{mtbound}.
Letting $\theta = \ln [(2+\epsilon)^{-1} x \ln x]$, and then $\epsilon \downarrow 0$ 
we get the desired upper bound---in fact, with the following improvement we will not find useful in the sequel:
\[
1 - F_t(x) \leq \exp [-x \ln x - x \ln \ln x + (1 + \ln 2) x + o(x)].
\]
\end{proof}

The continuous density function $f_t(x)$ enjoys the same uniform asymptotic bound for $0 < t < 1$ and $x > 4$. 

\begin{theorem}
\label{T:improved_rt_bdd_density}
Uniformly in $0 < t < 1$, for $x > 4$ the continuous density function $f_t$ satisfies
\[
f_t(x) \leq  \exp [-x \ln x - x \ln \ln x + O(x)].
\]
\end{theorem}

\begin{proof}
Fix $0 < t < 1$ and let $x > 4$. We first use the integral equation in~\refP{Integral equation}, namely,
\[
f_t(x) = \int{\P((L_3(t), R_3(t)) \in \ddx (l, r))} \cdot h_t(x \mid l, 
r),
\]
for $x \geq 0$, where, by a change of variables,
\begin{equation}
\label{htxlr}
h_t(x \mid l, r) = \int{f_{l,r}((r-l)(y-1)) \, f_{\frac{t - l}{r - l}}\!\left(\frac{x}{r - l} - y\right)\dd y};
\end{equation}
we consider the contribution to $f_t(x)$ from values $(l, r)$ satisfying $0 < l < t < r < 1$. 
Recall that the conditional density $f_{l, r}(z)$ vanishes if $z \geq 2$. 
 Thus the only nonzero contribution to~\eqref{htxlr} is from values of~$y$ satisfying
\[
y \leq \frac{2}{r-l} + 1.
\]
If this inequality holds, then the argument for the factor $f_{(t - l) / (r - l)}$ satisfies
\[
\frac{x}{r-l} - y \geq \frac{x-2}{r-l} - 1 \geq x - 3.
\]
Using $b(l, r)$ of \refL{L:bound} and~\eqref{E:bdd_rho_finite} to bound the $f_{l,r}$ factor, we obtain
\[
h_t(x \mid l, r) \leq b(l, r) \left(1 - F_{\frac{t-l}{r-l}} (x-3)\right).
\]
By \refT{L:improved_rt_bdd_distribution} and the last display in the proof of \refL{L:bound}, the contribution in question is thus bounded by $\exp[ - x \ln x - x \ln \ln x + O(x)]$, uniformly in~$t$, for $x > 4$.

For the contribution to $f_t(x)$ corresponding to the cases 
$0 = L_3(t) < t < R_3(t) < 1$ and $0 < L_3(t) < t < R_3(t) = 1$, we use the same idea as in the proof of~\refL{L:bdd_rho_0}. By symmetry, we need only consider the first of these two cases. Recall from the proof of~\refL{L:bdd_rho_0} that the contribution in question is bounded by the sum of $f_W(x)$, which is the density of $W = U_1 (1 + U_2 V)$ evaluated at $x$ [where $U_1$, $U_2$, and $V$ are independent, $U_1$ and $U_2$ are uniformly distributed on $(0, 1)$, and~$V$ is as in \refL{L:Stochastic_upper_bdd}], and the integral
\begin{align*}
\int_{r = 0}^1\!{r^{-1} \P\left( V > \frac{x}{r} - 1 \right) \dd r} &= 
\int_{v = x-1}^{\infty} {(v+1)^{-1} \P (V > v) \dd v}\\
&\leq x^{-1} \int_{v = x-1}^{\infty} {\P (V > v) \dd v} \\
&\leq \exp [-x \ln x - x \ln \ln x + O(x)].
\end{align*}
The last inequality here is obtained by applying a Chernoff bound and \refL{E:bdd_mgf_V} to the integrand and integrating; we omit the straightforward details.
To bound the density of $W$ at~$x$, observe that by conditioning on the values of $U_2$ and $V$, we have
\begin{align*}
f_W(x) &= \int_{u,v} (1 + u v)^{-1}\,\mathbb{1}(0 \leq x \leq 1+uv)\,\P(U_2 \in \ddx u,\,V \in \ddx v)\\
&= \int_{u = 0}^1 \int_{v = (x-1)/u}^\infty (1 + u v)^{-1}\,\P(V \in \ddx v) \dd u\\
&\leq x^{-1} \int_{u = 0}^1\!\P\!\left(V > \frac{x-1}{u} \right) \dd u\\
&\leq x^{-1} \P(V > x-1) \leq \exp [-x \ln x - x \ln \ln x + O(x)]. 
\end{align*}
This completes the proof.
\end{proof}

\section{Matching right-tail asymptotic lower bound}
\label{S:lower}
In this section we will prove for each fixed $t \in (0, 1)$ that the continuous density function $f_t$ satisfies
\[
f_t(x) \geq \exp[- x \ln x - x \ln \ln x + O(x)]\mbox{\ as $x \to \infty$},
\]
matching the upper bound
of \refT{L:improved_rt_bdd_distribution} to two logarithmic asymptotic terms, with remainder of the same order of magnitude.  While we are able to 
get a similarly matching lower bound to \refT{T:improved_rt_bdd_density} for the survival function $1 - F_t$ that is uniform in~$t$, we are unable 
to prove uniformity in~$t$ for the density lower bound.

We begin with consideration of the survival function.

\begin{theorem}
\label{T:improved_rt_bdd_distribution_lower}
Uniformly in $0 < t < 1$, the distribution function $F_t$ for $J(t)$ satisfies
\[
1 - F_t(x) \geq \exp[- x \ln x - x \ln \ln x + O(x)].
\]
\end{theorem}

\begin{proof}
With~$D$ denoting a random variable having the Dickman distribution with support $[1, \infty)$, for  
any $0 < t < 1$ we have from \refL{L:Stochastic_lower_bdd} that
\begin{align*}
1 - F_t(x) 
&= \P(J(t) > x) = \P(Z(t) > x + 1) \geq \P(D > x + 1) \\ 
&= \exp[- x \ln x - x \ln \ln x + O(x)]\mbox{\ as $x \to \infty$}.
\end{align*}
The asymptotic lower bound here follows by substitution of $x + 1$ for~$u$ in equation~(1.6) (for the unnormalized Dickman function) 
of Xuan~\cite{MR1245402}, who credits earlier work of de Bruijn~\cite{MR43838} and of Hua~\cite{Hua1951integral}.   
\end{proof}

Now we turn our attention to the densities.

\begin{theorem}
\label{T:improved_rt_bdd_density_lower}
For each fixed $t \in (0, 1)$ we have
\[
f_t(x) \geq \exp[- x \ln x - x \ln \ln x + O(x)]\mbox{\ as $x \to \infty$}.
\] 
\end{theorem}

\begin{proof}
From the calculations at the beginning of the proof of \refL{L:RHderiv}, for all $z > 0$ we have
\[
f_t(t + z) 
\geq
z \int\!\!\!\int_{\substack{v > u > 0:\,u+ v < 1, \\ u < (1 - t) / z,\ v - u < t / z}}\!
(t + u z)^{-1}\,v^{-1}\,f_{1 - \frac{u}{v}}\left( \frac{1 - u - v}{v} \right) \dd u \dd v.
\]
Thus, changing variables from~$u$ to $w = 1 - (u / v)$, we have
\[
f_t(t + t z) 
\geq
z \int_0^1\!\int_0^{\Upsilon(t, z, w)}\![1 + v (1 - w) z]^{-1}\,f_w\left( 
v^{-1} + w - 2 \right) \dd v \dd w,
\]
where $\Upsilon(t, z, w) := {\min\{(2 - w)^{-1},\ (1 - t) (t z)^{-1} (1 
- w)^{-1},\ z^{-1} w^{-1}\}}$.
Now let
\[ 
\Lambda(t, z, w) := \max\{0, t (1 - t)^{-1} z (1 - w) + w - 2, z w + w - 2\}
\]
and change variables from~$v$ to $s = v^{-1} + w - 2$ to find
\begin{align*}
\lefteqn{f_t(t + t z)} \\ 
&\geq
z \int_0^1\!\int_{\Lambda(t, z, w)}^{\infty}\![1 + (2 - w + s)^{-1} (1 - w) z]^{-1}\,(2 - w + s)^{-2} f_w(s) \dd s \dd w.
\end{align*}

Observe that if $\delta > 0$ and $t \leq w \leq (1 + \delta) t \leq 1$, then
\[
\Lambda(t, z, w) < (1 + \delta) t z
\]
and so
\begin{align*}
\lefteqn{f_t(t + t z)} \\ 
&\geq
z \int_t^{(1 + \delta) t}\!\int_{(1 + \delta) t z}^{\infty}\![1 + (2 - w + s)^{-1} (1 - w) z]^{-1}\,(2 - w + s)^{-2} f_w(s) \dd s \dd w.
\end{align*}
If $\delta \leq 1$, it follows 
that
\begin{align*}
\lefteqn{f_t(t + t z)} \\ 
&\geq
z \int_t^{(1 + \delta) t}\!\int_{(1 + \delta) t z}^{2 t z} 
[1 + (2 - w + s)^{-1} (1 - w) z]^{-1}\,(2 - w + s)^{-2} f_w(s) \dd s \dd w \\
&\geq \frac{z}{(2 + 2 t z)^2} \int_t^{(1 + \delta) t}\!\int_{(1 + \delta) 
t z}^{2  t z} 
\frac{1}{1 + (2 - w + (1 + \delta) t z)^{-1} (1 - w) z}\,f_w(s) \dd s \dd 
w \\
&\geq \frac{z}{(2 + 2 t z)^2} \int_t^{(1 + \delta) t}\!\int_{(1 + \delta) 
t z}^{2 t z}
\left[ 1 + \frac{1 - w}{(1 + \delta) t} \right]^{-1}\,f_w(s) \dd s \dd w \\
&\geq \frac{z}{(2 + 2 t z)^2} \frac{(1 + \delta) t}{1 + \delta t} \int_t^{(1 + \delta) t}\!\int_{(1 + \delta) t z}^{2 t z} 
f_w(s) \dd s \dd w \\
&\geq \frac{t z}{(2 + 2 t z)^2} \int_t^{(1 + \delta) t}\!\int_{(1 + \delta) t z}^{2 t z} f_w(s) \dd s \dd w \\
&= \frac{t z}{(2 + 2 t z)^2} \int_t^{(1 + \delta) t}\big[ \P(J(w) > (1 + \delta) t z) - \P(J(w) > 2 t z) \big] \dd w.
\end{align*}

Recall that~$D$ defined in~\refL{L:Stochastic_lower_bdd} is a random variable having the Dickman distribution on 
$[1, \infty)$ and that~$V$ is defined in~\eqref{E:V}. By~\refL{L:Stochastic_lower_bdd}, we have $D - 1 \leq J(w) \leq V - 1$ stochastically, and thus we can further lower-bound the density function as follows:
\begin{align*}
f_t(t + t z) &\geq 
\frac{t z}{(2 + 2 t z)^2} \int_t^{(1 + \delta) t}\big[ \P(D - 1 > (1 + \delta) t z) - \P(V > 2 t z) \big] \dd w \\
&= \delta t \frac{t z}{(2 + 2 t z)^2} \big[ \P(D - 1 > (1 + \delta) t z) - \P(V > 2 t z) \big].
\end{align*}
That is, if $0 < \delta \leq \min\{1, t^{-1} - 1\}$, then for every $z > 0$ we 
have
\[
f_t(t + z) \geq \delta t \frac{z}{(2 + 2 z)^2} \big[ \P(D - 1 > (1 + \delta) z) - \P(V > 2 z) \big].
\]

If  $z \geq \max\{1, t / (1 - t)\}$, then we can choose $\delta \equiv \delta_z = z^{-1}$ and conclude
\[ 
f_t(t + z) \geq t (2 + 2 z)^{-2} \big[ \P(D - 1 > z + 1) - \P(V > 2 z) \big].
\]
Moreover, as $z \to \infty$, we have
\[
(2 + 2 z)^{-2} \big[\P(D - 1 > z + 1) - \P(V > 2 z) \big] = \exp[- z \ln z - z \ln \ln z + O(z)]. 
\]
The stated result follows readily.
\end{proof}

\begin{remark}
The proof of \refT{T:improved_rt_bdd_distribution_lower} reveals that the 
result in fact holds uniformly for~$t$ in any closed subinterval of $(0, 1)$.  In fact, the proof shows that the result follows uniformly in $t \in (0, 1)$ and $x \to \infty$ satisfying
$x = \Omega(\ln[1 / \min\{t, 1 - t\}])$.
\end{remark}

\section{Right-tail large deviation behavior of {\tt QuickQuant}$(n, t)$}
\label{S:Large_deviation}

In 
this section, we investigate the right-tail large deviation behavior of {\tt QuickQuant}$(n, t)$, that is, of 
{\tt QuickSelect}$(n,m_n(t))$. Throughout this section, for each fixed $0 
\leq t \leq 1$ we consider any sequence $1 \leq m_n(t) \leq n$ such that $m_n(t)/n \to t$ as $n \to \infty$. We abbreviate the normalized number of key comparisons of {\tt QuickSelect}$(n,m_n(t))$ discussed in ~\refS{S:intro} as $C_n(t) := n^{-1} C_{n, m_n(t)}$. 

Kodaj and M\'ori~\cite[Corollary 3.1]{MR1454110} bound the convergence rate of $C_n(t)$ to its limit $Z(t)$ in the Wasserstein $d_1$-metric, showing that the distance is $O (\delta_{n, t} \log (\delta_{n, t}^{-1}))$, where 
$\delta_{n, t} = | n^{-1} m_n(t) - t | + n^{-1}$. Using their result, we bound the convergence rate in Kolmogorov--Smirnov distance in the following lemma
.
\begin{lemma}
\label{L:KS_distance}
Let $d_{\rm KS}( \cdot , \cdot )$ be Kolmogorov--Smirnov {\rm (KS)} distance. Then
\begin{equation}
\label{E:KS_distance}
d_{\rm KS}(C_n(t), Z(t)) = \exp \left[-\frac{1}{2} \ln \frac{1}{\delta_{n, t}} + \frac{1}{2} \ln \ln \frac{1}{\delta_{n, t}} + O(1)\right].
\end{equation}
\end{lemma}
\begin{proof}
The lemma is an immediate consequence of Fill and Janson~\cite[Lemma 5.1]{MR1932675}, since the random variable $Z(t)$ has a density function bounded by~$10$, according to~\refT{T:bdd}. Indeed, by that result we have
\[
d_{\rm KS}(C_n(t), Z(t)) \leq 2^{1/2} [10 \, d_1(C_n(t), Z(t))]^{1/2} = 
O([\delta_{n, t} \log (\delta_{n, t}^{-1})]^{1/2}).
\]
\end{proof}

Using the right-tail asymptotic bounds on the limiting {\tt QuickQuant}$(t)$ distribution function $F_t$ in Theorems~\ref{L:improved_rt_bdd_distribution} and~\ref{T:improved_rt_bdd_distribution_lower} (which extend to $t \in \{0, 1\}$ by known results about the Dickman distribution), we can now derive the right-tail large-deviation behavior of 
$C_n(t)$.

\begin{theorem}
\label{T:Large_deviation}
Fix $t \in [0, 1]$ and abbreviate $\delta_{n, t}$ as $\delta_n$.  Let $(\omega_n)$ be any sequence diverging to $+\infty$ as $n \to \infty$ and let $c > 1$. For integer $n \geq 3$, consider the interval 
\[
I_n := \left[c, \frac{1}{2} \frac{\ln \delta_n^{-1}}{\ln \ln \delta_n^{-1}} \left(1-\frac{\omega_n}{\ln \ln \delta_n^{-1}}\right)\right].
\]
\vspace{.01in}
\par\noindent
{\rm (a)}~Uniformly for $x \in I_n$ we have
\begin{equation}
\label{E:LD1}
\P(C_n(t) > x) = (1 + o(1)) \mathbb{P}(Z(t) > x) \quad \mbox{as $n \to \infty$}.
\end{equation}
{\rm (b)}~If $x_n \in I_n$ for all large~$n$, then
\begin{equation}
\label{E:LD2} 
\mathbb{P}(C_n(t) > x_n) = \exp[- x_n \ln x_n - x_n \ln \ln x_n + O(x_n)].
\end{equation}
\end{theorem}
\begin{proof}
The proof is similar to that of Fill and Hung~\cite[Theorem~3.3]{MR3909444} or its improvement in~\cite[Theorem 3.3]{MR3978217}. We prove part~(a) 
first. By~\refL{L:KS_distance}, it suffices to show that
\[
\exp \left[-\frac{1}{2} \ln \frac{1}{\delta_n} + \frac{1}{2} \ln \ln \frac{1}{\delta_n} + O(1)\right] \leq o(\P (Z(t) > x_n))
\]
with $x_n \equiv \frac{1}{2} \frac{\ln \delta_n^{-1}}{\ln \ln \delta_n^{-1}} \left(1-\frac{\omega_n}{\ln \ln \delta_n^{-1}}\right)$ and $\omega_n = o(\ln \ln \delta_n^{-1})$. Since, by~\refT{T:improved_rt_bdd_distribution_lower}, we have
\[
\P(Z(t) > x_n) \geq \exp [-x_n \ln x_n - x_n \ln \ln x_n +O(x_n)],
\]
it suffice to show that for any constant $C < \infty$ we 
have
\begin{equation}
\label{minus_infinity}
-\frac{1}{2} \ln \frac{1}{\delta_n} + \frac{1}{2} \ln \ln \frac{1}{\delta_n} + C + x_n \ln x_n + x_n \ln \ln x_n + C x_n \to -\infty.
\end{equation}
This is routine and similar to what is done in \cite[proof of Theorem~3.3]{MR3909444}.
This completes the proof of part~(a).

Part~(b) is immediate from part~(a) and Theorems~\ref{L:improved_rt_bdd_distribution} and~\ref{T:improved_rt_bdd_distribution_lower}.
\end{proof}

\begin{remark}
Consider the particular choice $m_n(t) = \lfloor nt \rfloor + 1$ (for $t \in [0, 1)$, with $m_n(1) = n$) of the sequences 
$(m_n(t))$.  That is, suppose that $C_n(t) = X_n(t)$ as defined in~\eqref{E:2}.  In this case, large-deviation upper bounds based on tail estimates of the limiting $F_t$ have broader applicability than as described in 
\refT{T:Large_deviation} and are easier to derive, too.  The reason is that, by Kodaj and M\'ori~\cite[Lemma 2.4]{MR1454110}, the random variable $X_n(t)$ is stochastically dominated by its continuous counterpart $Z(t)$. Then, by \refT{T:improved_rt_bdd_distribution_lower}, uniformly in $t 
\in [0, 1]$, we have
\begin{equation}
\label{VLD}
\P(X_n(t) > x) \leq \P(Z(t) > x) \leq \exp [-x \ln x - x \ln \ln x + O(x)] 
\end{equation}
for $x > 1$; there is \emph{no restriction at all} on how large $x$ can be in terms of~$n$ or~$t$.

Here is an example of a \emph{very} large value of~$x$ for which the tail probability is nonzero and the aforementioned bound still matches logarithmic 
asymptotics to lead order of magnitude, albeit not to lead-order term.   
The largest possible value for the number $C_{n, m}$ of comparisons needed by {\tt QuickSelect}$(n, m)$ is ${n \choose 2}$, corresponding in the natural coupling to any permutation of the~$n$ keys for which the $m - 1$ keys smaller than the target key appear in increasing order, the $n - m$ keys larger than the target key appear in decreasing order, and the target key appears last; 
thus
\[
\P\left( C_{n, m} = {n \choose 2} \right)
= \frac{1}{n!} {n - 1 \choose m - 1},
\] 
which lies between $1 / n!$ and ${n - 1 \choose \lceil (n - 1) / 2 \rceil} / n! \sim 2^{ n - (1 / 2)} / (n! \sqrt{\pi n})$.  We conclude that for $x_n = (n - 1) / 2$ we have, uniformly in $t \in [0, 1]$, that
\[  
P\left( X_n(t) \geq x_n \right) = \P\left( X_n(t) = x_n \right) = \exp[- 2 x_n \ln x_n + O(x_n)]. 
\] 
The bound~\eqref{VLD} on $\mathbb{P}(X_n(t) > x)$ is in fact also (by the 
same proof) a bound on the larger probability 
$\mathbb{P}(X_n(t) \geq x)$, and in this case implies
\[  
P\left( X_n(t) \geq x_n \right) = \exp[- x_n \ln x_n + O(x_n \log \log x_n)]. 
\]
The bound~\eqref{VLD} is thus loose only by an asymptotic factor of~$2$ in the logarithm of the tail probability. 
\end{remark}

\begin{remark}
(a)~We can use another result of Kodaj and M\'ori, namely, \cite[Lemma 3.2]{MR1454110}, in similar fashion to quantify the Kolmogorov--Smirnov continuity of the process~$Z$ discussed in~\refR{R:KS_conti}. 
Let  
$0 \leq t < u \leq 1/2$ and $\delta = u - t$. Then the lemma asserts
\[
d_1(Z(t), Z(u)) < 4 \delta (1 + 2 \log \delta^{-1}).
\]
It follows using Fill and Janson~\cite[Lemma 5.1]{MR1932675} that
\[
d_{\rm KS}(Z(t), Z(u)) \leq O((\delta \log \delta^{-1})^{1/2}) = \exp \left[-\tfrac{1}{2} \ln \delta^{-1} + \tfrac{1}{2} \ln \ln \delta^{-1} + O(1) \right],
\]
uniformly for $|u - t| \leq \delta$, as $\delta \downarrow 0$.  We thus have \emph{uniform} Kolmogorov--Smirnov continuity of~$Z$.  

(b)~Kodaj and M\'{o}ri~\cite{MR1454110}
did not consider a lower bound on either of the distances in~(a), but we can rather easily obtain a lower bound on the KS distance that is of order $\delta^2$ uniformly for $t$ and~$u$ satisfying $0 < t < t + \delta = 
u \leq \min\{1 / 2, 2 t\}$.

Indeed, for such~$t$ and~$u$ we have $\P(J(u) \leq u) = 0$ and, by \refT{T:left} (since $t \leq u \leq 1 / 2 \leq \min\{1 - t, 2 t\}$, as required by the hypotheses of the theorem) and in the notation of that theorem,
\begin{align*}
\P(J(t) \leq u) 
&= \int_t^u\!f_t(x) \dd x 
= t \int_0^{(u / t) - 1} \sum_{k = 1}^{\infty} (-1)^{k - 1} c_k z^k\,\dd z \\
&\geq t \int_0^{(u / t) - 1} (c_1 z - c_2 z^2)\,\dd z \\ 
&= t \left[ \frac{1}{2} c_1 \left( \frac{u}{t} - 1 \right)^2 - \frac{1}{3} c_2 \left( \frac{u}{t} - 1 \right)^3 \right] \\
&\geq \frac{1}{3} c_1 t \left( \frac{u}{t} - 1 \right)^2 > \frac{1}{150} (u - t)^2 = \frac{1}{150} \delta^2,
\end{align*}
where the penultimate inequality holds because $\frac{u}{t} - 1 = \frac{\delta}{t} < 1$ and $0 < c_2 \leq \frac{1}{2} c_1$.

(c)~The lower bound in~(b) can be improved to order~$\delta$ when $t = 0$.  Then for every $u \in [0, 1]$ we have 
$\P(J(0) \leq u) = e^{-\gamma} u$, and so for $u \in [0, 1 / 2]$ we have
\[
d_{\rm KS}(Z(0), Z(u)) \geq e^{- \gamma} u.  
\]
\end{remark}

\begin{ack}
We thank Svante Janson for helpful comments on a draft of this paper; they led to significant improvements.
\end{ack}

\appendix

\section{Proof of \refL{L:Lipschitz_2}}

In the proof of \refL{L:Lipschitz_2}, we show only that the function $c^{(3)}$ is Lipschitz continuous. In this Appendix, we show by similar arguments that $c^{(i)}$ is Lipschitz continuous for $i = 1, 2, 4, 5, 6$. Let $x, z \in \bbR$ with $z > x$ and consider $(l,r)$ satisfying $0 < l < t < r < 1$.
\begin{lemma}
\label{L:c_1_Lipschitz}
The function $c^{(1)}$ is Lipschitz continuous.
\end{lemma}
\begin{proof}
Redefine
\[
d_{l,r}(x, z, y) := f_{l,r}^{(1)}(z-y) - f_{l,r}^{(1)}(x-y)
\]
and recall that
\[
f_{l,r}^{(1)}(x) = \mathbb{1}(2-2l\leq x < 2-l) \frac{1}{1-l} \frac{1}{x-1+l}.
\]
We are interested in bounding the quantity
\begin{equation}
\label{E:c_1_Lip}
|c^{(1)}(z) - c^{(1)}(x)| \leq \int_{l, r, y}\!{|d_{l,r}(x, z, y)|\,\P(Y \in \ddx y \mid (L_3, R_3) = (l, r))}\dd l\dd r.
\end{equation}
Just as we did for \refL{L:Lipschitz_2}, we break the proof into consideration of two cases. 
\smallskip

\par\noindent
\emph{Case} 1.\ $z -x < l$. We bound $d_{l, r}(x, z, y)$ for~$y$ in each of the five subintervals of the real line determined by the four partition points  
\[
x-(2-l) < z -(2-l) < x - (2 - 2l) < z-(2-2l).
\]
For the two subcases $y \leq x-(2-l)$ and $y > z-(2-2l)$, we have $d_r(x, z, y) = 0$.  We bound the three nontrivial subcases (listed in order of convenience of exposition, not in natural order) as follows.%%%%%%%%%Start here 
\smallskip

\par\noindent
\emph{Subcase} 1(a).\ $z -(2-l) < y \leq x - (2 - 2l)$.
We have
\begin{align*}
d_{l,r}(x, z, y) &= \frac{1}{1-l} \left(\frac{1}{z-y-1+l} - \frac{1}{x-y-1+l}\right) \\
&= \frac{-1}{1-l} \frac{z-x}{(z-y-1+l)(x-y-1+l)}.
\end{align*}
Using the 
inequality $z-y > x-y \geq 2-2l$, we obtain
\[
|d_{l,r}(x, z, y) | \leq \frac{1}{1-l} \frac{|z-x|}{(1-l)^2}.
\]
Using inequality \eqref{E:c_1_Lip}, we conclude that the contribution from this subcase is bounded above by $\frac{t(2-t)}{2(1-t)} |z-x|$.
\smallskip

\par\noindent
\emph{Subcase} 1(b).\ $x -(2 - l) < y \leq z - (2 - l)$.
First note that in this subcase we have $f^{(1)}(z-y) = 0$. Using the fact that $y \leq z- (2-l) < x - (2- 2l)$, we have $x - y \geq 2 - 2l$ and
\[
| d_{l,r}(x, z, y) | \leq \frac{1}{1-l} \frac{1}{x - y - 1 + l} \leq \frac{1}{(1-l)^2}.
\]
By the constant bound on $f_t$ in \refT{T:bdd}, the contribution to $|c^{(1)}(z) - c^{(1)}(x)|$ is bounded above by
\begin{align*}
\int_{l = 0}^t \int_{r = t}^1 \int_{y = x-(2-l)}^{z-(2-l)} & \frac{1}{(1-l)^2}\frac{1}{r-l} f_{\frac{t-l}{r-l}}\left(\frac{y}{r-l}-1\right) \dd y \dd r \dd l\\
& \leq 10 (z-x) \int_{l = 0}^t \int_{r = t}^1\frac{1}{(1-l)^2}\frac{1}{r-l} \dd r \dd l\\
& = (z-x) \frac{10\, t \, (1 - \ln t)}{1-t}.
\end{align*}
\smallskip

\par\noindent
\emph{Subcase} 1(c).\ $x -(2-2l) < y \leq z - (2-2l)$.
In this case, the contribution from $f^{(1)}(x-y)$ vanishes. We have
\[
| d_{l,r}(x, z, y) | \leq \frac{1}{1-l} \frac{1}{z - y - 1 + l} \leq \frac{1}{(1-l)^2}.
\]
The corresponding contribution to $|c^{(1)}(z) - c^{(1)}(x)|$ can be bounded just as in Subcase 1(b).

This completes the proof for Case~1.
\smallskip

\par\noindent
\emph{Case}~2.\ $z -x \geq l$. In this case we 
simply use
\[
| d_{l,r}(x,z, y) | \leq f^{(1)}(z-y) + f^{(1)}(x-y),
\]
and show that each of the two terms on the right contributes at most a constant (depending on~$t$) times $(z - x)$ to the bound in~\eqref{E:c_1_Lip}.  Accordingly, let~$w$ be either~$x$ or~$z$.  We are interested in bounding 
\begin{equation}
\label{E:int_y_4}
\int_{l = 0}^{(z-x) \wedge t} \int_{r=t}^1 \int_{y = w-(2-l)}^{w-(2-2l)} {\frac{1}{1-l} \frac{1}{w-y-1+l} \mu(\ddx y, \ddx r, \ddx l)}
\end{equation}
with $\mu(\ddx y, \ddx r, \ddx l) := P(Y \in \ddx y \mid (L_3,R_3) =(l,r)) \dd r \dd l$. We bound the integrand as follows:
\[
\frac{1}{1-l} \frac{1}{w-y-1+l} \leq \frac{1}{(1-l)^2}.
\]
We first suppose $z-x < t$ and bound~\eqref{E:int_y_4} by
\[
\int_{l = 0}^{z-x} \int_{r=t}^1 {\frac{1}{(1-l)^2} \dd r \dd l} \leq \frac{1}{1-t} \int_{l = 0}^{z-x} {\dd l} = \frac{z-x}{1-t}.
\]
If instead $z - x \geq t$, we bound~\eqref{E:int_y_4} 
by
\[
\frac{z-x}{t} \int_{l = 0}^{t} \int_{r=t}^1 {\frac{1}{(1-l)^2} \dd r \dd l} = z - x.
\]
This completes the proof for Case~2 and thus the proof of Lipschitz continuity of $c^{(1)}$.
\end{proof}

\begin{lemma}
\label{L:c_2_Lipschitz}
The function $c^{(2)}$ is Lipschitz continuous.
\end{lemma}
\begin{proof}
The Lipschitz continuity of $c^{(2)}$ follows immediately from $c^{(1)}$ because
\[
f_{l,r}^{(2)}(x) = \mathbb{1}(2 r\leq x < 1+r) \frac{1}{r} \frac{1}{x-r},
\]
and we can replace $l$ in the proof of $c^{(1)}$ by $1- r$ to see that if $\Lambda^{(1)}(t)$ is a bound on the Lipschitz constant for~$c^{(1)}$, then $\Lambda^{(1)}(1 - t)$ is a bound on the Lipschitz constant for $c^{(2)}$.
\end{proof}

\begin{lemma}
\label{L:c_4_Lipschitz}
The function $c^{(4)}$ is Lipschitz continuous.
\end{lemma}
\begin{proof}
Redefine
\[
d_{l,r}(x, z, y) := f_{l,r}^{(4)}(z-y) - f_{l,r}^{(4)}(x-y)
\]
and reformulate
\[
f_{l,r}^{(4)}(x) = \mathbb{1}(2r-l\leq x < 2-l) \frac{1}{x}\left(\frac{1}{x+l} + \frac{1}{x-l}\right)
\]
from the expression for $f_{l,r}^{(4)}(x)$ found in \refS{S:density}.
We are interested in bounding the quantity
\begin{equation}
\label{E:c_4_Lip}
|c^{(4)}(z) - c^{(4)}(x)| \leq \int_{l, r, y}\!{|d_{l,r}(x, z, y)|\,\P(Y \in \ddx y \mid (L_3, R_3) = (l, r))}\dd l\dd r,
\end{equation}
where the conditional probability can also be written in density terms as
\[
\P(Y \in \ddx y \mid (L_3, R_3) = (l, r)) = \frac{1}{r-l} f_{\frac{t-l}{r-l}}\left(\frac{y}{r-l} - 1\right)\dd y.
\]
Just as we did for \refL{L:Lipschitz_2}, we break the proof into consideration of two cases. 
\smallskip

\par\noindent
\emph{Case} 1.\ $z -x < 2 (1-r)$. As in the proof for Case~1 of \refL{L:Lipschitz_2}, we bound $d_{l, r}(x, z, y)$ for~$y$ in each of the five subintervals of the real line determined by the four partition points  
\[
x-(2-l) < z -(2-l) < x - (2r - l) < z-(2r-l).
\]
For the two subcases $y \leq x-(2-l)$ and $y > z-(2r-l)$, we have $d_r(x, z, y) = 0$.  We bound the three nontrivial subcases (listed in order of convenience of exposition, not in natural order) 
as follows. 
\smallskip

\par\noindent
\emph{Subcase} 1(a).\ $z -(2-l) < y \leq x - (2r - l)$.
We have
\begin{align*}
&d_{l,r}(x, z, y) \\
&= \frac{1}{z-y} \left(\frac{1}{z-y+l} + \frac{1}{z-y-l}\right) - \frac{1}{x-y} \left(\frac{1}{x-y+l} + \frac{1}{x-y-l}\right)\\
&= \frac{1}{z-y} \left(\frac{1}{z-y+l} - \frac{1}{x-y+l}\right) + \left(\frac{1}{z-y} - \frac{1}{x-y}\right) \frac{1}{x-y+l}\\
&+ \frac{1}{z-y} \left(\frac{1}{z-y-l} - \frac{1}{x-y-l} \right) + \left(\frac{1}{z-y} - \frac{1}{x-y}\right) \frac{1}{x-y-l}.
\end{align*}
Using the 
inequality $z-y > x-y \geq 2r-l$, we obtain
\begin{align*}
& |d_{l,r}(x, z, y) | \\
&\leq \frac{1}{2r-l} \frac{z-x}{(2r)^2} + \frac{z-x}{(2r-l)^2}\frac{1}{2r}\\
&{} \qquad + \frac{1}{2r-l} \frac{z-x}{(z-y-l)(x-y-l)} + \frac{z-x}{(2r-l)^2} \frac{1}{2(r-l)}.
\end{align*}
Except for the third term, it is easy to see (by direct computation) that the corresponding contribution to the 
bound~\eqref{E:c_4_Lip} on $|c^{(4)}(z) - c^{(4)}(x)|$ is bounded by a constant (depending on~$t$) times $z - x$. So we now focus on bounding the contribution from the third term. Note that since $2r-l > t > 0$, we need 
only bound
\begin{equation}
\label{focus_4}
\int_{l, r, y}\!{\frac{1}{(z-y-l)(x-y-l)}\,\P(Y \in \ddx y \mid (L_3, R_3) = (l, r))}\dd l\dd r
\end{equation}
by a constant (which is allowed to depend on~$t$, but our constant will not). 

We first focus on the integral in~\eqref{focus_4} with respect to~$y$ and write it, using a change of variables, as 
\begin{equation}
\label{E:int_y_4a}
\int_{y \in I}\!{d_{l,r}^*(x, z, y)\,f_{\frac{t-l}{r-l}}(y) \dd y},
\end{equation}
with
\[
d_{l,r}^*(x, z, y) = \frac{1}{[z-(r-l)(y+1)-l][x-(r-l)(y+1)-l]}
\] 
and
 $I := \left\{y: \frac{z-(2-l)}{r-l} -1 < y \leq \frac{x-(2r-l)}{r-l} - 1\right\}$. 
Because the support of the density $f_{\frac{t - l}{r - l}}$ is contained in the nonnegative real line, the integral~\eqref{E:int_y_4a} vanishes unless the right endpoint of the interval~$I$ is positive, which is true if 
and only if
 \[
l > \frac{3r-x}{2}.
\]
So we see that the integral of~\eqref{E:int_y_4a} over $l \in (0, t)$ vanishes unless this lower bound on~$l$ is smaller than~$t$, which is true if 
and only if 
\begin{equation}
\label{rbound_4}
r < \frac{2t+x}{3}.
\end{equation}
But then the integral of~\eqref{E:int_y_4a} over $\{(l, r):0 < l < t < r < 
1\}$ vanishes unless this upper bound on~$r$ is greater than~$t$, which is true if and only if $x >  t$; we conclude that for $x \leq t$, that integral vanishes. 

So we may now suppose $x > t$, and we have seen that the integral of~\eqref{E:int_y_4a} over $\{(l, r):0 < l < t < r < 1\}$ is equal 
to its 
integral over the region
\[
R := \left\{ (l, r) : \frac{3 r - x}{2} \vee 0 < l < t < r < 1 \wedge \frac{2 t + x}{3} \right\}.
\]
Observe that on~$R$ we have
\begin{equation}
\label{on_R_4}
\frac{x-(2r-l)}{r-l} - 1 = \frac{x-l}{r-l} -3 > \frac{x-t}{r-l} -3 > \frac{1}{2} \frac{x-t}{r-l} -3.
\end{equation}
Define
\[
B := \left\{ (l,r) : \frac{x-t}{2(r-l)} - 3 > 0 \right\}.
\]
We now split our discussion of the contribution to the integral of~\eqref{E:int_y_4a} over $(l, r) \in R$ into two terms, corresponding to (i) $R \cap B^c$ and (ii) $R \cap B$. 
\smallskip

\par\noindent
\emph{Term} (i).\ $R \cap B^c$.
Since the support of the density $f_{\frac{t - l}{r - l}}$ is contained in the nonnegative real line, the value of~\eqref{E:int_y_4a} remains unchanged if we integrate $y$ over the positive part of $I$ instead.
Let
\[
I^* := \left\{y : \frac{x-t}{2(r-l)} - 3 < y \leq \frac{x-(2r-l)}{r-l} - 1\right\}.
\]
Since the left endpoint of $I^*$ is nonpositive, we can bound~\eqref{E:int_y_4a} by extending the range of integration from~$I$ to $I^*$.
Making use of the inequality~\eqref{E:exp_bdd}, the integral~\eqref{E:int_y_4a} is bounded, for any $\theta > 0$,  by
\[
\int_{y \in I^*} {\frac{1}{4(r-l)^2}\, C_{\theta} e^{-\theta y} \dd y} \leq \frac{C_{\theta}}{4 \theta(r-l)^2} \exp \left[-\frac{x-t}{2(r-l)} \theta + 3 \theta \right].
\]
The integral over $(l, r) \in R \cap B^c$ of~\eqref{E:int_y_4a} is therefore bounded by 
\begin{align}
&\frac{C_{\theta}}{4 \theta} \, e^{3 \theta} \int_{r = t}^{(2t+x)/3}\,\int_{l=(3r-x)/2}^{t} {\frac{1}{(r-l)^2} 
\exp \left[ -\frac{x-t}{2(r-l)} \theta \right]  \dd l  \dd r} \nonumber \\
&= \frac{C_{\theta}}{4 \theta} \, e^{3 \theta} \int_{r = t}^{(2t+x)/3} \int_{s=r-t}^{(x-r)/2} {\frac{1}{s^2} 
\exp \left(- \frac{x-t}{2} \theta s^{-1} \right) \dd s \dd r} \nonumber \\
&\leq \frac{C_{\theta}}{4 \theta} \, e^{3 \theta} \frac{2}{\theta (x-t)} \int_{r = t}^{(2t+x)/3} 
{\exp \left(- \frac{x-t}{2} \theta \frac{2}{x-r} \right) \dd r} \nonumber 
\\
\label{previous_4}
&\leq \frac{C_{\theta}}{2 \theta^2} \, e^{3 \theta} \frac{1}{x-t} e^{-\theta} \left(\frac{2t+x}{3} - t \right) 
= \frac{C_{\theta}}{6 \theta^2} \, e^{2 \theta} < \infty.
\end{align}
\smallskip

\par\noindent
\emph{Term} (ii).\ $R \cap B$.
We can bound~\eqref{E:int_y_4a} by the sum of the integrals of the same integrand over the intervals $I^*$ and 
\[
I' := \left\{y: 0 < y \leq \frac{x-t}{2(r-l)} - 3\right\}.
\]
The bound for the integral over $I^*$ is the same as the bound for the $R \cap B^c$ term. To bound the integral over $I'$, we first observe that
\[
d_{l,r}^*(x, z, y) \leq \frac{1}{[\frac{1}{2} (x+t) + 2r - 3l]^2} \leq \frac{4}{(x-t)^2},
\]
where the last inequality holds because $l < t < r$. The contribution to~\eqref{focus_4} can be bounded by integrating $4 / (x 
-t)^2$ with respect to $(l,r) \in R \cap B$.  We then extend this region of integration to~$R$, where $(3 r - x) / 2$ is at most the lower bound on~$l$ and $(2 t + x) / 3$ is at least the upper bound on~$r$. Thus we bound the contribution 
by
\[
\frac{4}{(x-t)^2} \int_{r = t}^{(2t+x)/3} {\left( t - \frac{3 r - x}{2} \right) \dd r}
= \frac{1}{3}.
\]
This completes the proof for Subcase 1(a).
\smallskip

\par\noindent
\emph{Subcase} 1(b).\ $x -(2r -l) < y \leq z - (2r-l)$.
First note that in this subcase we have $f^{(4)}(x-y) = 0$. We proceed in similar fashion as for Subcase 1(a), this time setting
\[
I := \left\{y: \frac{x-l}{r-l} -3 < y \leq \frac{z-l}{r-l} -3\right\}.
\]
Again using a linear change of variables and the original expression for $f^{(4)}$ defined in~\eqref{E:5}, the integral (with respect to~$y$ only, in this subcase) appearing on the right in~\eqref{E:c_4_Lip} in this subcase can be bounded by 
\begin{equation}
\label{E:int_y_4_2}
\int_{y \in I} {d^*_{l,r}(z, y) f_{\frac{t-l}{r-l}}(y) \dd y}
\end{equation}
where now
\[
d^*_{l,r}(z, y) = \frac{1}{z-(r-l)(y+1) +l} \times \frac{2}{z-(r-l)(y+1)-l}.
\]
Note that, unlike its analogue in Subcase 1(a), here $d^*_{l, r}(z, y)$ does not possess an explicit factor $z - x$. 

By the same discussion as in Subcase 1(a), 
we are interested in the integral of~\eqref{E:int_y_4_2} with respect to $(l,r) \in R$, where this time
\[
R := \left\{ (l, r) : \frac{3 r - z }{2} \vee 0 < l < t < r < 1 \wedge \frac{2 t + z }{3} \right\}.
\]
and we may suppose that $z > t$.

Observe that on~$R$ we have 
\[
\frac{z-l}{r-l} - 3 > \frac{z - t}{r-l} - 3 > \frac{1}{2} \frac{z -t}{r-l} - 3.
\]
Following a line of attack similar to that for Subcase 1(a), we define
\[
W := \left\{(l,r) : \frac{x-l}{r-l} - 3 > \frac{z - t}{2(r-l)} - 3\right\} 
\]
and split our discussion of the integral of~\eqref{E:int_y_4_2} over $(l, r) \in R$ into two terms, corresponding to (i) $R \cap W^c$ and (ii) $R \cap W$.
\smallskip

\par\noindent
\emph{Term} (i).\ $R \cap W$.
We bound~\eqref{E:int_y_4_2} by using the inequality~\eqref{E:exp_bdd} (for any $\theta > 0$) and obtain
\begin{align*}
\lefteqn{\hspace{-0.5in}\int_{y \in I} {\frac{1}{r} \frac{1}{2(r-l)} C_{\theta} 
\exp\left[ -\theta \left(\frac{z - t}{2(r-l)} - 3\right) \right] \dd y}} \\ 
&\leq 
\frac{1}{2} \frac{1}{t} \frac{1}{(r-l)^2} C_{\theta} e^{3\theta} \exp\left[ -\theta \left(\frac{z -t}{2(r-l)}\right) \right] (z-x).
\end{align*}
Integrating this bound with respect to $(l,r) \in R \cap W$, we get 
no more than
\[
(z-x) \frac{1}{2} \frac{1}{t} C_{\theta} e^{3\theta} \int_{r = t}^{(2t+z)/3} \int_{l = (3r-z)/2}^{t} {\frac{1}{(r-l)^2} 
\exp\left[ - \frac{z - t}{2(r-l)} \theta \right] \dd l \dd r},
\]
which [compare~\eqref{previous}] is bounded by $(z-x)$ times a constant depending only on~$t$ and~$\theta$.
\smallskip

\par\noindent
\emph{Term} (ii).\ $R \cap W^c$.
We partition the interval~$I$ of $y$-integration into the two subintervals 
\[
I^* := \left\{y : \frac{z - t}{2(r-l)} - 3 < y \leq \frac{z-l}{r-l} -3\right\}
\]
and
\[
I' := \left\{y: \frac{x-l}{r-l} -3 < y \leq \frac{z - t }{2(r-l)} - 3\right\}.
\]
Observe that 
the length of each of the intervals $I^*$ and $I'$ is no more than the length of~$I$, which is 
$(z - x) / (r - l)$. 
We can bound the integral over $y \in I^*$ and $(l,r) \in R \cap W^c$ just as we did for Term~(i). For the integral over 
$y \in I'$  and $(l,r) \in R \cap W^c$, observe the following 
inequality:
\begin{align*}
d^*_{l,r}(z, y)  \leq \frac{2}{\frac{1}{2} (z + t) + 2r-l}  \times \frac{1}{\frac{1}{2} (z + t) + 2r-3l} \leq \frac{4}{3 t} \frac{2}{z-t}.
\end{align*}
Using the constant bound in~\refT{T:bdd}, the integral of 
$d^*_{l, r}(z, y) f_{\frac{t - l}{r - l}}(y)$ 
with respect to $y \in I'$ and $(l,r) \in R \cap W^c$ is bounded above 
by
\[
80 \frac{z-x}{3t} \frac{1}{z-t} \int_{r = t}^{(2t + z)/3} \int_{l = 
(3r-z)/2}^{t} {\frac{1}{r-l} \dd l \dd r}
= \frac{80}{3} \frac{z-x}{t} (\ln 3 - \ln 2).
\]
This completes the proof for Subcase 1(b). 
\smallskip

\par\noindent
\emph{Subcase} 1(c).\ $x -(2-l) < y \leq z - (2-l)$.
In this case, the contribution from $f^{(4)}(z-y)$ vanishes. Without loss
of generality we may suppose $z - x < 1 - t$, because otherwise we
can insert a factor $(z-x)/(1-t)$ in our upper bound, and the desired
upper bound follows from the fact that the densities $f_{\tau}$ are all
bounded by~$10$.
Observe that the integrand $|d_{l, r}(x, z, y)|$ in the
bound~\eqref{E:c_4_Lip} is
\begin{align*}
\frac{2}{x - y + l} \times \frac{1}{x - y - l} &\leq \frac{1}{x - z + 2
} \times \frac{2}{x - z + 2 - 2 l} \\
&\leq \frac{1}{1+t} \times \frac{2}{1+t-2 l} \leq \frac{2}{(1+t)(1-t)}.
\end{align*}
Integrate this constant bound directly with respect to
\[
P(Y \in \ddx y | (L_3,R_3) = (l,r)) \dd r \dd l
\]
on the region $x -(2-l) < y \leq z - (2-l)$ and $0 < l < t < r < 1$ and
use the fact that the density is bounded by $10$; we conclude that this
contribution is bounded by $(z-x)$ times a constant that depends on $t$.
This completes the proof for Subcase 1(c) and also for Case~1.
\smallskip

\par\noindent
\emph{Case}~2.\ $z -x \geq 2 (1-r)$. In this case we 
simply use
\[
| d_{l,r}(x,z, y) | \leq f^{(4)}(z-y) + f^{(4)}(x-y),
\]
and show that each of the two terms on the right contributes at most a constant (depending on~$t$) times $(z - x)$ to the bound in~\eqref{E:c_4_Lip}.  Accordingly, let~$w$ be either~$x$ or~$z$.  We are interested in 
bounding 
\begin{equation}
\label{E:int_y_4_3}
\int_{r = \left[1 - \frac{1}{2} (z-x) \right] \vee t}^1 \int_{l = 0}^{t} \int_{y = w-(2-l)}^{w-(2r-l)} {\frac{2}{w-y+l} \frac{1}{w-y-l} \mu(\ddx y, \ddx l, \ddx r)}
\end{equation}
with $\mu(\ddx y, \ddx l, \ddx r) := P(Y \in \ddx y \mid (L_3,R_3) =(l,r)) \dd l \dd r$. We bound the integrand as follows:
\[
\frac{2}{w-y+l} \frac{1}{w-y-l} \leq \frac{1}{r} \frac{1}{2r - 2l} \leq \frac{1}{2} \frac{1}{t} \frac{1}{r-l}.
\]
We first suppose $z-x < 1-t$ and bound~\eqref{E:int_y_4_3} 
by
\begin{align*}
\frac{1}{2} \frac{1}{t} \int_{r = 1 - \frac{1}{2} (z-x)}^{1} \int_{l=0}^t {\frac{1}{r-l} \dd l \dd r} 
&\leq \frac{1}{2} \frac{1}{t} \int_{r = 1 - \frac{1}{2} (z-x)}^{1} {[- \ln (r-t)] \dd r}\\
&\leq \frac{1}{2} \frac{1}{t} \left[ - \ln \left(1 - \frac{z-x}{2} - t\right) \right] \frac{z-x}{2}\\
&\leq (z-x) \frac{\ln(2 / (1-t))}{4t}.
\end{align*}
If instead $z - x \geq 1-t$, we bound~\eqref{E:int_y_4_3} 
by
\begin{align*}
\frac{1}{2} \frac{z-x}{1-t} \int_{l = 0}^{t} \int_{r=t}^1 {\frac{1}{t}\frac{1}{r-l} \dd r \dd l}
&\leq \frac{1}{2} \frac{z-x}{t(1-t)} \int_{l = 0}^{t} {[-\ln (t-l)] \dd l}\\
&\leq (z-x) \frac{1 - \ln t}{2\,(1-t)}.
\end{align*}
This completes the proof for Case~2 and thus the proof of Lipschitz continuity of $c^{(4)}$.
\end{proof}

\begin{lemma}
\label{L:c_5_Lipschitz}
The function $c^{(5)}$ is Lipschitz continuous.
\end{lemma}
\begin{proof}
Redefine
\[
d_{l,r}(x, z, y) := f_{l,r}^{(5)}(z-y) - f_{l,r}^{(5)}(x-y)
\]
and recall
\[
f_{l,r}^{(5)}(x) = \mathbb{1}(2r-l\leq x < 2r) \frac{1}{r}\frac{1}{x-r}.
\]
We are interested in bounding the quantity
\begin{equation}
\label{E:c_5_Lip}
|c^{(5)}(z) - c^{(5)}(x)| \leq \int_{l, r, y}\!{|d_{l,r}(x, z, y)|\,\P(Y \in \ddx y \mid (L_3, R_3) = (l, r))}\dd l\dd r,
\end{equation}
where the conditional probability can also be written in density terms as
\[
\P(Y \in \ddx y \mid (L_3, R_3) = (l, r)) = \frac{1}{r-l} f_{\frac{t-l}{r-l}}\left(\frac{y}{r-l} - 1\right)\dd y.
\]
Just as we did for \refL{L:Lipschitz_2}, we break the proof into consideration of two cases. 
\smallskip

\par\noindent
\emph{Case} 1.\ $z -x < l$. As in the proof for Case~1 of 
\refL{L:Lipschitz_2}, we bound $d_{l, r}(x, z, y)$ for~$y$ in each of the 
five subintervals of the real line determined by the four partition points  
\[
x- 2 r < z - 2 r < x - (2 r - l) < z - (2 r - l).
\]
For the two subcases $y \leq x- 2r$ and $y > z-(2r-l)$, we have $d_r(x, z, y) = 0$.  We bound the three nontrivial subcases (listed in order of convenience of exposition, not in natural order) 
as follows. 
\smallskip

\par\noindent
\emph{Subcase} 1(a).\ $z - 2 r < y \leq x - (2 r - l)$.
We have
\begin{align*}
d_{l,r}(x, z, y) &= \frac{1}{r} \left(\frac{1}{z-y-r} - \frac{1}{x-y-r}\right)\\
&= - \frac{z-x}{r} \frac{1}{(z-y-r)(x-y-r)}.
\end{align*}
Note that since $r > t > 0$, we need only bound
\begin{equation}
\label{focus_5}
\int_{l, r, y}\!{\frac{1}{(z-y-r)(x-y-r)}\,\P(Y \in \ddx y \mid (L_3, R_3) = (l, r))}\dd l\dd r
\end{equation}
by a constant (which is allowed to depend on~$t$, but our constant will not).

We first focus on the integral in~\eqref{focus_5} with respect to~$y$ and write it, using a change of variables, as
\begin{equation}
\label{E:int_y_5_4}
\int_{y \in I}\!{d_{l,r}^*(x, z, y)\,f_{\frac{t-l}{r-l}}(y) \dd y},
\end{equation}
with
\[
d_{l,r}^*(x, z, y) = \frac{1}{[z-(r-l)(y+1)-r][x-(r-l)(y+1)-r]}
\] 
and
 $I := \left\{y: \frac{z-(2r)}{r-l} -1 < y \leq \frac{x-(2r-l)}{r-l} - 1\right\}$. 
Because the support of the density $f_{\frac{t - l}{r - l}}$ is contained 
in the nonnegative real line, the integral~\eqref{E:int_y_5_4} vanishes unless the right endpoint of the interval~$I$ is positive, which is true if and only if
 \[
l > \frac{3r-x}{2}.
\]
So we see that the integral 
of~\eqref{E:int_y_5_4} over $l \in (0, t)$ vanishes unless this lower bound on~$l$ is smaller than~$t$, which is true if 
and only if 
\begin{equation}
\label{rbound_5}
r < \frac{2t+x}{3}.
\end{equation}
But then the integral of~\eqref{E:int_y_5_4} over $\{(l, r):0 < l < t < r < 1\}$ vanishes unless this upper bound on~$r$ is greater than~$t$, which is true if and only if $x >  t$; we conclude that for $x \leq t$, that integral vanishes. 

So we may now suppose $x > t$, and we have seen that the integral of~\eqref{E:int_y_5_4} over $\{(l, r):0 < l < t < r < 1\}$ is bounded above by its integral over the region
\[
R := \left\{ (l, r) : \frac{3 r - x}{2} \vee 0 < l < t < r < 1 \wedge \frac{2 t + x}{3} \right\}.
\]
Observe that on~$R$ we have
\begin{equation}
\label{on_R_5}
\frac{x-(2r-l)}{r-l} - 1 = \frac{x-l}{r-l} -3 > \frac{x-t}{r-l} -3 > \frac{1}{2} \frac{x-t}{r-l} -3.
\end{equation}
Define
\[
B := \left\{ (l,r) : \frac{x-t}{2(r-l)} - 3 > 0 \right\}.
\]
We now split our discussion of the contribution to the integral of~\eqref{E:int_y_5_4} over $(l, r) \in R$ into two terms, corresponding to (i) $R \cap B^c$ and (ii) $R \cap B$. 
\smallskip

\par\noindent
\emph{Term} (i).\ $R \cap B^c$.
Since the support of the density $f_{\frac{t - l}{r - l}}$ is contained in the nonnegative real line, the value of~\eqref{E:int_y_5_4} remains unchanged if we integrate $y$ over the positive part of $I$ instead. 
Let
\[
I^* := \left\{y : \frac{x-t}{2(r-l)} - 3 < y \leq \frac{x-(2r-l)}{r-l} - 1\right\}.
\]
Since the left endpoint of $I^*$ is nonpositive, we can bound~\eqref{E:int_y_5_4} by extending the range of integration from~$I$ to $I^*$.
Making use of the inequality~\eqref{E:exp_bdd}, the integral~\eqref{E:int_y_5_4} is bounded, for any $\theta > 0$,  by
\[
\int_{y \in I^*} {\frac{1}{(r-l)^2}\, C_{\theta} e^{-\theta y} \dd y} \leq \frac{C_{\theta}}{ \theta(r-l)^2} \exp \left[-\frac{x-t}{2(r-l)} \theta + 3 \theta \right].
\]
The integral over $(l, r) \in R \cap B^c$ of~\eqref{E:int_y_5_4} is therefore bounded by
\begin{align}
&\frac{C_{\theta}}{ \theta} \, e^{3 \theta} \int_{r = t}^{(2t+x)/3}\,\int_{l=(3r-x)/2}^{t} {\frac{1}{(r-l)^2} 
\exp \left[ -\frac{x-t}{2(r-l)} \theta \right]  \dd l  \dd r} \nonumber \\
&= \frac{C_{\theta}}{ \theta} \, e^{3 \theta} \int_{r = t}^{(2t+x)/3} \int_{s=r-t}^{(x-r)/2} {\frac{1}{s^2} 
\exp \left(- \frac{x-t}{2} \theta s^{-1} \right) \dd s \dd r} \nonumber \\
&\leq \frac{C_{\theta}}{ \theta} \, e^{3 \theta} \frac{2}{\theta (x-t)} \int_{r = t}^{(2t+x)/3} 
{\exp \left(- \frac{x-t}{2} \theta \frac{2}{x-r} \right) \dd r} \nonumber 
\\
\label{previous_5}
&\leq \frac{C_{\theta}}{ \theta^2} \, e^{3 \theta} \frac{2}{x-t} e^{-\theta} \left(\frac{2t+x}{3} - t \right) 
= \frac{2C_{\theta}}{3 \theta^2} \, e^{2 \theta} < \infty.
\end{align}
\smallskip

\par\noindent
\emph{Term} (ii).\ $R \cap B$.
We can bound~\eqref{E:int_y_5_4} by the sum of the integrals of the same integrand over the intervals $I^*$ and 
\[
I' := \left\{y: 0 < y \leq \frac{x-t}{2(r-l)} - 3\right\}.
\]
The bound for the integral over $I^*$ is the same as the bound for the $R \cap B^c$ term. To bound the integral over $I'$, we first observe that
\[
d_{l,r}^*(x, z, y) \leq \frac{1}{[\frac{1}{2} (x+t) + r-2l]^2} \leq \frac{4}{(x-t)^2},
\]
where the last inequality holds because $l < t < r$.
The contribution to~\eqref{focus_5} can be bounded by integrating $4 / (x -t)^2$ with respect to $(l,r) \in R \cap B$. We then extend this region of integration to~$R$, where $(3 r - x) / 2$ is at most the lower bound on~$l$ and $(2 t + x) / 3$ is at least the upper bound on~$r$. Thus we bound the contribution 
by
\[
\frac{4}{(x-t)^2} \int_{r = t}^{(2t+x)/3} {\left( t - \frac{3 r - x}{2} \right) \dd r}  = \frac{1}{3}.
\]
This completes the proof for Subcase 1(a).
\smallskip

\par\noindent
\emph{Subcase} 1(b).\ $x -(2r -l) < y \leq z - (2r-l)$.
First note that in this subcase we have $f^{(5)}(x-y) = 0$. We proceed in similar fashion as for Subcase 1(a), this time setting
\[
I := \left\{y: \frac{x-l}{r-l} -3 < y \leq \frac{z-l}{r-l} -3\right\}.
\]
Again using a linear change of variables, the integral (with respect to~$y$ only, in this subcase) appearing on the right 
in~\eqref{E:c_5_Lip} in this subcase can be bounded by
\begin{equation}
\label{E:int_y_5_2}
\int_{y \in I} {d^*_{l,r}(z, y) f_{\frac{t-l}{r-l}}(y) \dd y}
\end{equation}
where now
\[
d^*_{l,r}(z, y) = \frac{1}{r} \times \frac{1}{z-(r-l)(y+1)-r}.
\]
Note that, unlike its analogue in Subcase 1(a), here $d^*_{l, r}(z, y)$ does not possess an explicit factor $z - x$. 

By the same discussion as in Subcase 1(a), 
we are interested in the integral of~\eqref{E:int_y_5_2} with respect to $(l,r) \in R$, where this time
\[
R := \left\{ (l, r) : \frac{3 r - z }{2} \vee 0 < l < t < r < 1 \wedge \frac{2 t + z }{3} \right\}.
\]
and we may suppose that $z > t$.

Observe that on~$R$ we have 
\[
\frac{z-l}{r-l} - 3 > \frac{z - t}{r-l} - 3 > \frac{1}{2} \frac{z -t}{r-l} - 3.
\]
Following a line of attack similar to that for Subcase 1(a), we define
\[
W := \left\{(l,r) : \frac{x-l}{r-l} - 3 > \frac{z - t}{2(r-l)} - 3\right\} 
\]
and split our discussion of the integral of~\eqref{E:int_y_4_2} over $(l, r) \in R$ into two terms, corresponding to (i) $R \cap W^c$ and (ii) $R \cap W$.
\smallskip

\par\noindent
\emph{Term} (i).\ $R \cap W$.
We bound~\eqref{E:int_y_5_2} by using the inequality~\eqref{E:exp_bdd} (for any $\theta > 0$) and obtain
\begin{align*}
\lefteqn{\hspace{-0.5in}\int_{y \in I} {\frac{1}{r} \frac{1}{(r-l)} C_{\theta} 
\exp\left[ -\theta \left(\frac{z - t}{2(r-l)} - 3\right) \right] \dd y}} \\ 
&\leq 
\frac{1}{t} \frac{1}{(r-l)^2} C_{\theta} e^{3\theta} \exp\left[ -\theta \left(\frac{z -t}{2(r-l)}\right) \right] (z-x).
\end{align*}
Integrating this bound with respect to $(l,r) \in R \cap W$, we get no more than
\[
(z-x) \frac{1}{t} C_{\theta} e^{3\theta} \int_{r = t}^{(2t+z)/3} \int_{l = (3r-z)/2}^{t} {\frac{1}{(r-l)^2} 
\exp\left[ - \frac{z - t}{2(r-l)} \theta \right] \dd l \dd r},
\]
which [compare~\eqref{previous}] is bounded by $(z-x)$ times a constant depending only on~$t$ and~$\theta$.
\smallskip

\par\noindent
\emph{Term} (ii).\ $R \cap W^c$.
We partition the interval~$I$ of $y$-integration into the two subintervals 
\[
I^* := \left\{y : \frac{z - t}{2(r-l)} - 3 < y \leq \frac{z-l}{r-l} -3\right\}
\]
and
\[
I' := \left\{y: \frac{x-l}{r-l} -3 < y \leq \frac{z - t }{2(r-l)} - 3\right\}.
\]
Observe that  
the length of each of the intervals $I^*$ and $I'$ is no more than the length of~$I$, which is $(z - x) / (r - l)$. 
We can bound the integral over $y \in I^*$ and $(l,r) \in R \cap W^c$ just as we did for Term~(i). For the integral over 
$y \in I'$  and $(l,r) \in R \cap W^c$, observe the following 
inequality:
\begin{align*}
d^*_{l,r}(z, y)  \leq \frac{1}{r}  \frac{1}{\frac{1}{2} (z + t) + r-2l} \leq \frac{1}{t} \frac{2}{z-t}.
\end{align*}
Using the constant bound in~\refT{T:bdd}, the integral of $d^*_{l, r}(z, y) f_{\frac{t - l}{r - l}}(y)$ with respect to 
$y \in I'$ and $(l,r) \in R \cap W^c$ is bounded above 
by
\begin{align*}
&20 \frac{z-x}{t} \frac{1}{z-t} \int_{r = t}^{(2t + z)/3} \int_{l = 
(3r-z)/2}^{t} {\frac{1}{r-l} \dd l \dd r}\\
&= 20 \frac{z-x}{t} \frac{1}{z-t} \int_{r = t}^{(2t + z)/3} {\left[ - 
\ln (r-t) + \ln \left(\frac{1}{2} (z-r) \right) \right] \dd r}\\
&= 20 \frac{z-x}{t} \frac{1}{z-t} \frac{z-t}{3} \left[1 - \ln \left(\frac{z-t}{3}\right) + \ln (z-t) - \left(1 + \ln \left(\frac{8}{9}\right)\right)\right]\\
&= 20 \frac{z-x}{t} (\ln 3 - \ln 2).
\end{align*}
This completes the proof for Subcase 1(b). 
\smallskip

\par\noindent
\emph{Subcase} 1(c). \ $x - 2r < y \leq z - 2r$.
In this case, the contribution from $f^{(5)}(z-y)$ vanishes. Without loss of generality we may suppose $z - x < t/2$, otherwise we can insert a 
factor $2 (z - x) / t$ in our upper bound, and the desired upper bound follows 
from the fact that the densities $f_{\tau}$ are all bounded by~$10$. Observe that the integrand $|d_{l, r}(x, z, y)|$ in the bound~\eqref{E:c_5_Lip} is
\[
\frac{1}{r} \times \frac{1}{x - y - r} 
\leq \frac{1}{r } \times \frac{1}{x - z + r}
\leq \frac{1}{t} \times \frac{1}{r - (t / 2)} 
\leq \frac{2}{t^2}.
\]
Integrate this constant bound directly with respect to 
\[
P(Y \in \ddx y | (L_3,R_3) = (l,r)) \dd r \dd l
\]
on the region $x - 2r < y \leq z - 2r$ and $0 < l < t < r < 1$ and use the fact that the density is bounded by $10$; we conclude that this contribution is bounded by $(z-x)$ times a constant that depends on $t$. This completes the proof for Subcase 1(c) and also for Case~1.
\smallskip

\par\noindent
\emph{Case}~2.\ $z -x \geq l$. In this case we 
simply use
\[
| d_{l,r}(x,z, y) | \leq f^{(5)}(z-y) + f^{(5)}(x-y),
\]
and show that each of the two terms on the right contributes at most a constant (depending on~$t$) times $(z - x)$ to the bound in~\eqref{E:c_5_Lip}.  Accordingly, let~$w$ be either~$x$ or~$z$.  We are interested in 
bounding 
\begin{equation}
\label{E:int_y_5_3}
\int_{l = 0}^{(z-x)\wedge t} \int_{r= t}^1 
\int_{y = w - 2 r}^{w-(2r-l)} {\frac{1}{r} \frac{1}{w-y-r} \mu(\ddx y, \ddx r, \ddx l)}
\end{equation}
with $\mu(\ddx y, \ddx r, \ddx l) := P(Y \in \ddx y \mid (L_3,R_3) =(l,r)) \dd r \dd l$. We bound the integrand as follows:
\[
\frac{1}{r} \frac{1}{w-y-r} \leq \frac{1}{r} \frac{1}{r - l} \leq  \frac{1}{t} \frac{1}{r-l}.
\]
We first suppose $z-x < t/2$ and bound~\eqref{E:int_y_5_3} by
\begin{align*}
\frac{1}{t} \int_{l = 0}^{z-x} \int_{r= t}^1 {\frac{1}{r-l} \dd r \dd l} 
&\leq \frac{1}{t} \int_{l = 0}^{z-x} {[- \ln (t - l)] \dd l}\\
&\leq \frac{1}{t} \left[ - \ln \left( t - (z-x)\right) \right] (z-x)\\
&\leq (z-x) \frac{\ln(2 / t)}{t}.
\end{align*}
If instead $z - x \geq t/2$, we bound~\eqref{E:int_y_5_3} 
by
\begin{align*}
\frac{z-x}{t/2} \int_{l = 0}^{t} \int_{r=t}^1 {\frac{1}{t}\frac{1}{r-l} \dd r \dd l} 
&\leq 2 \frac{z-x}{t^2} \int_{l = 0}^{t} {[-\ln (t-l)] \dd l}\\
&\leq 2 (z-x) \frac{1 - \ln t}{t}.
\end{align*}
This completes the proof for Case~2 and thus the proof of Lipschitz continuity of $c^{(5)}$.
\end{proof}

\begin{lemma}
\label{L:c_6_Lipschitz}
The function $c^{(6)}$ is Lipschitz continuous.
\end{lemma}
\begin{proof} 
Redefine
\[
d_{l,r}(x, z, y) := f_{l,r}^{(6)}(z-y) - f_{l,r}^{(6)}(x-y)
\]
and recall
\[
f_{l,r}^{(6)}(x) = \mathbb{1}(1+r-2l\leq x < 2-2l) \frac{1}{1-l} \frac{1}{x-1+l}.
\]
We are interested in bounding the quantity
\begin{equation}
\label{E:c_6_Lip}
|c^{(6)}(z) - c^{(6)}(x)| \leq \int_{l, r, y}\!{|d_{l,r}(x, z, y)|\,\P(Y \in \ddx y \mid (L_3, R_3) = (l, r))}\dd l\dd r,
\end{equation}
where the conditional probability can also be written in density terms as
\[
\P(Y \in \ddx y \mid (L_3, R_3) = (l, r)) = \frac{1}{r-l} f_{\frac{t-l}{r-l}}\left(\frac{y}{r-l} - 1\right)\dd y.
\]
Just as we did for \refL{L:Lipschitz_1}, we break the proof into consideration of two cases. 
\smallskip

\par\noindent
\emph{Case} 1.\ $z -x < 1-r$. As in the proof for Case~1 of \refL{L:Lipschitz_2}, we bound $d_{l, r}(x, z, y)$ for~$y$ in each of the five subintervals of the real line determined by the four partition points  
\[
x-(2-2l) < z -(2-2l) < x - (1+r - 2l) < z-(1+r-2l).
\]
For the two subcases $y \leq x-(2-2l)$ and $y > z-(1+r-2l)$, we have $d_r(x, z, y) = 0$.  We bound the three nontrivial subcases (listed in order of convenience of exposition, not in natural order) as follows. 
\smallskip

\par\noindent
\emph{Subcase} 1(a).\ $z -(2-2l) < y \leq x - (1+r - 2l)$.
We have
\begin{align*}
d_{l,r}(x, z, y) &= \frac{1}{1-l} \left(\frac{1}{z-y-1+l} - \frac{1}{x-y-1+l}\right) \\
&= - \frac{z-x}{1-l} \frac{1}{(z-y-1+l)(x-y-1+l)}.
\end{align*}
Note that since $1- l > 1-t > 0$, we need only bound
\begin{equation}
\label{focus_6}
\int_{l, r, y}\!{\frac{1}{(z-y-1+l)(x-y-1+l)}\,\P(Y \in \ddx y \mid (L_3, R_3) = (l, r))}\dd l\dd r
\end{equation}
by a constant (which is allowed to depend on~$t$, but our constant will not).

We first focus on the integral in~\eqref{focus_6} with respect to~$y$ and write it, using a change of variables, as
\begin{equation}
\label{E:int_y_6}
\int_{y \in I}\!{d_{l,r}^*(x, z, y)\,f_{\frac{t-l}{r-l}}(y) \dd y},
\end{equation}
with
\[
d_{l,r}^*(x, z, y) = \frac{1}{[z-(r-l)(y+1)-1+l][x-(r-l)(y+1)-1+l]}
\] 
and
 $I := \left\{y: \frac{z-(2-2l)}{r-l} -1 < y \leq \frac{x-(1+r-2l)}{r-l} - 1\right\}$. 
Because the support of the density $f_{\frac{t - l}{r - l}}$ is contained in the nonnegative real line, the integral~\eqref{E:int_y_6} vanishes unless the right endpoint of the interval~$I$ is positive, which is true if and only if
 \[
r < \frac{x-1+3l}{2}.
\]
So we see that the integral of~\eqref{E:int_y_6} over $r \in (t, 1)$ vanishes unless this upper bound on~$r$ is larger than~$t$, which is true if and only if 
\begin{equation}
\label{lbound_6}
l > \frac{1-x+2t}{3}.
\end{equation}
But then the integral of~\eqref{E:int_y_6} over $\{(l, r):0 < l < t < r < 1\}$ vanishes unless this lower bound on~$l$ is smaller than~$t$, which is true if and only if $x > 1 - t$; we conclude that for $x \leq 1-t$, that integral vanishes. 

So we may now suppose $x > 1 - t$, and we have seen that the integral of~\eqref{E:int_y_6} over $\{(l, r):0 < l < t < r < 1\}$ is bounded above by its integral over the region
\[
R := \left\{ (l, r) : \frac{1 - x + 2 t}{3} \vee 0 < l < t < r < 1 \wedge \frac{x - 1 + 3 l}{2} \right\}.
\]
Observe that on~$R$ we have
\begin{equation}
\label{on_R_6}
\frac{x-(1+r-2l)}{r-l} - 1 = \frac{x-1+l}{r-l} -2 > \frac{2}{3} \frac{x+t-1}{r-l} -2 > \frac{1}{2} \frac{x+t-1}{r-l} -2.
\end{equation}
Define
\[
B := \left\{ (l,r) : \frac{x+t-1}{2(r-l)} - 2 > 0 \right\}.
\]
We now split our discussion of the contribution to the integral of~\eqref{E:int_y_6} over $(l, r) \in R$ into two terms, corresponding to (i) $R \cap B^c$ and (ii) $R \cap B$. 
\smallskip

\par\noindent
\emph{Term} (i).\ $R \cap B^c$.
Since the support of the density $f_{\frac{t - l}{r - l}}$ is contained in the nonnegative real line, the value of~\eqref{E:int_y_6} remains unchanged if we integrate $y$ over the positive part of $I$ instead.
Let
\[
I^* := \left\{y : \frac{x+t-1}{2(r-l)} - 2 < y \leq \frac{x-(1+r-2l)}{r-l} - 1\right\}.
\]
Since the left endpoint of $I^*$ is nonpositive, we can bound~\eqref{E:int_y_6} by extending the range of integration from~$I$ to $I^*$.
Making use of the inequality~\eqref{E:exp_bdd}, the integral~\eqref{E:int_y_6} is bounded, for any $\theta > 0$,  by
\[
\int_{y \in I^*} {\frac{1}{(r-l)^2}\, C_{\theta} e^{-\theta y} \dd y} \leq \frac{C_{\theta}}{\theta(r-l)^2} \exp \left[-\frac{x+t-1}{2(r-l)} \theta + 2 \theta \right].
\]
The integral over $(l, r) \in R \cap B^c$ of~\eqref{E:int_y_6} is therefore bounded by
\begin{align}
&\frac{C_{\theta}}{ \theta} \, e^{2 \theta} \int_{l = (1-x+2t)/3}^{t}\,\int_{r=t}^{(x-1+3l)/2} {\frac{1}{(r-l)^2} 
\exp \left[ -\frac{x+t-1}{2(r-l)} \theta \right]  \dd r  \dd l} \nonumber 
\\
&= \frac{C_{\theta}}{ \theta} \, e^{2 \theta} \int_{l = (1-x+2t)/3}^{t} \int_{s=t-l}^{(x-1+l)/2} {\frac{1}{s^2} 
\exp \left(- \frac{x+t-1}{2} \theta s^{-1} \right) \dd s \dd l} \nonumber 
\\
&\leq \frac{C_{\theta}}{ \theta} \, e^{2 \theta} \frac{2}{\theta (x+t-1)} \int_{l = (1-x+2t)/3}^{t} 
{\exp \left(- \frac{x+t-1}{2} \theta \frac{2}{x-1+l} \right) \dd l} \nonumber \\
\label{previous_6}
&\leq \frac{2 C_{\theta}}{ \theta^2} \, e^{2 \theta} \frac{1}{x+t-1} e^{-\theta} \left(t  - \frac{1-x+2t}{3} \right) 
= \frac{2 C_{\theta}}{3 \theta^2} \, e^{\theta} < \infty.
\end{align}
\smallskip

\par\noindent
\emph{Term} (ii).\ $R \cap B$.
We can bound~\eqref{E:int_y_6} by the sum of the integrals of the same integrand over the intervals $I^*$ and 
\[
I' := \left\{y: 0 < y \leq \frac{x+t-1}{2(r-l)} - 2\right\}.
\]
The bound for the integral over $I^*$ is the same as the bound for the $R \cap B^c$ term. To bound the integral over $I'$, we first observe that
\[
d_{l,r}^*(x, z, y) \leq \frac{1}{[\frac{1}{2} (x-t-1) + r]^2} \leq \frac{4}{(x+t-1)^2},
\]
where the last inequality holds because $r > t$.
The contribution to~\eqref{focus_6} can be bounded by integrating $4 / (x + t - 1)^2$ with respect to $(l,r) \in R \cap B$. We then extend this region of integration to~$R$, where $(2t+1-x) / 3$ is at most the lower bound on~$l$ and $(x - 1 + 3l) / 2$ is at least the upper bound on~$r$. Thus we bound the contribution
by
\[
\frac{4}{(x+t-1)^2} \int_{l = \frac{2t+1-x}{3}}^t {\left(\frac{x-1+3l}{2} - t\right) \dd l} = \frac{1}{3}. 
\]
This completes the proof for Subcase 1(a).
\smallskip

\par\noindent
\emph{Subcase} 1(b).\ $x -(1+r -2l) < y \leq z - (1+r-2l)$.
First note that in this subcase we have $f^{(6)}(x-y) = 0$. We proceed in similar fashion as for Subcase 1(a), this time setting
\[
I := \left\{y: \frac{x-1+l}{r-l} -2 < y \leq \frac{z-1+l}{r-l} -2\right\}.
\]
Again using a linear change of variables, the integral (with respect to~$y$ only, in this subcase) appearing on the right in~\eqref{E:c_6_Lip} in this subcase can be written as 
\begin{equation}
\label{E:int_y_6_2}
\int_{y \in I} {d^*_{l,r}(z, y) f_{\frac{t-l}{r-l}}(y) \dd y}
\end{equation}
where now
\[
d^*_{l,r}(z, y) = \frac{1}{1-l} \times \frac{1}{z-(r-l)(y+1)-1+l}.
\]
Note that, unlike its analogue in Subcase 1(a), here $d^*_{l, r}(z, y)$ does not possess an explicit factor $z - x$. 

By the same discussion as in Subcase 1(a), 
we are interested in the integral of~\eqref{E:int_y_6_2} with respect to $(l,r) \in R$, where this time
\[
R := \left\{ (l, r) : \frac{1 - z + 2 t}{3} \vee 0 < l < t < r < 1 \wedge \frac{z - 1 + 3 l}{2} \right\}.
\]
and we may suppose that $z > 1 - t$.

Observe that on~$R$ we have 
\[
\frac{z-1+l}{r-l} - 2 > \frac{2}{3} \frac{z + t - 1}{r-l} - 2 > \frac{1}{2} \frac{z + t - 1}{r-l} - 2.
\]
Following a line of attack similar to that for Subcase 1(a), we define
\[
W := \left\{(l,r) : \frac{x-1+l}{r-l} - 2 > \frac{z + t - 1}{2(r-l)} - 2\right\}
\]
and split our discussion of the integral of~\eqref{E:int_y_6_2} over $(l, r) \in R$ into two terms, corresponding to (i) $R \cap W^c$ and (ii) $R \cap W$.
\smallskip

\par\noindent
\emph{Term} (i).\ $R \cap W$.
We bound~\eqref{E:int_y_6_2} by using the inequality~\eqref{E:exp_bdd} (for any $\theta > 0$) and obtain
\begin{align*}
\lefteqn{\hspace{-0.5in}\int_{y \in I} {\frac{1}{1-l} \frac{1}{r-l} C_{\theta} 
\exp\left[ -\theta \left(\frac{z + t - 1}{2(r-l)} - 2\right) \right] \dd y}} \\ 
&\leq 
\frac{1}{1-t} \frac{1}{(r-l)^2} C_{\theta} e^{2\theta} \exp\left[ -\theta \left(\frac{z + t - 1}{2(r-l)} \right) \right] (z-x).
\end{align*}
Integrating this bound with respect to $(l,r) \in R \cap W$, we get no more than
\[
(z-x) \frac{1}{1-t} C_{\theta} e^{2\theta} \int_{l = (1-z+2t)/3}^t \int_{r = t}^{(z-1+3l)/2} {\frac{1}{(r-l)^2} 
\exp\left[ - \frac{z + t - 1}{2(r-l)} \theta \right] \dd x},
\]
which [see~\eqref{previous}] is bounded by $(z-x)$ times a constant depending only on~$t$ and~$\theta$.
\smallskip

\par\noindent
\emph{Term} (ii).\ $R \cap W^c$.
We partition the interval~$I$ of $y$-integration into the two subintervals 
\[
I^* := \left\{y : \frac{z + t - 1}{2(r-l)} - 2 < y \leq \frac{z-1+l}{r-l} -2\right\}
\]
and
\[
I' := \left\{y: \frac{x-1+l}{r-l} -2 < y \leq \frac{z + t - 1}{2(r-l)} - 2\right\}.
\]
Observe that
the length of each of the intervals $I^*$ and $I'$ is no more than the length of~$I$, which is $(z - x) / (r - l)$. 
We can bound the integral over $y \in I^*$ and $(l,r) \in R \cap W^c$ just as we did for Term~(i). For the integral over 
$y \in I'$  and $(l,r) \in R \cap W^c$, observe the following
inequality:
\begin{align*}
d^*_{l,r}(z, y)  \leq \frac{1}{1-t} \times \frac{1}{\frac{1}{2} (z - t - 1) + r}.
\end{align*}
Using the constant bound in~\refT{T:bdd}, the integral of $d^*_{l, r}(z, y) f_{\frac{t - l}{r - l}}(y)$ with respect to 
$y \in I'$ and $(l,r) \in R \cap W^c$ is bounded above by
\begin{equation}
\label{10eq_6}
10 \frac{z - x}{1-t} \int_{l = (1-z+2t)/3}^t \int_{r = t}^{(z-1+3l)/2} {\frac{1}{r-l} \frac{1}{\frac{1}{2} (z - t - 1) + r} \dd r \dd l}.
\end{equation}
Write the integrand here in the form
\[
\frac{1}{r-l} \frac{1}{\frac{z - t - 1}{2} + r} = \left(\frac{1}{r-l} - \frac{1}{r + \frac{z - t - 1}{2} }\right) \frac{1}{l+\frac{z - t - 1}{2}},
\]
and observe that $l+\frac{z - t - 1}{2} > \frac{z + t - 1}{6} > 0$.  Hence we can bound~\eqref{10eq_6} by
\begin{align*}
&10 \frac{z - x}{1-t} \int_{l = (1-z+2t)/3}^t {\frac{1}{l+\frac{z - t - 1}{2}} \left[ \ln \frac{z-1+l}{2} - \ln (t-l) \right] \dd l}\\
&\leq 
10 \frac{z - x}{1-t} \frac{6}{z + t - 1} \left[ \frac{z + t - 1}{3} \ln \frac{z + t - 1}{2} - \int_{l = \frac{1-z+2t}{3}}^t {\ln (t-l) \dd l }\right]\\
&= 20 \frac{z - x}{1-t} \left[ \ln \frac{z + t - 1}{2} - \ln \left(\frac{z + t - 1}{3}\right) + 1 \right] \\
&= 20 \left(1 + \ln 3 - \ln 2 \right) \frac{(z-x)}{1-t}.
\end{align*}
This completes the proof for Subcase 1(b). 
\smallskip

\par\noindent
\emph{Subcase} 1(c).\ $x -(2-2l) < y \leq z - (2-2l)$.
In this case, the contribution from $f^{(6)}(z-y)$ vanishes. Without loss of generality we may suppose $z - x < (1-t)/2$, otherwise we can insert a factor $2 (z-x)/(1-t)$ in our upper bound, and the desired upper bound follows from the fact that the densities $f_{\tau}$ are all bounded by~$10$. Observe that the integrand $|d_{l, r}(x, z, y)|$ in the bound~\eqref{E:c_6_Lip} is
\[
\frac{1}{1-l} \frac{1}{x - y - 1 + l} \leq \frac{1}{1-l} \frac{1}{x - z + 1 - l} \leq \frac{1}{1-t} \frac{2}{1+t-2l} 
\leq \frac{2}{(1-t)^2}.
\]
Integrate this constant bound directly with respect to 
\[
P(Y \in \ddx y | (L_3,R_3) = (l,r)) \dd r \dd l
\]
on the region $x -(2-2l) < y \leq z - (2-2l)$ and $0 < l < t < r < 1$ and use the fact that the density is bounded by $10$; we conclude that this contribution is bounded by $(z-x)$ times a constant that depends on $t$. This completes the proof for Subcase 1(c) and also for Case~1.
\smallskip

\par\noindent
\emph{Case}~2.\ $z -x \geq 1-r$. In this case we 
simply use
\[
| d_{l,r}(x,z, y) | \leq f^{(6)}(z-y) + f^{(6)}(x-y),
\]
and show that each of the two terms on the right contributes at most a constant (depending on~$t$) times $(z - x)$ to the bound in~\eqref{E:c_6_Lip}.  Accordingly, let~$w$ be either~$x$ or~$z$.  We are interested in bounding 
\begin{equation}
\label{E:int_y_6_3}
\int_{r = [1-(z-x)]\vee t}^{1} \int_{l=0}^t \int_{y = w-(2-2l)}^{w-(1+r-2l)} {\frac{1}{1-l} \frac{2}{w-y-1+l} \mu(\ddx y, \ddx l, \ddx r)}
\end{equation}
with $\mu(\ddx y, \ddx l, \ddx r) := P(Y \in \ddx y \mid (L_3,R_3) =(l,r)) \dd l \dd r$. We bound the integrand as 
follows:
\[
\frac{1}{1-l} \frac{2}{w-y-1+l} \leq \frac{1}{1-l} \frac{2}{r - l} \leq \frac{1}{1-t} \frac{2}{r-l}.
\]
We first suppose $z-x < (1-t)/2$ and bound~\eqref{E:int_y_6_3} by
\begin{align*}
\frac{2}{1-t} \int_{r = 1-(z-x)}^{1} \int_{l=0}^t {\frac{1}{r-l} \dd l \dd r} 
&\leq \frac{2}{1-t} \int_{r = 1-(z-x)}^{1} {[- \ln (r-t)] \dd r}\\
&\leq \frac{2}{1-t} \left[ - \ln \left(1 - (z-x) - t\right) \right] (z-x)\\
&\leq 2 (z-x) \frac{\ln[2 / (1-t)]}{(1-t)}.
\end{align*}
If instead $z - x \geq (1-t)/2$, we bound~\eqref{E:int_y_6_3} by
\[
4 \frac{z-x}{1 - t} \int_{l = 0}^{t} \int_{r=t}^1 {\frac{1}{1-l}\frac{1}{r-l} \dd r \dd l} \leq (z-x) \frac{2 \pi^2}{3\,(1-t)}.
\]
This completes the proof for Case~2 and thus the proof of Lipschitz continuity of $c^{(6)}$.
\end{proof}

Lemmas \ref{L:c_1_Lipschitz}--\ref{L:c_6_Lipschitz}, together with Lemmas \ref{L:Lipschitz_1}--\ref{L:Lipschitz_2}, complete the proof of ~\refT{T:Lipschitz_cont}.

\bibliography{bib_file}
\bibliographystyle{plain}

\end{document}